%% file: main.tex
\documentclass[reqno]{amsart}[standalone]
\usepackage[T1]{fontenc}
\usepackage[ruled,linesnumbered]{algorithm2e}
\usepackage{amsfonts}
\usepackage{amsmath}
\usepackage{amssymb}
\usepackage{amsthm}
\usepackage[toc,page]{appendix}
\usepackage{bm}
\usepackage{booktabs}
\usepackage{pifont}
\usepackage{url}
\usepackage{comment}
\usepackage{multirow, multicol}
\usepackage{glossaries}
\usepackage{graphicx}
\usepackage{mathtools}
\usepackage[inline]{enumitem}
\usepackage{cases}
\usepackage{hyperref}
\usepackage{rotating}
\usepackage{subcaption}
\usepackage[table]{xcolor}
\usepackage{float}
\usepackage{siunitx}
\usepackage[top=1in, bottom=1.25in, left=1.25in, right=1.25in]{geometry}

\usepackage{tikz,xcolor}

\usetikzlibrary{patterns, arrows.meta, positioning}
\definecolor{mathlabblue}{HTML}{065895}
\definecolor{mathlaborange}{HTML}{F79A25}

\glsdisablehyper

\definecolor{lime}{HTML}{A6CE39}
\DeclareRobustCommand{\orcidicon}{%
	\begin{tikzpicture}
	\draw[lime, fill=lime] (0,0) 
	circle [radius=0.16] 
	node[white] {{\fontfamily{qag}\selectfont \tiny ID}};
	\draw[white, fill=white] (-0.075,0.095) 
	circle [radius=0.007];
	\end{tikzpicture}
	\hspace{-2mm}
}
\foreach \x in {A, ..., Z}{%
	\expandafter\xdef\csname orcid\x\endcsname{\noexpand\href{https://orcid.org/\csname orcidauthor\x\endcsname}{\noexpand\orcidicon}}
}

\usepackage{todonotes} % remove later
\usepackage{mathrsfs}
\usepackage[foot]{amsaddr}
\usepackage{caption}
\usepackage{mathbbol}
\usepackage[giveninits=true, doi=false, url=false, eprint=false]{biblatex}
\AtEveryBibitem{
    \clearfield{urlyear}
    \clearfield{urlmonth}
    \clearfield{month}
    \clearfield{issn}
    \clearfield{language}
    \clearfield{isbn}
     \iffieldundef{doi}
    {}
    {\clearfield{url}}%
}

\newbibmacro{title+link}[1]{%
  \iffieldundef{doi}
    {\iffieldundef{url}
      {\iffieldequalstr{eprinttype}{arxiv}
        {\href{https://arxiv.org/abs/\thefield{eprint}}{#1}}
        {#1}}
      {\href{\thefield{url}}{#1}}}
    {\href{https://doi.org/\thefield{doi}}{#1}}}

\DeclareFieldFormat{title}{\usebibmacro{title+link}{\mkbibemph{#1}}}
\DeclareFieldFormat[article,inbook,incollection,inproceedings,patent,thesis,unpublished]{title}%
  {\usebibmacro{title+link}{\mkbibquote{#1\isdot}}}
  
\graphicspath{{images/}}

\theoremstyle{definition}
\newtheorem{definition}{Definition}[section]

\theoremstyle{remark}
\newtheorem{remark}[definition]{Remark}

\theoremstyle{plain}
\newtheorem{theorem}[definition]{Theorem}
\newtheorem{proposition}[definition]{Proposition}
\newtheorem{lemma}[definition]{Lemma}
\newtheorem{corollary}[definition]{Corollary}

\newcommand{\R}{\mathbb{R}}
\newcommand{\mP}{\mathbb{P}}
\newcommand{\mA}{\mathbb{A}}

\newcommand{\ek}{^{(k)}}
\newcommand{\ekmo}{^{(k-1)}} %elevated to the K Minus One$
\newcommand{\ua}{_{\bm a}} %Underscore A
\newcommand{\uh}{_{\bm h}}
\newcommand{\lk}{\bm{l}^{(k)}}

\newcommand{\fk}{\bm{\phi}^{(k)}}

\newcommand{\bk}{\bm{b}^{(k)}}
\newcommand{\hk}{\bm{h}^{(k)}}
\newcommand{\hkmo}{\bm{h}^{(k-1)}}
\newcommand{\sk}{\sigma^{(k)}}
\newcommand{\wk}{\bm{W}^{(k)}}
\newcommand{\wka}{\bm{W}^{(k)}_{\bm a}}
\newcommand{\wkh}{\bm{W}^{(k)}_{\bm h}}
\newcommand{\NN}[1]{\mathcal{NN}_\mathrm{#1}}
\newcommand{\decnh}{\mathcal{D}_{\bm{\theta}}^{N_h}}
\newcommand{\enh}{\bm{\Phi}^{N_h}}
\newcommand{\ethnh}{\mathcal{E}^{N_h}}

\newcommand{\lipdec}{L_{\mathcal{NN}_\text{dec}}}
\newcommand{\deltas}{\Delta\bm{s}}

\newacronym{ae}{AE}{Autoencoder}
\newacronym{ano}{ANO}{Averaging Neural Operator}
\newacronym{cfd}{CFD}{Computational Fluid Dynamics}
\newacronym{cnef}{CNeF}{Conditional Neural Field}
\newacronym{cnn}{CNN}{Convolutional Neural Network}
\newacronym{deeponet}{DeepONet}{Deep Operator Network}
\newacronym{dof}{DoF}{Degree of Freedom}
\newacronym{elu}{ELU}{Exponential Linear Unit}
\newacronym{fe}{FE}{Finite Element}
\newacronym{film}{FiLM}{Feature-wise Linear Modulation}
\newacronym{filmldm}{FiLM-LDM}{FiLM Latent Dynamics Model}
\newacronym{filmrom}{FiLM-ROM}{FiLM Reduced Order Model}
\newacronym{fnn}{FNN}{Feedforward Neural Network}
\newacronym{fno}{FNO}{Fourier Neural Operator}
\newacronym{fom}{FOM}{Full-Order Method}
\newacronym{fv}{FV}{Finite Volume}
\newacronym{gca}{GCA-ROM}{Graph Convolutional Autoencoder}
\newacronym{gcn}{GCN}{Graph Convolutional Netork}
\newacronym{gdl}{GDL}{Geometric Deep Learning}
\newacronym{gfn}{GFN-ROM}{Graph Feedforward Network}
\newacronym{gnn}{GNN}{Graph Neural Network}
\newacronym{gno}{GNO}{Graph Neural Operator}
\newacronym{gpr}{GPR}{Gaussian Process Regression}
\newacronym{inr}{INR}{Implicit Neural Representation}
\newacronym{lbfgs}{L-BFGS}{Limited-memory BFGS}
\newacronym{ld}{LD}{Lid-driven Cavity}
\newacronym{ldfnn}{LD-FNN}{Latent Dynamics Feedforward Neural Network}
\newacronym{ldgcn}{LD-GCN}{Latent Dynamics Graph Convolutional Network}
\newacronym{ldm}{LDM}{Latent Dynamics Model}
\newacronym{ldnet}{LDNet}{Latent Dynamics Network}
\newacronym{mlp}{MLP}{MultiLayer Perceptron}
\newacronym{mh}{MH}{Moving Hole Advection}
\newacronym{mor}{MOR}{Model Order Reduction}
\newacronym{mse}{MSE}{Mean Squared Error}
\newacronym{nef}{NeF}{Neural Field}
\newacronym{nif}{NIF}{Neural Implicit Flow}
\newacronym{nn}{NN}{Neural Network}
\newacronym{no}{NO}{Neural Operator}
\newacronym{node}{NODE}{Neural Ordinary Differential Equation}
\newacronym{nrmse}{NRMSE}{Normalized Root Mean Square Error}
\newacronym{ode}{ODE}{Ordinary Differential Equation}
\newacronym{pde}{PDE}{Partial Differential Equation}
\newacronym{pinn}{PINN}{Physics-Informed Neural Network}
\newacronym{pod}{POD}{Proper Orthogonal Decomposition}
\newacronym{rb}{RB}{Reduced Basis}
\newacronym{rbf}{RBF}{Radial Basis Function}
\newacronym{relu}{ReLU}{Rectified Linear Unit}
\newacronym{rnn}{RNN}{Recurrent Neural Network}
\newacronym{rom}{ROM}{Reduced Order Model}
\newacronym{sa}{SA}{Square Advection}
\newacronym{sciml}{SciML}{Scientific Machine Learning}
\newacronym{shred}{SHRED}{Shallow Recurrent Decoder}
\newacronym{sindy}{SINDy}{Sparse Identification of Nonlinear Dynamics}
\newacronym{sindynet}{SINDy-Net}{SINDy-Net}
\newacronym{uat}{UAT}{Universal Approximation Theorem}

\begin{document}
\title[Meshless Latent Dynamics Models]{Meshless Latent Dynamics Models for Model Order Reduction of time-dependent Partial Differential Equations}
\author{Lorenzo Tomada$^{1}$ \orcidA, Federico Pichi$^{1}$ \orcidB, Gianluigi Rozza$^{1}$ \orcidC}
\address{$^1$ mathLab, Mathematics Area, SISSA, via Bonomea 265, I-34136 Trieste, Italy}
\email{\texttt{\{ltomada, fpichi, grozza\}@sissa.it}}

\input{sections/abstract}
\maketitle
%\newif\ifincludeKeywords
%\includeKeywordstrue 
\glsresetall

\tableofcontents
\input{sections/introduction}
\input{sections/methodology}
\input{sections/analysis}
\input{sections/results}
\input{sections/conclusion}

%\clearpage
\renewcommand{\bibfont}{\footnotesize}
\printbibliography[]
\input{sections/appendix}

\end{document}

%% file: sections/abstract.tex
\begin{abstract}
Latent dynamics learning has emerged as a prominent approach for \gls{mor} of time-dependent, parameterized \glspl{pde}.
Traditional autoencoder-based approaches are intrinsically tied to fixed spatial discretizations during both training and inference, motivating the development of meshless methodologies.
In this context, conditioning provides a natural mechanism for representing temporal evolution through spatially continuous functions, enabling mesh-agnostic \glspl{ldm}.

In this work, we introduce a unified mathematical framework for mesh-agnostic \glspl{ldm} based on conditioning in \glspl{nef}.
Leveraging a taxonomy of conditioning mechanisms (weak, strong, and pure), we formulate meshless \glspl{ldm} as \glspl{nef} conditioned on a latent state.
We establish a \gls{uat} for this class of models and analyze their stability by deriving Lipschitz bounds for different conditioning mechanisms.
Our analysis reveals an expressivity--stability trade-off, providing insight into why more expressive conditioning, despite its higher representational capacity, can potentially degrade performance during temporal extrapolation.
Finally, we provide a comprehensive comparison of meshless and mesh-based architectures in the continuous limit of the spatial discretization, showing that meshless models guarantee grid-independent Lipschitz bounds, unlike mesh-based counterparts.

Our experiments across Burgers' equation, Navier--Stokes flow past a cylinder, and linear advection show that meshless \glspl{ldm} can achieve predictive accuracy comparable to or better than their mesh-based counterparts while using only 2\% of the trainable parameters and exhibiting nearly constant error across spatial discretizations.

\textbf{Code availability}: \url{https://github.com/lorenzotomada/meshless_ldms}

\end{abstract}

%% file: sections/introduction.tex
\section{Introduction}\label{sec:introduction}
The numerical approximation of \glspl{pde} underpins a wide range of applications, including engineering~\cite{li_geometry-informed_2023}, medicine~\cite{fedele_comprehensive_2023}, and computational physics~\cite{pope_turbulent_2000}.
Nevertheless, a detailed investigation of parameterized systems using full-order models, e.g.\ \gls{fe} and \gls{fv} methods~\cite{quarteroni_numerical_2017}, is often infeasible due to the prohibitive computational burden of repeated simulations, which is further exacerbated for time-dependent problems~\cite{peherstorfer_data-driven_2016}.
\gls{mor} addresses this limitation by designing surrogate models that enable reliable many-query and real-time system investigation at a reduced computational cost~\cite{benner_model_2017, quarteroni_reduced_2016}.

While classical linear strategies such as the \gls{rb} approach~\cite{hesthaven_certified_2016,rozza_real_2024} are well-established and enable efficient online evaluation with certified error bounds, they frequently display significant limitations when applied to complex systems characterized by a slow-decaying Kolmogorov $n$-width~\cite{ohlberger_reduced_2016}.
Recent advances in deep learning have considerably broadened the scope of \gls{mor} by replacing the linear ansatz with nonlinear latent representations learned directly from data, frequently using \glspl{ae}~\cite{lee_model_2020,fresca_comprehensive_2021, pichi_graph_2024}.
In the setting of time-dependent systems, a prominent instance of this paradigm is latent dynamics modeling~\cite{takeishi_learning_2017, lusch_deep_2018}, which aims at capturing the system's temporal evolution within a low-dimensional latent space, while a separate mapping reconstructs the corresponding high-dimensional physical state.
This decoupling substantially reduces the computational complexity of time integration in \glspl{ldm} by exploiting the intrinsic low-dimensional structure of the solution manifold~\cite{linot_deep_2020}.
Importantly, the empirical success of these reduced architectures is now increasingly supported by rigorous theoretical guarantees, such as \glspl{uat}~\cite{farenga_latent_2025, tomada_latent_2026}. 

Despite their success, many existing latent space approaches remain inherently tied to a prescribed spatial discretization, as both the dimensionality reduction and reconstruction stages operate exclusively on mesh-based state vectors with a fixed number of degrees of freedom~\cite{kovachki_neural_2023}.
This limits their applicability across varying spatial resolutions and prevents evaluating the solution at arbitrary, off-grid locations.
Furthermore, it ties the model's parameter count to the specific discretization employed during training and evaluation.
As a result, changing the spatial resolution may require the use of interpolation, which introduces numerical errors, or training a new model, which incurs additional computational costs. Both options ultimately limit the transferability of the discovered dynamics across discretizations.

A promising alternative is provided by \glspl{nef}, which represent physical fields as continuous functions of the spatial coordinates~\cite{xie_neural_2022}.
By directly learning maps from these low-dimensional coordinates to physical quantities, \glspl{nef} naturally yield resolution-independent representations, decoupled from the underlying (typically high-dimensional) discretization of the domain~\cite{chen_crom_2023}.
However, time-dependent or parameterized systems require the field representation to account for changes in the system state without retraining.
To address this, the temporal evolution captured in the latent space is directly incorporated into the continuous field through \emph{conditioning}, a procedure by which the latent variables modulate the network's weights and biases.
The resulting \glspl{cnef} allow a single architecture to describe an entire family of spatial fields corresponding to different physical parameters and times.
Building upon this idea, recent methodologies combine low-dimensional latent dynamics with meshless decoders~\cite{serrano_operator_2023, yin_continuous_2023, regazzoni_learning_2024}, resulting in mesh-agnostic architectures for parameterized, time-dependent systems.

Nonetheless, the integration of \glspl{cnef} into dynamics modeling currently lacks a unified theoretical framework, leaving the design of latent conditioning mechanisms largely empirical.
This raises fundamental questions regarding the trade-offs among different conditioning schemes, particularly with respect to error accumulation over time.
More broadly, several theoretical properties of these architectures remain poorly understood, including approximation guarantees, stability under latent perturbations, and the continuous limit of mesh refinement (connected to the operator learning framework~\cite{berner_principled_2026}).

In this work, we address these gaps by introducing a unified framework for meshless \glspl{ldm} based on \glspl{cnef}, in which the decoder's dependence on the latent variables is interpreted as a conditioning mechanism.
Within this framework, we show that the conditioning strategy determines the sensitivity of the decoded solution to latent perturbations, revealing an expressivity--stability trade-off: highly expressive decoders can amplify temporal integration errors, whereas weak conditioning yields tighter bounds on this amplification.
We further establish a \gls{uat} for weakly conditioned meshless \glspl{ldm} (which naturally extends to stronger forms of conditioning), showing that weak conditioning retains sufficient expressivity while offering greater robustness to latent perturbations.
Moreover, we compare meshless and mesh-based \glspl{ldm} by investigating their Lipschitz constants in the continuous limit of mesh refinement, further quantifying how these bounds influence the magnitude of the learned latent trajectories.
Finally, we numerically validate our theoretical findings on three computational mechanics benchmarks: the Burgers' equation, the Navier--Stokes equations, and a linear advection problem.

The remainder of this work is organized as follows.
After a brief overview of the literature on \glspl{cnef} in \gls{sciml}, Section~\ref{sec:methodology} formulates the problem and introduces conditioning in \glspl{nn} alongside mesh-based and meshless methodologies.
Section~\ref{sec:analysis} presents our theoretical analysis and a comparison with standard \gls{sciml} paradigms.
Finally, Section~\ref{sec:results} validates the proposed approach on the three aforementioned computational mechanics benchmarks, while Section~\ref{sec:conclusion} summarizes our findings, limitations, and future research directions.

\subsection{Related works}
In the \gls{mor} literature, strategies are broadly classified into \emph{intrusive} approaches, which require projecting the underlying governing equations onto a reduced space, and \emph{non-intrusive} (or \emph{data-driven}) approaches, which rely entirely on snapshot data to construct surrogates~\cite{hesthaven_non-intrusive_2018}.
Within the latter category, the use of \glspl{ae}~\cite{hinton_reducing_2006,prince_understanding_2023} as a nonlinear extension of classical dimensionality reduction techniques, such as \gls{pod}~\cite{hesthaven_certified_2016}, has been widely investigated~\cite{lee_model_2020,fresca_comprehensive_2021}.
This paradigm has been extensively applied to latent dynamics modeling, where the reduced state is advanced in time by means of a learned evolution operator, for instance using~\glspl{node}~\cite{lusch_deep_2018, rubanova_latent_2019} or \gls{sindy}~\cite{champion_data-driven_2019,bakarji_discovering_2023, conti_reduced_2023}.

Building upon the non-intrusive \gls{mor} concepts outlined above, several contemporary approaches have substituted discrete spatial decoders with \glspl{nef}~\cite{sitzmann_implicit_2020,seidman_nomad_2022, li_smore_2026}, bridging reduced order modeling with coordinate-based functional representations, which are resolution-invariant and thus aligned with the operator learning paradigm~\cite{kovachki_operator_2024}. 
\glspl{pinn}~\cite{raissi_physics-informed_2019} provide an early example of coordinate-based architectures in \gls{sciml}.
Expanding on these foundations, recent approaches have further investigated the integration of decoder conditioning with respect to a latent state, both in physics-based~\cite{kim_physics-informed_2024} and data-driven scenarios~\cite{regazzoni_learning_2024}.

Existing approaches based on \glspl{nef} can be broadly categorized according to the quantity evolved during time integration.
One class of methods evolves the network weights of the \gls{nef} itself, thus updating the entire functional representation over time~\cite{bruna_neural_2024,wan_evolve_2023,du_evolutional_2021,rosso_weight-parameterization_2025,akhare_implicit_2025}.
Alternatively, several approaches evolve only a low-dimensional latent state while conditioning the decoder on the latent representation~\cite{serrano_operator_2023,kim_generalizable_2023}.
The present work focuses on this latter strategy, where common methods for evolving the latent dynamics include discrete recurrent updates~\cite{chen_crom_2023}, \glspl{node}~\cite{yin_continuous_2023, li_deep_2026}, transformer-based architectures~\cite{hagnberger_vectorized_2024,han_predicting_2022}, and sequence models~\cite{tiwari_latent_2025}.

A key component of these architectures is the conditioning mechanism through which the latent representation modulates the decoder.
Originally introduced in computer vision~\cite{xie_neural_2022}, conditioning strategies are commonly categorized into \emph{global} and \emph{local} approaches.
Global conditioning injects a shared latent representation uniformly across all spatial query points~\cite{shokar_conditioning_2025}, whereas local conditioning additionally depends on the spatial query or its neighborhood, thereby enabling operator learning capabilities~\cite{azizzadenesheli_neural_2024}.
Practical implementations include input concatenation, hypernetworks, and \gls{film}~\cite{perez_film_2017, koishekenov_exploration_2023, brockschmidt_gnn-film_2020}, with the latter becoming particularly popular in both \glspl{cnef}~\cite{guo_conditional_2026} and \glspl{no}~\cite{wu_linear-nonlinear_2026,mikus_ae-vit_2026}, as well as in hybrid numerical methods~\cite{chen_hypernetwork-conditioned_2026}.
More broadly, several existing \gls{no} architectures can be interpreted as \glspl{cnef}~\cite{wang_cvit_2025}, or as functional generalizations thereof~\cite{berner_principled_2026,alkin_universal_2024}, establishing a unified language for function approximation and operator learning.

Finally, recent developments have increasingly moved beyond the classical encoder--decoder paradigm toward encoder-free representations~\cite{regazzoni_learning_2024, faraji_shallow_2025,tomasetto_reduced_2025}.
By avoiding fixed-dimensional spatial encoders, these architectures reduce the dependence of the model on the underlying discretization and naturally shift the focus toward the decoder as the component mapping latent representations back to the physical space.
At the same time, the rapid development of such mesh-agnostic architectures has highlighted the need for a stronger theoretical understanding of deep learning-based \gls{mor}, building on analyses developed for discrete settings~\cite{anderson_mlasdi_2026,brivio_error_2024, franco_approximation_2023} and their links with \glspl{no}~\cite{bhattacharya_model_2021}.
In particular, while the conditioning mechanisms discussed above are central to the expressivity of \glspl{cnef}, their behavior in time-dependent settings and under spatial refinement remains less understood.
Motivated by this gap, we provide a theoretical characterization of mesh-agnostic \glspl{ldm}.
Specifically, we derive a \gls{uat} for \glspl{cnef} and characterize their sensitivity to latent perturbations, while also investigating the behavior of decoder Lipschitz constants under spatial refinement and highlighting their dependence on the discretization parameter for mesh-based models.
Overall, these results shed light on the stability and discretization dependence of continuous latent representations.

%% file: sections/methodology.tex
\section{Latent Dynamics Models}\label{sec:methodology}
In this section, we introduce the mathematical framework underlying \glspl{ldm} for parameterized \glspl{pde}.
We begin by stating the formal problem and introducing \glspl{ldm} as a reduced order representation of the temporal evolution of the system.
Next, we review \glspl{nef} and the conditioning mechanisms for \glspl{nn}, showing that they provide a natural framework for meshless \glspl{ldm}.
Building on this perspective, we comprehensively present both mesh-based and mesh-agnostic architectures, highlighting similarities and differences among several existing methods.

\subsection{Formal problem statement and Latent Dynamics Models}\label{subsec:introduction}
Consider the parameterized evolutionary \gls{pde}:
\begin{equation}\label{eq:abstract_PDE}
    \begin{cases}
       \dfrac{\partial\bm{u}}{\partial t}(\bm x, t) = \mathcal{F}(t, \bm u(\bm x, t);  \bm\mu(t)) & \forall(\bm x,t)\in\Omega\times (t_0,T],\\
       \bm u(\bm x, t_0) = \bm u_0(\bm x;\bm\mu)& \forall\bm x \in\Omega,
    \end{cases}
\end{equation}
depending on a possibly time-dependent signal\footnote{If the signals under investigation are independent of time, we identify $\mP$ with a subset of $\R^{d_{\bm\mu}}$ with a slight abuse of notation.} $\bm\mu\in\mP\subset\{\bm\nu:[t_0,T]\to\R^{d_{\bm\mu}}\}$, where $\Omega\subset\R^d$ is a bounded spatial domain, $I=[t_0,T]$ denotes the time interval, and $\bm u:{\Omega}\times I \to\R^{d_{\bm u}}$ is the solution of interest.
Additionally, $\mathcal{F}$ denotes a suitable differential operator, and $\bm u_0(\bm x;\bm\mu)$ is the initial condition.
In general, several choices for $\bm u_0(\bm x;\bm\mu)$ may be considered for the same parameter instance $\bm\mu\in\mP$.
For clarity, given $\bm\mu\in\mP$, we will denote by $\bm u(\bm x,t;\bm\mu)$ the corresponding solution to Equation~\eqref{eq:abstract_PDE} evaluated at the point $\bm x$ at time $t$.

In practice, the solution of Equation~\eqref{eq:abstract_PDE} is usually sought employing a suitable numerical discretization~\cite{quarteroni_numerical_2017}, which we refer to as the \emph{full-order method}.
We will assume, as commonly holds for numerical methods for time-dependent \glspl{pde}, that the procedure consists of two different steps.
First, a spatial discretization is performed, yielding a set of \glspl{ode} on the grid points, which reads as follows:
\begin{equation}\label{eq:discretized_abstract_PDE}
    \begin{cases}
    \dot{\bm{u}}_h(t;\bm\mu) = \bm{f}_h(t, \bm u_h(t;\bm\mu); \bm\mu(t)) & \forall t\in (t_0,T],\\
       \bm u_h(t_0;\bm\mu) = \bm u_{0,h}(\bm\mu),
    \end{cases}
\end{equation}
where $N_h$ denotes the number of \glspl{dof} associated with the discretization parameter $h$, and $\bm u_h:I\times\mP\to\R^{N_h\times d_{\bm u}}$ represents the discretized solution at the $N_h$ selected \glspl{dof}\footnote{For simplicity, we assume that the number of \glspl{dof} coincides with the number of grid points, although this assumption is not required.}, while $\bm u_{0,h}:\mP\to\R^{N_h\times d_{\bm u}}$ and $\bm f_h:(t_0,T]\times\R^{N_h\times d_{\bm u}}\times\R^{d_{\bm\mu}}\to\R^{N_h\times d_{\bm u}}$ denote suitable discretizations of $\bm u_0$ and $\mathcal{F}$ at the grid points, respectively.
Then, the semi-discrete system in~\eqref{eq:discretized_abstract_PDE} can be integrated using any \gls{ode} solver, providing an approximation of the solution to the original system in Equation~\eqref{eq:abstract_PDE}.

In many applications, however, $N_h\gg1$, meaning that a large number of \glspl{dof} is required to achieve sufficient accuracy, leading to high computational costs and infeasible simulations in parameterized settings.
As a result, suitable \gls{mor} techniques have been developed to reduce the computational burden~\cite{benner_model_2017, hesthaven_certified_2016, quarteroni_reduced_2016}.

A common approach for dynamical systems is to model the temporal evolution in a low-dimensional reduced space~\cite{lee_model_2020} by identifying a suitable system of \glspl{ode}, which we refer to as the \emph{latent dynamics}.
Once the system has been integrated, the resulting latent trajectories are mapped back to the full-order space by means of a decoder function, providing an approximation of the physical solution.
For meshless models, the decoder can be formulated as a \gls{nef}, with the latent state acting as a condition that determines the reconstructed spatial field.
This naturally motivates a closer examination of how the latent state is incorporated into the decoder, which we formalize through the conditioning framework introduced next.

\subsection{Conditioning in Neural Fields}\label{subsec:conditioning}
We now introduce \glspl{nef} and formalize the mechanisms for conditioning them, expanding upon the frameworks established in~\cite{xie_neural_2022, koishekenov_exploration_2023, brockschmidt_gnn-film_2020}, to which we refer for a broader overview.
Importantly, we aim to analyze these mechanisms from a theoretical perspective, discussing their expressivity and the Lipschitz stability of the network.
The formal mathematical notation employed throughout this section is detailed in Appendix~\ref{appendix:formalism}.

\subsubsection{Definitions and types of conditioning}
Given a bounded spatial domain $\Omega\subset\R^{d}$ and, possibly, a time interval $I=[t_0,T]$, a \gls{nef}, or \emph{implicit neural representation}, is a \gls{nn} $\bm g_{\bm\theta}:\Omega\times I\to\R^{d_{\bm u}}$ parameterized with weights $\bm\theta\in\R^{N_{\bm\theta}}$ that directly maps spatio-temporal coordinates $(\bm x,t)\in\Omega\times I$ to the corresponding field values, unlike networks operating on discrete grids.

To capture parameterized families of fields without retraining, conditioning mechanisms enable the weights of a \gls{nn} to adapt to characteristic quantities $\bm a\in\mA\subset\R^{N_{\bm a}}$, such as time, physical parameters, or latent states.
We now define the structure of a conditioned layer.
 
\begin{definition}[Conditioned layer]\label{def:conditioned_layer}
    Let $\bm W^{(k)}:\mA\to\R^{N_{k}\times N_{k-1}}$ and $\bm b^{(k)}: \mA\to\R^{N_k}$.
    We say that a \gls{nn} layer\footnote{With a slight abuse of notation, we will denote with $\bm l\ek$ both the family of layers $\left\{\bm l\ek(\,\cdot\,|\,\bm a)\right\}_{\bm a\in\mA}$ and a specific realization for given $\bm a\in\mathbb A$.} $\lk$ is \emph{(purely)  conditioned} on a parameter $\bm a$, and we denote it by $\bm l\ek(\cdot\,|\,\bm a):\R^{N_{k-1}}\to\R^{N_k}$, if for any input $\hkmo\in\R^{N_{k-1}}$ its forward pass can be written as
\begin{equation}\label{eq:conditioning}
        \hk=\bm l\ek(\hkmo\,|\,\bm a)=\sk\left(\wk(\bm a)\,\hkmo+\bk(\bm a)\right),
\end{equation}
where $\sk:\R\to\R$ is an activation function applied elementwise.
\end{definition}

Conditioning mechanisms can be broadly divided into \emph{local} and \emph{global} conditioning~\cite{xie_neural_2022}.
In the former, the conditioning parameter $\bm a$ is a field depending on the query point $\bm x$, allowing the modulation to vary across the spatial domain, consistent with the operator learning setting~\cite{wang_cvit_2025}.
In contrast, global conditioning employs the same conditioning parameter across the spatial domain, uniformly modulating the network independently of the query point.
In this work, we focus on global conditioning, as it provides a direct way to describe the modulation of the spatial decoder by a latent state that is independent of the query point.
Several forms of global conditioning have been proposed, including hypernetworks~\cite{pan_neural_2023, sitzmann_implicit_2020}, \gls{film}~\cite{yin_continuous_2023, hagnberger_vectorized_2024}, input concatenation~\cite{regazzoni_learning_2024, serrano_operator_2023}, and low-rank weight parameterizations~\cite{skorokhodov_adversarial_2021}.
To classify these mechanisms according to how the conditioning parameter $\bm a$ interacts with the layer parameters, we follow the discussion in~\cite{koishekenov_exploration_2023}.
In particular, we distinguish three paradigms: pure conditioning (Definition~\ref{def:conditioned_layer}), and weak and strong conditioning, which we formalize below and which are special cases of pure conditioning, as shown in Lemma~\ref{lemma:conditioning}.

\begin{definition}[Types of conditioning]\label{def:types}
    The conditioned layer $\lk:\R^{N_{k-1}}\to\R^{N_{k}}$ is said to be:
    \begin{enumerate}[label=(\roman*)]
        \item \emph{weakly conditioned} (concatenating input features with the condition), if the weight matrix $\tilde{\bm W}\ek\in\R^{N_{k}\times(N_{k-1}+N_{\bm a})}$ is independent of $\bm h\ekmo$ and $\bm a$, and the update is given by
    \begin{equation}\label{eq:weak}
        \hk=\sk\left(\tilde{\bm W}^{(k)}(\hkmo\oplus\bm a)\right),
    \end{equation}
    where the symbol $\oplus$ denotes the concatenation of $\hkmo$ and $\bm a$ as a vector in $\R^{N_{k-1}+N_{\bm a}}$;
    \item \emph{strongly conditioned} (or \emph{FiLM conditioned}) if the weight matrix $\wkh\in\R^{N_{k}\times N_{k-1}}$ is independent of $\hkmo$ and $\bm a$, while $\fk:\mA\to\R^{{N_k}}$ and $\bk:\mA\to\R^{N_{k}}$ depend on $\bm a$, so that the update is given by
    \begin{equation}\label{eq:strong}
        \hk=\sk\left(\fk(\bm a)\odot\wkh\hkmo +\bk(\bm a)\right),
    \end{equation}
    where the symbol $\odot$ denotes the Hadamard (i.e., componentwise) product.
    \end{enumerate}
\end{definition}

By Lemma~\ref{lemma:conditioning}, weak conditioning corresponds to the choice of only adding a conditional bias depending on $\bm a$ without altering the parameters of $\wk$, i.e., there exist suitable $\wkh\in\R^{N_k\times N_{k-1}}$ and $\wka\in\R^{N_k\times N_{\bm a}}$, independent of $\bm h\ekmo$ and $\bm a$, such that $\bm W\ek(\bm a)=\wkh$ and $\bm b\ek(\bm a)\coloneq \wka\bm a$.
In contrast, both strong and pure conditioning allow $\bm a$ to parameterize the weight matrix $\wk$.
In particular, pure conditioning offers maximum expressivity by dynamically generating the entire transformation matrix $\wk$ and bias $\bk$; strong conditioning (\gls{film}) restricts this parameterization to an elementwise gating $\bm\phi^{(k)}$ applied to the output of a static linear layer $\wkh$~\cite{brockschmidt_gnn-film_2020}, preventing $\bm a$ from contributing to channel mixing.
We highlight that the conditioning mechanisms in Definition~\ref{def:types} form a hierarchy of increasing expressivity: Remark~\ref{rmk:all_ones} and Lemma~\ref{lemma:conditioning} establish that weak conditioning is a special case of strong conditioning, which is, in turn, a special case of pure conditioning.

While Definition~\ref{def:conditioned_layer} is very general, in practice $\bm W^{(k)}$, $\bm b^{(k)}$, and $\bm\phi^{(k)}$ are implemented as \glspl{nn}, commonly referred to as \emph{hypernetworks}, allowing them to be learned from data.
Unlike the dense hypernetworks required for pure conditioning, which scale poorly and might incur prohibitive training costs when conditioning on a latent state $\bm a=\bm s(t;\bm\mu)$ (as discussed in Appendix~\ref{appendix:formalism}), weak and \gls{film} conditioning only modulate selected components.
By retaining a set of base weights independent of $\bm a$, this structural separation (mimicking the well-known technique of separation of variables for \glspl{pde}~\cite{evans_partial_1998}) significantly reduces parameter counts and the number of operations in a forward pass (see Appendix~\ref{appendix:formalism}) and enhances interpretability.
Moreover, it can be formally justified by means of the following \gls{uat}.
\begin{corollary}[FiLM neural networks are universal approximators]\label{theo:film}
Shallow networks with the first layer conditioned using \gls{film} and a second non-conditioned linear layer are universal approximators of functions $g:\R^d\times\mathbb A\to\R$.
\end{corollary}
The proof of Corollary~\ref{theo:film} is reported in Appendix~\ref{appendix:formalism} and is an immediate consequence of Remark~\ref{rmk:all_ones} and classical \glspl{uat} for weak conditioning.

Although conditioning a single layer already yields a conditioned \gls{nn}, empirical evidence suggests that conditioning multiple layers often leads to performance improvements~\cite{koishekenov_exploration_2023}.
Nevertheless, as discussed in Section~\ref{sec:analysis}, latent dynamics modeling exhibits structural challenges that may prevent this increased depth of conditioning from directly translating into improved performance.

\subsubsection{Lipschitz constant of conditioned layers}
We now analyze the impact of conditioning on the Lipschitz constant of the network.
In particular, in the next proposition (whose proof is deferred to Appendix~\ref{appendix:stability}), we show that a more expressive conditioning yields a larger upper bound on the Lipschitz constant of the network with respect to the conditioning parameter, potentially impacting the stability of latent dynamics learning and \gls{mor}.

\begin{proposition}[Stability result for conditioned networks]\label{prop:lipschitz}
    Let $\bm l$ be a conditioned layer, and assume that $\sigma:\R\to\R$, $\bm W:\mathbb A\to\R^{N_\mathrm{out}\times N_\mathrm{in}}$ and $\bm b:\mathbb A\to\R^{N_\mathrm{out}}$ are Lipschitz continuous with respect to the $p$-norm $\|\cdot\|_p$, with $p\geq1$, and denote their Lipschitz constants by $C_\sigma$, $C_{\bm W}$ and $C_{\bm b}$, respectively.
    Then, for each $\bm h\in\R^{N_\mathrm{in}}$, it holds that
    \begin{equation}\label{eq:stability}
        \|\bm{l}(\bm h\,|\,\bm a_1)-\bm{l}(\bm h\,|\,\bm a_2)\|_p\leq C\|\bm a_1-\bm a_2\|_p,\quad \forall \bm a_1,\bm a_2\in\mathbb A,
    \end{equation}
    where $C=C_\sigma(C_{\bm W}\|\bm h\|_p+C_{\bm b})$.
    If $\bm b(\bm a)=\bm W\ua \bm a$,  then $C_{\bm b}\leq\|\bm W_{\bm a}\|_p$.
In addition, if strong conditioning is used and $\bm\phi:\mathbb A\to\R^{N_\mathrm{out}}$ is Lipschitz continuous with constant $C_{\bm\phi}$, it holds that $C_{\bm W}\leq C_{\bm \phi}\|\bm W_{\bm h}\|_p$.
Finally, if weak conditioning is used, $C_{\bm W}=0$.
\end{proposition}

A direct consequence of the previous proposition is that weak conditioning benefits from greater stability properties since the Lipschitz bound $C = C_\sigma C_{\bm b}$ only depends on the additive bias shift.
In contrast, strong and pure conditioning introduce a non-zero $C_{\bm W}$ term, coupling the parameter perturbation bound directly to the magnitude of the intermediate features $\|\bm h\|_p$. 

Although this term can be mitigated by careful hypernetwork design, the dependence on $\|\bm h\|_p$ suggests that in deep architectures, perturbations in $\bm a$ amplify multiplicatively across layers, justifying the standard practice of employing suitable kinds of normalization~\cite{prince_understanding_2023, koishekenov_exploration_2023} to help control the feature norms $\|\bm h\|_p$, thereby stabilizing the forward pass.
As we will show in Subsection~\ref{subsec:stability_conditioning}, controlling this bound is critical when $\bm a$ represents a latent state evolving via an approximate \gls{ode}, where integration errors can be amplified by large Lipschitz constants.

\subsection{Meshless and mesh-based Latent Dynamics Models}\label{subsec:ldms}
We now formally introduce the setting of latent dynamics modeling and discuss its use in combination with meshless and mesh-based models, which are illustrated in Figure~\ref{fig:architecture}.
For \glspl{ldm}, the core idea is to identify a set of \glspl{ode} governing the latent dynamics, which can subsequently be evolved in time.
Decoding the reduced state then yields an efficient approximation of the original solution.
The choice of the decoder determines the mathematical nature of the output: it can either be a set of discrete values associated with $N_h$ degrees of freedom or a function evaluable at arbitrary points $\bm x\in\Omega$, as in the operator learning paradigm~\cite{berner_principled_2026}.

\begin{figure}[tbp]
    \centering
    \includegraphics[width=\textwidth]{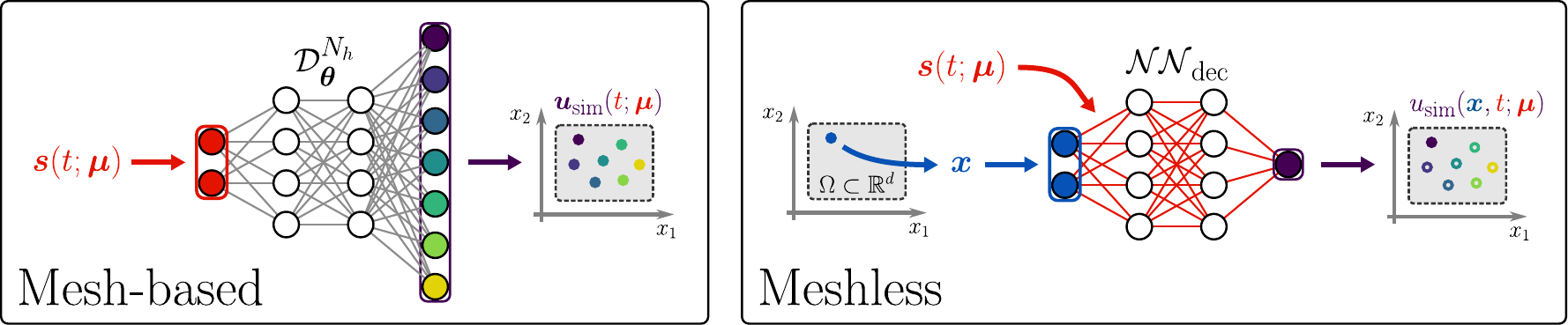}
    \caption{Sketches of mesh-based and meshless decoder architectures.}
    \label{fig:architecture}
\end{figure}

A key challenge in dynamical systems modeling is ensuring that the resulting \gls{rom} preserves causality, i.e., the intrinsic temporal consistency of Equations~\eqref{eq:abstract_PDE} and~\eqref{eq:discretized_abstract_PDE}.
Although alternative approaches exist~\cite{pfaff_learning_2021,chen_time_2025, kovachki_operator_2024}, a common strategy consists in describing the temporal evolution of the system in a suitable reduced space $\R^n$,  with $n\ll N_h$~\cite{lusch_deep_2018}, through a set of \glspl{ode}:
\begin{equation}\label{eq:latent_dynamics}
\begin{cases}
    \dot{\bm z}(t;\bm\mu) = \bm f_n(t,\bm z(t;\bm\mu); \bm\mu(t)) &\forall (t,\bm\mu)\in (t_0,T]\times\mP, \\
    \bm z(t_0;\bm\mu) = \enh(\bm u_{0,h}(\bm \mu)) &\forall \bm\mu\in\mP,
\end{cases}
\end{equation}
where $\bm z(t;\bm\mu)\in\R^n$ represents the reduced state of the system, $\enh(\bm u_{0,h}(\bm\mu))\in\R^n$ is an encoding of the full-order initial condition $\bm u_{0,h}$, and $\bm f_n:(t_0,T]\times\R^n\times\R^{d_{\bm\mu}}\to\R^n$ defines the reduced dynamics.
Equation~\eqref{eq:latent_dynamics} can then be integrated using an \gls{ode} solver to efficiently evolve the latent state in time and reconstruct an approximation of the full-order solution through a suitable decoder~\cite{lee_model_2020}. 
The key benefit is that the dynamical evolution is modeled in a low-dimensional latent space, thereby reducing the computational complexity of the original system.
Moreover, the reduced representation often provides additional insight into the underlying mechanisms governing the system~\cite{champion_data-driven_2019, conti_veni_2026, tomada_sparse_2025}.

In this setup, the encoder only influences the system via the choice of the latent initial condition.
If the latter is fixed for all $\bm\mu\in\mP$, the encoder is trivial and its output can be set to $\bm{0}_{\R^n}$ without loss of generality~\cite{regazzoni_learning_2024, tomada_latent_2026}.

An important property of the latent model is time continuity~\cite{farenga_latent_2025}: in contrast to models that directly learn a discrete-time evolution map with fixed $\Delta t$, Equation~\eqref{eq:latent_dynamics} defines a continuous-time trajectory that can be evaluated at arbitrary time instances independently of the time discretization used by the numerical solver.

The reconstruction of the full-order solution from the latent state generally follows two complementary paradigms: a discrete mesh-based decoder or a continuous meshless decoder.
The former employs a map $\decnh:\R^n\to\R^{N_h\times d_{\bm u}}$ that reconstructs the solution simultaneously at all grid points of a fixed spatial discretization.
Conversely, the meshless approach utilizes a continuous function $\NN{dec}:\R^d\times\R^n\to\R^{d_{\bm u}}$ that can be queried at arbitrary spatial coordinates $\bm x\in\Omega$.
Thus, while both paradigms share the same latent dynamics formulation, they differ in the structure of their decoders.

In practice, the encoder $\enh$, the latent dynamics function $\bm f_n$, and the decoder are generally unknown and must be approximated from data.
Accordingly, the encoder and the latent dynamics in Equation~\eqref{eq:latent_dynamics} are replaced by learned approximations $\ethnh$ and $\NN{dyn}$, respectively.
Similarly, the decoder is given by the learned map $\decnh$ in the mesh-based setting and by $\NN{dec}$ in the meshless one.
In particular, we denote by $\bm s(t;\bm\mu)$ the approximate latent trajectory obtained by integrating the latent dynamics with $\bm f_n$ replaced by $\NN{dyn}$.

\subsubsection{Mesh-based dynamics models}
In the mesh-based setting, under suitable regularity assumptions~\cite{farenga_latent_2025}, a low-dimensional latent state $\bm z(t;\bm\mu)\in\R^n$ is said to satisfy a latent dynamics problem if there exist an encoder $\enh:\R^{N_h\times d_{\bm u}}\to\R^n$, a decoder $\mathcal{D}^{N_h}:\R^n\to\R^{N_h\times d_{\bm u}}$ and a latent dynamics function $\bm f_n:(t_0,T]\times\R^n\times\R^{d_{\bm\mu}}\to\R^n$ such that, for all $(t,\bm\mu)\in I\times\mP$,
\begin{equation}\label{eq:LDM_meshbased1}
    \bm u_h(t;\bm\mu)\approx \mathcal{D}^{N_h}\left(\enh(\bm u_{0,h}(\bm \mu)) +\int_{t_0}^t \bm f_n(\tau, \bm z(\tau;\bm\mu);\bm\mu(\tau))d\tau\right)\in\R^{N_h\times d_{\bm u}},
\end{equation}
where $\bm z$ solves the reduced initial value problem given by Equation~\eqref{eq:latent_dynamics}.

We refer to a \gls{nn} approximation of a mesh-based \gls{ldm} as a \gls{ldfnn} model and denote by $\decnh$ the corresponding learned decoder.
Under appropriate assumptions, a \gls{uat} analogous to Theorem~\ref{theo:uat} holds for these architectures~\cite{farenga_latent_2025}, guaranteeing that, for any $\varepsilon>0$, one can learn approximations that are continuous in time and with an error bounded by $\varepsilon$.

Despite these theoretical guarantees, mesh-based \glspl{ldm} display some significant computational bottlenecks.
Since the solution is reconstructed simultaneously across the entire spatial grid, the encoder and decoder networks $\ethnh$ and $\mathcal{D}_{\bm\theta}^{N_h}$ need to have at least $n N_hd_{\bm u}$ trainable parameters even assuming that only a single fully-connected layer is used.

This $\mathcal{O}(N_h)$ scaling hinders their application to high-resolution meshes.
Furthermore, mesh-based models are intrinsically tied to their training grid, and they cannot be evaluated at arbitrary spatial coordinates without relying on specialized techniques~\cite{lu_learning_2021,morrison_gfn_2024}.
To address these structural limitations, we turn our attention to meshless architectures.

\subsubsection{Meshless dynamics models}
To overcome the $\mathcal{O}(N_h)$ scaling bottleneck of mesh-based approaches and enable multi-resolution training and testing, several approaches take as input the spatial coordinates $\bm x\in\Omega$ to return an approximation of the solution $\bm u(\bm x, t;\bm\mu)$.
A prominent example is that of \glspl{pinn}~\cite{raissi_physics-informed_2019}, which act as \glspl{nef} trained via \gls{pde} residuals~\cite{xie_neural_2022}.

Unlike the mesh-based paradigm, this approach involves learning the pointwise relationship between $\bm x$ and $\bm u(\bm x, t)$ rather than reconstructing the entire field simultaneously for a fixed grid.
While exact solutions $\bm u(\bm x,t;\bm\mu)$ are generally unavailable, discrete solutions (such as \gls{fe} approximations) are often defined outside the mesh points via interpolation~\cite{quarteroni_numerical_2017}.
This allows one to interpret mesh points simply as collocation points to train the architecture, thereby providing a theoretical justification for learning continuous solutions.
Hereafter, we explicitly include the spatial argument\footnote{When $\bm x$ is omitted, $\bm u_h(t;\bm\mu)$ refers to the discrete values in $\R^{N_h \times d_{\bm u}}$ at the grid nodes.}, denoting the continuous evaluation as $\bm u_h(\bm x,t;\bm\mu)$.
This continuous formulation offers significant flexibility, enabling the batching of collocation points during training to reduce computational costs, for instance by using active learning or adaptive strategies~\cite{regazzoni_learning_2024}.

When dealing with time-dependent problems, a common approach is to consider a \gls{nef} characterized by a set of weights that evolve over time.
Several methodologies have been proposed for this purpose, such as sequentially optimizing the parameters at discrete time steps~\cite{chen_implicit_2023}, using hypernetworks to map temporal coordinates directly to the network parameters~\cite{pan_neural_2023}, conditioning the decoder on discrete time-varying latent codes~\cite{chen_crom_2023}, or leveraging vectorized attention mechanisms~\cite{hagnberger_vectorized_2024}.
A further approach entails writing an \gls{ode} that governs all the weights of the \gls{nef}~\cite{bruna_neural_2024, akhare_implicit_2025, rosso_weight-parameterization_2025, wan_evolve_2023, du_evolutional_2021}, thereby defining a family of \glspl{nef} $\{\mathcal{NN}_t:\Omega\to\R^{d_{\bm u}}\}_{t\in I}$.
While this \gls{ode}-based strategy can be formally interpreted as a form of time-conditioning in the sense of Subsection~\ref{subsec:conditioning}, evolving the entire weight space yields a high-dimensional differential problem, making numerical time-stepping computationally challenging.

To circumvent this computational bottleneck and be consistent with the \gls{mor} setting, instead of writing an \gls{ode} on the full set of weights, here we identify a low-dimensional state that is subsequently used to condition the \gls{nef}.
This reduced set of parameters can naturally be interpreted as a latent state $\bm z(t;\bm\mu)$, often mirroring the underlying physical or dynamical properties of the governing system \cite{champion_data-driven_2019, linot_deep_2020}.

Formally, in the case of meshless \glspl{ldm}, we assume the existence of an encoder $\enh:\R^{N_h\times d_{\bm u}}\to\R^n$, a dynamics function $\bm f_n : I\times \R^n\times \R^{d_{\bm\mu}}\to\R^n$, and a decoder $\bm\Psi:\R^d\times\R^n\to\R^{d_{\bm u}}$ such that, for all $(\bm x, t;\bm\mu)\in\Omega\times I\times\mP$,
\begin{equation}\label{eq:meshless_models}
    \bm u_h(\bm x, t;\bm\mu)\approx\bm\Psi\left(\bm x,\, \enh(\bm u_{0,h}(\bm \mu)) +\int_{t_0}^t \bm f_n(\tau, \bm z(\tau;\bm\mu);\bm\mu(\tau))d\tau\right)\in\R^{d_{\bm u}}.
\end{equation}
Compared to mesh-based models in Equation~\eqref{eq:LDM_meshbased1}, the decoder here does not map $\bm z(t;\bm\mu)$ to the values on a discrete grid; instead, it returns a spatial field conditioned on the latent state.
While this shifts the output from a discrete vector to a spatial mapping, the architecture remains straightforward to train using standard data from numerical \gls{pde} solvers, which provide pointwise values at the mesh nodes used to perform simulations.

As in the case of mesh-based models, suitable approximations of $\enh,\bm\Psi$, and $\bm f_n $ must be learned from data.
The encoder $\enh$ is typically approximated by a feed-forward \gls{nn} $\ethnh$, while the decoder is approximated by means of a \gls{nef} that is conditioned on the latent state, i.e., $\NN{dec}(\cdot\,|\,\bm z(t;\bm\mu)):\R^d\to\R^{d_{\bm u}}$, where the spatial input $\bm h^{(0)}=\bm{x}$ is processed according to Equation~\eqref{eq:conditioning} with $\bm a=\bm z(t;\bm\mu)$.
We highlight that the operator that maps $\bm z(t;\bm\mu)$ to $\NN{dec}(\,\cdot\,|\,\bm z(t;\bm\mu))\in L^p(\Omega)$, for $p\in[1,+\infty]$, can be viewed within the framework of \glspl{no} \cite{huang_operator_2025}.

Finally, the function $\bm f_n$ is approximated by a \gls{node}, denoted as $\NN{dyn}$, which outputs the time derivative $\dot{\bm{s}}(t;\bm\mu)$ used to evolve the approximate latent trajectory $\bm s(t;\bm\mu)$.
Importantly, since $\NN{dyn}$ is a \gls{nn} conditioned on $t$ and $\bm\mu(t)$, more expressive conditioning mechanisms can be employed instead of simple input concatenation, analogously to the conditioning used for the decoder.

Summarizing, a learnable meshless \gls{ldm} can be formulated as follows.
Given the value of a signal $\bm\mu(t)\in\R^{d_{\bm\mu}}$ and a latent state $\bm{s}(t;\bm\mu)\in\R^n$ at time $t\in I=[t_0,T]\subset\R^+$, $\NN{dyn}$ returns its time derivative $\dot{\bm{s}}(t;\bm\mu)$ at time $t$.
Conversely, a second sub-network $\NN{dec}:\R^d\times\R^n\to\R^{d_{\bm u}}$ acts as a learned decoder, taking as input the value of $\bm x\in\Omega$, as well as the latent state $\bm s(t;\bm\mu)$ at time $t$, and returning the reconstructed field.
Ultimately, the complete high-fidelity trajectory is approximated by integrating the following initial value problem:
\begin{equation}\label{eq:latent_ODE_2}
    \begin{cases}
        \dot{\bm{s}}(t; {\bm \mu})=\NN{dyn}\bigl(\bm{s}(t; {\bm \mu})\,\big |\,(t,{\bm{\mu}}(t))\bigr),& t\in(t_0,T]\\
         \bm{s}(t_0; {\bm \mu})=\ethnh(\bm u_{0, h}(\bm\mu)),\\
         \bm{u}_\mathrm{sim}(\bm x, {t};\bm \mu)=\NN{dec}\bigl(\bm x\,\big |\, \bm{s}( t; \bm \mu)\bigr),& t\in(t_0,T].
    \end{cases}
\end{equation}
To discretize Equation~\eqref{eq:latent_ODE_2} and advance in time, we use the explicit Euler method for simplicity, although any kind of differentiable time-stepping scheme for \glspl{ode} could be used~\cite{farenga_latent_2025}.

\begin{remark}\label{rmk:signature}
Consistently with other recently proposed methodologies~\cite{regazzoni_learning_2024, tomada_latent_2026}, it is possible to include in the decoder a dependence on $(t,\bm\mu(t))$ as well, e.g., by concatenating it to the latent state.
Similarly, the dependence on $t$ of $\NN{dyn}$ can be dropped when modeling autonomous systems.
We did not consider this case in the theoretical exposition to simplify the formalism.
\end{remark}

To demonstrate the necessity of a common formal framework based on the conditioning taxonomy established in Section~\ref{subsec:conditioning}, as an example, we recast \gls{ldnet} within our framework (assuming here that times are not scaled by any hyperparameter).

\begin{remark}[LDNets as meshless Latent Dynamics Models]\label{rmk:ldnet}
The \gls{ldnet} architecture is a specific instance of a learnable meshless \gls{ldm}, characterized by the following architectural constraints:
\begin{enumerate}[label=(\roman*)]
  \item The \gls{nef} decoder $\NN{dec}$ and the dynamics network $\NN{dyn}$ employ weak conditioning exclusively, and only in their respective first layers (i.e., only concatenation of the condition is used);
    \item No encoder is employed and $\bm s(t_0;\bm\mu) =\bm0_{\R^{n}},\,\forall\bm\mu\in\mP$.
\end{enumerate}
\end{remark}

In the case of fixed initial conditions, the lack of an encoder as in \gls{ldnet} allows one to obtain an architecture that is fully independent of the discretization parameter $N_h$.
We remark that, following recent approaches such as auto-decoding~\cite{pizzamiglio_learning_2026, serrano_operator_2023, yin_continuous_2023}, it is possible to keep the parameter count independent of $N_h$ while allowing for varying initial conditions.

By framing meshless models through this rigorous conditioning taxonomy, we aim at providing a principled architectural design choice, guided by the expressivity--stability trade-off, rather than fully relying on empirical heuristics.

%% file: sections/analysis.tex
\section{Theoretical comparison of meshless and mesh-based models}\label{sec:analysis}
This section analyzes the theoretical properties of meshless and mesh-based models, focusing on their approximation capabilities and potential limitations.

To ensure our comparison is methodologically grounded, we build upon existing guarantees for mesh-based architectures~\cite{farenga_latent_2025} by first establishing a \gls{uat} for time-continuous meshless \glspl{ldm}.
Additionally, we derive an off-grid error bound to formalize the super-resolution capabilities of meshless models on unseen spatial discretizations.

Subsequently, we investigate how different conditioning schemes propagate perturbations of the latent state to the decoded solution, and compare meshless and mesh-based \glspl{ldm} in the continuous limit as $N_h\to\infty$.
This analysis reveals a condition for the stability of mesh-based models that is grounded in high-dimensional probability.
We further discuss quantitative bounds on the magnitude of latent trajectories and their relationship with the Lipschitz constant. 

Finally, we review typical training pathologies associated with \glspl{nef}, such as spectral bias~\cite{xie_neural_2022}, summarize the architectural trade-offs, and provide a comparison with existing methodologies.

\subsection{Universal Approximation Theorem for meshless models}\label{subsec:UAT}
In our analysis, we focus on \gls{ldnet} as a prototypical architecture to prove a \gls{uat} for meshless \glspl{ldm}.
Leveraging Remark~\ref{rmk:ldnet}, we formally state that there exists a learnable \gls{ldnet} (and thus, a learnable meshless \gls{ldm}) satisfying the desired approximation properties.
Contrary to the original \gls{ldnet} methodology, we assume the existence of a learnable encoder, as this addition is necessary to tackle problems characterized by varying initial conditions.
For simplicity, we assume that $d_{\bm u}=1$ (scalar functions), and that $\mP$ is a compact subset of $\R^{d_{\bm\mu}}$ (constant signals).
As previously remarked, if the encoder is not used, the resulting architecture has a number of parameters that is independent of $N_h$.

Following~\cite{farenga_latent_2025}, which establishes an analogous theorem for mesh-based models, we now detail the ingredients specific to the meshless setting, starting from the core assumption on the embedding of the dynamics~\cite{brivio_error_2024}.

\begin{definition}[Perfect embedding assumption]\label{def:perfect_embedding}
Let $I=[t_0,T]$ and $D_n=(t_0,T]\times\R^n\times\mP$; assume that the initial condition $\bm\mu\mapsto\bm u_{0,h}(\bm\mu)=\bm u_h(t_0;\bm\mu)$ of~\eqref{eq:discretized_abstract_PDE} is continuous on $\mP$ and that $u_h(\,\cdot\,, t;\bm\mu)$ is defined on $\overline\Omega$.
We say that a \emph{perfect embedding assumption} is satisfied if there exists $n^\ast\in\mathbb N$ such that, for all $n\geq n^\ast$, there exist
\begin{enumerate}
\item a latent dynamics function $\bm f_n:D_n\to\R^n$ and constants $\ell_t,\ell_{\bm z},\ell_{\bm\mu}\geq0$ such that
\begin{equation*}
\|\bm f_n(t,\bm z;\bm\mu)-\bm f_n(t',\bm z';\bm\mu')\|_2\leq\ell_t|t-t'|+\ell_{\bm z}\|\bm z-\bm z'\|_2+\ell_{\bm\mu}\|\bm\mu-\bm\mu'\|_2\qquad\forall(t,\bm z,\bm\mu),(t',\bm z',\bm\mu')\in D_n,
\end{equation*}
\item an encoder $\enh:\bm u_{0,h}(\mP)\to\R^n$ that is Lipschitz continuous,
\item a decoder $\Psi:\overline\Omega\times\R^n\to\R$ that is Lipschitz continuous on $\overline\Omega\times K$ for every compact set $K\subset\R^n$,
\end{enumerate}
satisfying $u_h(\bm x,t;\bm\mu)=\Psi(\bm x,\bm z(t;\bm\mu))$ for all $(\bm x,t,\bm\mu)\in\overline\Omega\times I\times\mP$, with $\bm z:I\times\mP\to\R^n$ the latent trajectory such that $\bm z(\,\cdot\,;\bm\mu)\in C^1(I;\R^n)$ is the solution of~\eqref{eq:latent_dynamics}.
We further assume that $\bm z$ is Lipschitz continuous.
\end{definition}

\begin{remark}[Perfect embedding in the meshless setting]\label{rmk:relu}
In the mesh-based setting~\cite{farenga_latent_2025}, the perfect embedding assumption involves the approximability of a discrete state vector $\bm u_h(t;\bm\mu)\in\R^{N_h}$ by means of a low-dimensional state $\bm z(t;\bm\mu)$.
By contrast, in the present meshless framework we aim to approximate a continuous map $ u_h(\,\cdot\,,t;\bm\mu):\Omega\to\R$.
Although the hypothesis in Definition~\ref{def:perfect_embedding} might seem different from its discrete counterpart, it can nonetheless be interpreted as a hypothesis on a finite set of values in terms of the spatial variable.
Indeed, in various full-order methodologies, such as the \gls{fe} method, the solution is uniquely determined by its values at the \glspl{dof}, with the continuous representation obtained by combining the corresponding basis functions.
In particular, for linear \gls{fe} discretizations, since piecewise-linear functions can be represented exactly by sufficiently deep \glspl{nn} with \gls{relu} activation~\cite{he_relu_2020, hao_sparsely_2026}, the perfect reconstruction of the continuous function in Definition~\ref{def:perfect_embedding} is  equivalent to reconstructing its values at the \glspl{dof}.
The network can then exactly represent the corresponding piecewise-linear basis functions and evaluate the reconstructed solution at arbitrary points in the domain.
In this sense, consistently with full-order methods and various \gls{mor} techniques, Equation~\eqref{eq:meshless_models} still acts as a reduction of the set of \glspl{ode} on the grid points introduced in~\eqref{eq:discretized_abstract_PDE}, establishing a conceptual bridge between operator learning and traditional \gls{mor}.
\end{remark}

\begin{theorem}[Universal Approximation Theorem for meshless \glspl{ldm}]\label{theo:uat}
Assume the perfect embedding assumption of Definition~\ref{def:perfect_embedding} and let $n\geq n^*$.
Let $\sigma$ be continuous and satisfy the assumptions of~\cite[Proposition~4.8]{guhring_approximation_2021} (e.g., $\tanh$ or \gls{elu}), and consider the uniform grids $\mathcal T_{\Delta t}$ and a convergent quadrature formula with weights and nodes $\{(\omega_j,\tau_j)\}_{j=0}^{N_{\Delta t}}$ as in Definition~\ref{def:quadrature}.
Then, for each $\varepsilon>0$, there exist:
\begin{itemize}
\item a decoder $\NN{dec}:\R^d\times\R^n\to\R$ having at most $O(\varepsilon^{-(d+n)})$ active weights;
\item an encoder $\ethnh:\R^{N_h}\to\R^n$ having at most $O(n\varepsilon^{-N_h})$ active weights;
\item a dynamics network $\NN{dyn}:\R\times\R^{d_{\bm\mu}}\to\R^n$ having at most $O(n\varepsilon^{-(d_{\bm\mu}+1)})$ active weights;
\item a threshold $\Delta t^*>0$, depending on $\varepsilon$, on $\NN{dyn}$ and on the quadrature formula;
\end{itemize}
such that
\begin{equation}\label{eq:uat_bound}
\max_{k\in\{1,\dots,N_{\Delta t}\}}\bigl|u_h(\bm x,t_k;\bm\mu)-u^k_{\mathrm{sim}}(\bm x;\bm\mu)\bigr|\leq\varepsilon\qquad\forall\bm x\in\overline\Omega,\ \forall\bm\mu\in\mP,\ \forall\Delta t=\tfrac{T-t_0}{N_{\Delta t}}\leq\Delta t^*,
\end{equation}
where
\begin{equation*}
u^k_{\mathrm{sim}}(\bm x;\bm\mu)=\NN{dec}\bigl(\bm x\,\big|\,\bm s^k(\bm\mu)\bigr),\qquad
\bm s^k(\bm\mu)=\ethnh(\bm u_{0,h}(\bm\mu))+\Delta t\sum_{j=0}^{k}\omega_j\NN{dyn}(\tau_j;\bm\mu).
\end{equation*}
\end{theorem}

The core principle underlying the proof of Theorem~\ref{theo:uat} (the proof of which is in Appendix~\ref{appendix:proofs}) is that, provided the perfect embedding assumptions hold, the problem can be recast to leverage classical \glspl{uat}.
Specifically, the meshless and mesh-based target maps, respectively the \gls{nef} decoder $\Psi: \overline\Omega\times \R^n \to \R$  and the discrete spatial state mapping $\mathcal{D}^{N_h}:\R^n\to \R^{N_h}$, can be formulated as functions from $\R^k$ to $\R^j$ (with $k=n+d$ and $j=1$ in the meshless context, and $k=n$ and $j=N_h$ in the standard mesh-based case).

This dimensional shift introduces a trade-off in the asymptotic scaling of the network bounds; indeed, as highlighted in Theorem~\ref{theo:uat}, the meshless decoder requires at most $O(\varepsilon^{-(d+n)})$ active weights.
While this exhibits a worse asymptotic dependence on the tolerance $\varepsilon$ (due to the exponent $d+n$) compared to the $O(N_h\varepsilon^{-n})$ bound of its mesh-based counterpart, the bound for meshless models does not depend on $N_h$, which can be large, while $n$ is kept small in practice.

Furthermore, for an extended discussion detailing the design choices and the specific functional signature of the dynamics network $\NN{dyn}$, we refer interested readers to~\cite{tomada_latent_2026}.
Finally, complementary to our analysis, the approximation error can also be characterized through the time-continuity properties of the model, as originally proposed in~\cite{farenga_latent_2025}.

\begin{remark}[UAT for the $L^p(\Omega)$ norm]\label{rmk:l2_uat}
Notice that, under the hypotheses of Theorem~\ref{theo:uat}, it is immediate to recover a bound in the $L^p(\Omega)$ norm for all $p\in[1,\infty]$.
Indeed, also for the case $p<\infty$, to obtain an $L^p(\Omega)$ error at most $\varepsilon$, it suffices to define\footnote{We assume that all the involved functions, such as the \glspl{nn} and the discrete solution $u_h$, are bounded on $\Omega$.} 
\begin{equation*}
  f(\bm x) =  \sup_{k\in\{1,\dots,N_{\Delta t}\}}\|u_h(\bm x, t_k;\cdot)-u_{\mathrm{sim}}^k(\bm x;\cdot)\|_{L^{\infty}(\mP)}, \qquad \text{and} \qquad \tilde{\varepsilon}=|\Omega|^{-\frac{1}{p}}\varepsilon 
\end{equation*}
so that, applying Theorem~\ref{theo:uat} with tolerance $\tilde{\varepsilon}$, the bound easily follows from
\begin{equation*}
    \|f \|_{L^p(\Omega)}\leq{|\Omega|^{\frac{1}{p}}}\|f\|_{L^\infty(\Omega)}\leq |\Omega|^{\frac{1}{p}}\tilde{\varepsilon}=\varepsilon.
\end{equation*}
\end{remark}

\begin{remark}[Off-grid error bounds and super-resolution]\label{rmk:off_grid}
The continuous nature of meshless decoders allows for direct evaluation at any spatial coordinate $\bm x \in \Omega$, enabling zero-shot super-resolution.
However, since these models are inherently trained on discrete data, direct error control is restricted to the grid points, in contrast to the continuous bounds established in Equation~\eqref{eq:uat_bound}.
To study the off-grid error, assume that the model is trained on a spatial grid $\mathcal M_h$ with $N_h$ nodes $\{\bm x_i\}_{i=1}^{N_h}$ and denote by $h_{\max}=\sup_{\bm y\in\Omega}\min_{i\in\{1,\dots,N_h\}}\|\bm y-\bm x_i\|_2$ the maximum distance of any point of $\Omega$ to its closest grid point.
Then, under suitable Lipschitz-continuity assumptions detailed in Appendix~\ref{appendix:proofs}, the pointwise error at any arbitrary coordinate $\bm x\in\Omega$ can be bounded as
\begin{equation*}
    |u_h(\bm x,t;\bm\mu) - u_{\mathrm{sim}}(\bm x,t;\bm\mu)| \leq \varepsilon_{\mathrm{grid}} + C h_{\max},
\end{equation*}
where $C\geq0$ is a constant independent of $h_{\max}$ and $\varepsilon_\text{grid}$ denotes the maximum error at the grid points. 
This bound provides a rigorous theoretical foundation for evaluating meshless methods on arbitrary spatial discretizations unseen during training.
\end{remark}

\subsection{Stability of different types of conditioning}\label{subsec:stability_conditioning}
We now turn to the choice of the conditioning type in the setting of \glspl{ldm}.
Indeed, we emphasize that backpropagating through fully integrated trajectories imposes a significant computational bottleneck, and thus the use of \gls{film} alleviates the computational burden when compared to pure conditioning (i.e., hypernetworks).

Beyond computational constraints, \glspl{ldm} compute temporal evolution recursively via time-stepping, so that any small integration error introduced at a given time step accumulates over the trajectory.
As a result, the Lipschitz constant of the decoder needs to be kept small in order to avoid error amplification.
We now observe a fundamental trade-off: highly expressive conditioning mechanisms can potentially degrade the stability of the output with respect to perturbations in the latent state.

Assuming for mathematical simplicity that the decoder consists of a single layer, under the assumptions of Proposition~\ref{prop:lipschitz}, Equation~\eqref{eq:stability} translates in this context to the existence of constants $C_1,C_2\geq 0$ such that
\begin{equation}\label{eq:pointwise_lipschitz_conditioning}
        \|\NN{dec}(\bm x\,|\,\bm s+\deltas)-\NN{dec}(\bm x\,|\,\bm s)\|_p\leq (C_1\|\bm x\|_p+C_2)\|\deltas\|_p,\quad \forall \deltas\in\R^n,
\end{equation}
where $\deltas$ represents a perturbation of the latent state at time $t$.
Furthermore, this bound naturally extends to the $L^p(\Omega)$ norm:
\begin{equation}\label{eq:lipschitz_conditioning}
    \|\NN{dec}(\,\cdot\,\,|\,\bm s+\deltas)-\NN{dec}(\,\cdot\,|\,\bm s)\|_{L^p(\Omega)}\leq |\Omega|^{\frac{1}{p}}(C_1r+C_2)\|\deltas\|_p,\quad \forall \deltas\in\R^n,
\end{equation}
in which $r = \sup_{\bm x\in\Omega}\|\bm x\|_p<+\infty$ is the radius of the smallest ball centered at $0$ that contains $\Omega$.

In the previous expressions, $C_1$ (which is always equal to $0$ for weak conditioning) can become large for strong and pure conditioning, making the network sensitive to small variations in the latent state.
Moreover, when the latter conditioning mechanisms are used, the bound on the magnitude of the propagation grows affinely with the norm of the query point $\bm x\in\Omega$, although this factor could be reduced by means of a suitable scaling of the coordinates given the boundedness of $\Omega$.

In many \gls{mor} and operator learning tasks, the conditioning variables are known a priori, and are thus not affected by perturbations.
However, in \glspl{ldm}, integration errors in the latent state evolution affect the temporal extrapolation even if the underlying dynamics is identified correctly.
As a result, while less powerful, weak conditioning could be preferred in some situations to avoid error amplification during decoding.
Additionally, the complementary strategy of conditioning multiple layers can help in recovering expressivity, but it also amplifies the bound on the global Lipschitz constant in Equation~\eqref{eq:lipschitz_conditioning}.

Stronger forms of conditioning can still be employed~\cite{wan_evolve_2023}.
For instance, \gls{nif}~\cite{pan_neural_2023} bypasses time-stepping altogether through pure conditioning, mapping time directly to a latent state via a hypernetwork.
While this regression-based approach eliminates compounding integration errors, long-term extrapolation instead depends on how well the learned mapping captures the temporal evolution of the system.
Thus, avoiding error accumulation does not by itself guarantee temporal consistency or stable long-term extrapolation.

Ultimately, this theoretical result suggests a trade-off between decoder expressivity and error propagation.
In latent integration settings, the conditioning strategy must carefully balance representational capacity with robustness to integration errors.
This analysis provides insights into the architectural choice in \glspl{ldnet} to rely on a single weakly conditioned layer whenever the base decoder is sufficiently expressive.
We investigate this trade-off empirically in Section~\ref{sec:results}, showing how strong conditioning increases sensitivity to latent perturbations during extrapolation, and assessing whether this translates to a loss in predictive performance.

\subsection{Stability in the continuous limit}\label{subsec:stability_ML_MB}
Beyond the stability implications of the different conditioning techniques, we are interested in the theoretical comparison of such properties between meshless and mesh-based architectures.
In particular, for both model families, we bound the Lipschitz constant of the decoder with respect to the latent state, and discuss the behavior of the continuous limit $N_h\to\infty$ for highly refined meshes.

Importantly, we consider the $L^2(\Omega)$ setting for the continuous solution; indeed, the induced norm can also be utilized during training via Monte Carlo integration or more sophisticated quadrature rules~\cite{berner_principled_2026}.
Our goal is to demonstrate that, for $N_h\gg1$, the underlying problem may lose its $L^2(\Omega)$-stability.
In the following discussion, discrete spaces are endowed with the Euclidean norm and the corresponding induced matrix spectral norm, given by $\left\|\bm A\right\|_2=\sqrt{\rho\left(\bm A^T\bm A\right)}$, $\forall\bm A\in\R^{d_\mathrm{out}\times d_\mathrm{in}}$.
To clarify these claims, we assume for simplicity that $d_{\bm u}=1$ (scalar functions) and that our \glspl{nn} consist of a single linear layer followed by an activation function.
We restrict our meshless analysis to weak conditioning for ease of computation.

Specifically, we consider a family of meshes $\{\mathcal{M}_h\}_{h>0}$ associated with a discretization parameter $h>0$ and denote by $N_h$ the corresponding number of nodes.
For each $N_h\geq1$, we consider a different mesh-based decoder $\decnh:\R^n\to\R^{N_h}$ and we denote by $\bm W^{N_h}_{\bm\theta}\in\R^{N_h\times n}$ its weight matrix.
Since a meshless decoder does not depend on the mesh, we consider a single model for all $h>0$ and denote its bias by $\bm b(\bm s(t;\bm\mu))=\bm W_{\bm s}\bm s(t;\bm\mu)$ for some matrix $\bm W_{\bm s}$.
If a separate meshless model is trained for each mesh, one obtains a family $\{\bm W_{\bm s}^{(h)}\}_{h>0}$ and Equation~\eqref{eq:lipsch_const_ml} in the next proposition holds with $\bm W_{\bm s}^{(h)}$.

\begin{proposition}[$L^2(\Omega)$ Lipschitz bounds for the decoder mappings]\label{prop:stability_L2}
Let us assume that:
\begin{itemize}
    \item $\{\mathcal{M}_h\}_{h>0}$ is a shape-regular and quasi-uniform family of meshes, for which we denote by $\{\bm x_i\}_{i=1}^{N_h}$ the nodes of each mesh $\mathcal{M}_h$, for $h>0$;
    \item $\sigma:\R\to\R$ is a Lipschitz-continuous, nonlinear function with Lipschitz constant $C_\sigma$;
    \item $\NN{dec}:\R^d\times\R^n\to\R$ is a one-layer weakly conditioned \gls{nef};
    \item for each $h>0$, $\mathcal{D}_{\bm\theta}^{N_h}:\R^n\to\R^{N_h}$ is a mesh-based, one-layer \gls{nn}, for a grid with $N_h$ nodes.
\end{itemize} 
Denote by $\tilde{\mathcal{D}}_{\bm\theta}^{N_h}:\Omega\times\R^n\to\R$ the continuous \gls{fe} interpolant of $\mathcal{D}^{N_h}_{\bm\theta}$ on the mesh nodes, defined as
\begin{equation*}
    \tilde{\mathcal{D}}^{N_h}_{\bm\theta}(\bm x, \bm s) = \sum_{i=1}^{N_h}\mathcal{D}^{N_h}_{\bm\theta}(\bm s)_i\varphi_i(\bm x),
\end{equation*}
where $\{\varphi_i\}_{i=1}^{N_h}$ is the \gls{fe} basis of degree $1$ associated with the mesh $\mathcal{M}_h$.
Then:
\begin{enumerate}
    \item\label{item:stability_1} For meshless models, it holds that
    \begin{equation}\label{eq:lipsch_const_ml}
     \|\NN{dec}(\,\cdot\,|\,\bm s+\deltas)-\NN{dec}(\,\cdot\,|\,\bm s)\|_{L^2(\Omega)}\leq\sqrt{|\Omega|}C_\sigma\|\bm W_{\bm s}\|_2\|\deltas\|_2,\quad\forall\deltas\in\R^n
\end{equation}
and denoting by $L_{\NN{dec}}$ the Lipschitz constant of the operator $\R^n\ni \bm s\mapsto \NN{dec}(\,\cdot\,|\,\bm s)\in L^2(\Omega)$, we have that
\begin{equation*}
    L_{\NN{dec}}\leq \sqrt{|\Omega|}C_\sigma\|\bm W_{\bm s}\|_2.
\end{equation*}
\item\label{item:stability_2} For mesh-based models, there exists a constant $C_M>0$ independent of $h$ such that
\begin{equation}\label{eq:lipsch_const_mb}
 \left \|\tilde{\mathcal{D}}^{N_h}_{\bm\theta}(\cdot\,,\,\bm s)-\tilde{\mathcal{D}}^{N_h}_{\bm\theta}(\cdot\,,\,\bm s+\deltas)\right\|_{L^2(\Omega)}\leq h^{\frac{d}{2}}C_MC_\sigma \left\|\bm W^{N_h}_{\bm\theta}\right\|_2\|\deltas\|_2,\quad\forall\deltas\in\R^n.
\end{equation}
Consequently, denoting by $L_{\tilde{\mathcal D}^{N_h}_{\bm\theta}}$ the Lipschitz constant of the operator $\R^n\ni \bm s\mapsto \tilde{\mathcal D}_{\bm\theta}^{N_h}(\,\cdot\,,\,\bm s)\in L^2(\Omega)$, we have that
\begin{equation*}
L_{\tilde{\mathcal D}^{N_h}_{\bm\theta}}\le h^{\frac d2}C_MC_\sigma \left\|\bm W^{N_h}_{\bm\theta}\right\|_2=O\left(N_h^{-\frac12}\left\|\bm W^{N_h}_{\bm\theta}\right\|_2\right).
\end{equation*}
\end{enumerate}

\end{proposition}
The previous result, proved in Appendix~\ref{appendix:stability}, establishes a profound architectural distinction between the two considered families of models.
Equation~\eqref{eq:lipsch_const_ml} reveals that if we consider a single meshless model for the entire family of meshes, then $\lipdec$ is constant.
Even in the case of a separate mesh-agnostic model for each value of $N_h$, the bound is independent of $N_h$ whenever $\sup_h\|\bm W_{\bm s}^{(h)}\|_2<\infty$ (verified empirically in Subsection~\ref{subsec:burgers}).
This fact helps to isolate the amplification of latent space errors from the spatial resolution.

In contrast, for mesh-based models, the bound depends on $N_h$ both through the spatial scaling factor $N_h^{-\frac{1}{2}}$ and the size of the weight matrix $\bm W^{N_h}_{\bm\theta}\in\R^{N_h\times n}$, which in turn can influence its spectral norm.
From classical results of random matrix theory~\cite{vershynin_high-dimensional_2018}, if the entries of $\bm W^{N_h}_{\bm\theta}$ are independent, zero-mean, and sub-Gaussian (a family that naturally includes standard Gaussian and uniform initialization schemes), then $\left\|\bm W^{N_h}_{\bm\theta}\right\|_2 = O(\sqrt{N_h})$ with high probability for fixed $n$ as $N_h \to \infty$.
Under these conditions, such as when using He initialization~\cite{he_delving_2015}, the bound on the Lipschitz constant behaves nicely as
\begin{equation}\label{eq:kaiming}
L_{\tilde{\mathcal{D}}^{N_h}_{\bm\theta}}=O\left(\frac{\left\|\bm W^{N_h}_{\bm\theta}\right\|_2}{\sqrt{N_h}}\right)=O\left(\frac{\sqrt{N_h}}{\sqrt{N_h}}\right)= O(1).
\end{equation}
This ensures that small perturbations are not asymptotically amplified in the large mesh limit, at least at the beginning of training, when the weights satisfy our hypothesis.

However, other common initializations, such as Glorot~\cite{glorot_understanding_2010}, scale the variance inversely proportional to the output dimension $N_h$.
In this case, the matrix norm becomes approximately constant with respect to $N_h$, meaning that $\left\|\bm W^{N_h}_{\bm\theta}\right\|_2 = O(1)$.
Therefore, we have
\begin{equation}\label{eq:xavier}
L_{\tilde{\mathcal{D}}^{N_h}_{\bm\theta}}=O\left(\frac{\left\|\bm W_{\bm\theta}^{N_h}\right\|_2}{\sqrt{N_h}}\right)=O\left(\frac{1}{\sqrt{N_h}}\right),
\end{equation}
which in turn implies that $L_{\tilde{\mathcal{D}}^{N_h}_{\bm\theta}}\xrightarrow{N_h \to \infty} 0$.
Recalling that only constant functions have the Lipschitz constant equal to $0$, the previous observation means that, in the continuous limit, the initialized decoder becomes asymptotically independent of the latent state, and its output will be constant in time and parameter in the $L^2(\Omega)$ sense.
Under this set of hypotheses, if the model is trained with a loss representing the average error across trajectories, the continuous limit of a mesh-based decoder would then simply learn the mean field across simulations.

It is important to note, however, that this predicted mean-field regime primarily characterizes the network at initialization.
While scaling weights to ease optimization~\cite{glorot_understanding_2010} initially drives the decoder toward this state, the assumptions regarding $\bm W^{N_h}_{\bm\theta}$ typically break down as training progresses, as reflected in our results (Subsection~\ref{subsec:burgers}). 
More generally, empirical evidence shows that the weights of well-trained networks frequently converge to heavy-tailed distributions~\cite{martin_implicit_2021}, which could even cause the maximum singular value to grow faster than $O(\sqrt{N_h})$.

As we formalize next, a small decoder Lipschitz constant forces the latent trajectories of a model to be large, even in the case of well-trained networks.

\begin{proposition}[Norm of the latent trajectories]\label{corollary:magnitude}
For a given $N_h\in\mathbb{N}$, denote by $C$ the quantity
\begin{equation*}%\label{eq:equation_C}
   C= \sup_{\substack{{t_1,t_2\in I}\\ \bm\mu\in\mP}}\left\| u_h(\,\cdot\,,t_1;\bm\mu)- u_h(\,\cdot\,,t_2;\bm\mu)\right\|_{L^2(\Omega)},
\end{equation*}
and let $\varepsilon\in\R$ such that $0<\varepsilon<C$.
Then, if $(\ethnh,\,\NN{dyn}, \NN{dec})$ are such that
\begin{equation}\label{eq:condition_corollary}
    \sup_{(t,\bm\mu)\in I\times \mP} \|u_h(\,\cdot\,,t;\bm\mu)-\NN{dec}(\,\cdot\,|\, \bm s(t;\bm\mu))\|_{L^2(\Omega)}\leq\frac{\varepsilon}{2},
\end{equation}
it holds that
\begin{equation}\label{eq:latent_magnitude_2}
    S=\sup_{(t,\bm\mu)\in I\times\mP}\|\bm s(t;\bm\mu)\|_2 \geq \frac{C-\varepsilon}{2L_{\NN{dec}}}.
\end{equation}
Furthermore, if $\tilde{\mathcal D}_{\bm\theta}^{N_h}$ satisfies an analogous error bound, then
\begin{equation}\label{eq:latent_magnitude_3}
    S_{N_h} \geq \frac{C-\varepsilon}{2L_{\tilde{\mathcal{D}}_{\bm\theta}^{N_h}}}.
\end{equation}
\end{proposition}
The proof of Proposition~\ref{corollary:magnitude} can be found in Appendix~\ref{appendix:stability}, and although we restrict the discussion to the $L^2(\Omega)$ norm, analogous claims hold for $p\in[1,\infty]$.
The bound in Equation~\eqref{eq:condition_corollary} is satisfied, for instance, by the networks constructed in the proof of Theorem~\ref{theo:uat} when discrete time stepping is substituted by continuous time integration (Remark~\ref{rmk:uat_continuous}).
Equations~\eqref{eq:latent_magnitude_2} and~\eqref{eq:latent_magnitude_3} serve as an analytical diagnostic tool for trained architectures.
They formalize the principle that if a well-trained decoder is characterized by a small Lipschitz constant $L_D$, with $D\in\{\NN{dec}, \tilde{\mathcal{D}}_{\bm\theta}^{N_h}\}$, then the latent trajectories cannot collapse; rather, they must expand across the latent space, i.e., $S$ needs to be large to match the spread of the high-fidelity data.
Conversely, a large $L_{D}$ allows for small latent trajectories, but could potentially amplify errors.

Additionally, for meshless methodologies, imposing weight regularization as a penalty term directly impacts the resulting Lipschitz constant in view of Proposition~\ref{prop:stability_L2}, as well as the magnitude $S$ of the latent trajectories.
Conversely, in the mesh-based case, the Lipschitz constant might also depend on $N_h$. 

The constant $C$ represents the maximum distance between snapshots over time and parameters.
While $S$ can be evaluated over full trajectories to account for test-time behavior, to obtain an estimate $C_\text{train}$ of $C$, we can approximate it by the maximum pairwise $L^2(\Omega)$ distance between snapshots within the empirical training dataset, so that it holds $C_\text{train}\leq C$, and the bound in Equation~\eqref{eq:latent_magnitude_2} remains valid with $C$ replaced by $C_\text{train}$.

Overall, our findings show that the limit for $N_h\to\infty$ of mesh-based models may exhibit instabilities in the $L^2(\Omega)$ sense if not properly trained, whereas meshless models do not suffer from this limitation a priori.

\subsection{Comparison with existing architectures}\label{subsec:spectral_bias}
In this subsection, we compare the primary methodological properties of the aforementioned architectures with other state-of-the-art models in the literature (see Table~\ref{tab:comparison_architectures}).
In particular, meshless \glspl{ldm} allow evaluation at arbitrary points in space and time (hence, super- and sub-resolution, as well as extrapolation), keep a low number of parameters (unlike \glspl{no}) and independent\footnote{In this discussion, we assume that no encoder is used, as in the case of \glspl{ldnet}.} of $N_h$ (unlike traditional \glspl{ae}).
Similarly to \glspl{no}, \glspl{ldm} can be trained on multifidelity data, and the map from $\bm s(t;\bm\mu)$ to $\bm u_\mathrm{sim}(\,\cdot\,,t;\bm\mu)$ is a vector-to-function map, consistent with some prominent existing \gls{no} architectures~\cite{huang_operator_2025}, such as the autoencoder neural operator~\cite{berner_principled_2026}, or other \glspl{no} that rely on latent time-stepping~\cite{li_deep_2026} and more generally on continuous dynamics~\cite{abueidda_time_2026}.

\begin{table}[tbp]
\caption{Comparison between various existing architectures. The symbol $\square$ indicates properties that depend on architectural choices and implementation.}
\label{tab:comparison_architectures}
\centering
\begin{tabular}{c|c|c|c|c|c}
         & Causality & Time continuity & Multifidelity & Off-grid evaluation & \# Params \\ 
\midrule
        Neural operator        & $\square$ & $\square$ & \checkmark & \checkmark & $O(1)$ \\ 
        Autoregressive    & \checkmark & \ding{55} & $\square$ & $\square$  & $\square$ \\ 
        Vanilla CNeF  & \ding{55}& \ding{55} & \checkmark & \checkmark & $O(1)$ \\ 
        Mesh-based LDM       & \checkmark & \checkmark & \ding{55} & \ding{55} & $O(N_h)$ \\ 
        Meshless LDM            & \checkmark & \checkmark & \checkmark  & \checkmark & $O(1)$ \\ 
\bottomrule
\end{tabular}
\end{table}

Despite these advantages, the reliance on \glspl{cnef} introduces specific challenges, such as the well-known \emph{spectral bias}, the inherent tendency of coordinate-based \glspl{nn} toward learning functions with low spatial frequencies~\cite{rahaman_spectral_2019,cao_towards_2021}.
Although mitigation strategies exist, such as the use of periodic activation functions~\cite{sitzmann_implicit_2020} or suitable coordinate embeddings~\cite{tancik_fourier_2020}, capturing high-frequency details remains an intrinsic training difficulty for \glspl{nef}. 
This challenge can also be viewed through the lens of architectural locality.
In a mesh-based decoder, each \gls{dof} corresponds to a separate output coordinate, so that, for a fully connected output layer, an entry in the $k$-th row of the weight matrix only affects the prediction at the $k$-th \gls{dof}.
In a meshless decoder, instead, outputs at all spatial locations are produced by the same set of weights, and updating any weight alters the reconstructed field over the entire domain.
This localized influence may give mesh-based models more granular control over small-scale, high-frequency spatial details.

Furthermore, to reconstruct the values of the solution at $N_h$ nodes, a mesh-based model only requires a single forward pass, whereas a meshless model needs to perform $N_h$ evaluations, potentially making the training slower and increasing peak memory consumption in spite of the much smaller number of trainable parameters.
Nevertheless, it is possible to find suitable remedies to alleviate the computational burden, such as gradient checkpointing~\cite{chen_training_2016}, significantly reducing the memory overhead.
Additionally, the training speed can be improved by uniformly subsampling the spatial coordinates at each epoch, a strategy that circumvents the need to evaluate the full grid and naturally opens the door to active learning or adaptive batching paradigms~\cite{regazzoni_learning_2024}.

Finally, as mesh-based architectures do not explicitly process spatial coordinates, they enable the use of preliminary linear dimensionality reduction techniques such as \gls{pod}.
By predicting only a low-dimensional set of $N_\mathrm{POD}$ coefficients instead of full high-dimensional fields, this approach significantly reduces the \gls{nn} parameter count~\cite{fresca_pod-dl-rom_2022}.
In spite of its computational efficiency, this two-stage strategy (i) is limited to the specific spatial discretization used during the \gls{pod} phase, (ii) can struggle to accurately represent phenomena characterized by a slowly decaying Kolmogorov $n$-width~\cite{khamlich_optimal_2025} due to its linear nature, and (iii) cannot continuously evaluate the solution at off-grid spatial coordinates.

%% file: sections/results.tex
\section{Numerical results}\label{sec:results}
In this section, we evaluate and compare the performance of meshless \glspl{ldm} with their mesh-based counterparts across three computational mechanics benchmarks.
The test cases under investigation---the Burgers' equation, a Hopf bifurcation in the Navier--Stokes equations, and a linear advection problem---are specifically chosen to assess the methodology's accuracy, model complexity, and computational efficiency.
Respectively, we investigate: (i) the model's robustness to steep gradients and shocks, together with an empirical validation of our theoretical bounds, (ii) the impact of a strongly conditioned decoder in a more complex scenario, together with a physically interpretable visualization of spectral bias, and (iii) the effect of the number and type of conditioning layers in the decoder on accuracy and sensitivity to latent perturbations.
Following the discussion in Section~\ref{sec:methodology}, we do not consider here the pure conditioning case (i.e., hypernetworks), as the conditioning variable is the latent state $\bm s\in\R^n$ (with $n$ up to $25$), and its cost, scaling as $nN^2$ (Appendix~\ref{appendix:formalism}), would be considerably higher than the $N^2$ of the adopted weak and strong conditioning.

Our findings demonstrate that meshless models achieve comparable or better accuracy while requiring up to $45$ times fewer trainable parameters.
Furthermore, the results are consistent with our stability analysis: strong conditioning increases the sensitivity to latent perturbations in extrapolation; this may degrade performance, as in the Navier--Stokes benchmark, or it may yield more accurate models, as in the advection benchmark.
Finally, we empirically confirm that, unlike the mesh-based models considered, mesh-agnostic methods exhibit an error, a latent trajectory magnitude, and a decoder Lipschitz constant that remain stable across varying mesh resolutions.

All models, including meshless \glspl{ldm}, \gls{ldfnn}, and \gls{ldgcn}, were implemented using PyTorch~\cite{paszke_pytorch_2019} and PyTorch Geometric~\cite{fey_pyg_2025}.
Architectural hyperparameters and training procedures are comprehensively detailed in Appendix~\ref{appendix:architectures}.

As concerns the error metrics we denote by $\bm{u}_h, \bm{u}_\text{sim}:I_h\times\mathbb{P}_h\to\mathbb{R}^{N_h\times d_{\bm{u}}}$ the full-order and reduced solutions evaluated at the \glspl{dof} for $(t,\bm{\mu})\in I_h\times\mathbb{P}_h$, and define the relative error field $\bm{e}_\text{rel}:I_h\times\mathbb{P}_h\to\mathbb{R}^{N_h\times d_{\bm{u}}}$ as 
\begin{equation*}%\label{eq:rel_ell}
    \bm{e}_\text{rel}(t; \bm{\mu})= \frac{\bm{u}_h(t; \bm{\mu})-\bm{u}_\text{sim}(t; \bm{\mu})}{\|\bm{u}_h(t; \bm{\mu})\|_{2}}, \qquad \text{and} \quad
    \varepsilon_\text{rel}(t; \bm{\mu})=\|\bm{e}_\text{rel}(t; \bm{\mu})\|_{2},\quad\forall \,(t,\bm{\mu})\in I_h\times\mathbb{P}_h,
\end{equation*}
indicating with $\varepsilon^\text{train}_\text{mean}$/$\varepsilon^\text{test}_\text{mean}$ and $\varepsilon^\text{train}_\text{max}$/$\varepsilon^\text{test}_\text{max}$ the mean and maximum relative errors over the training and testing datasets.

\subsection{Burgers' equation with changing initial conditions}\label{subsec:burgers}
The first benchmark we consider is the one-dimensional viscous Burgers' equation, a prototypical nonlinear convection--diffusion problem widely used as a simplified model for fluid dynamics, characterized by shock formations possibly leading to meshless \glspl{ldm} to suffer from challenging spectral bias. 
Our primary goals here are to demonstrate the applicability of the methodology without an explicit encoder, even when dealing with varying initial conditions, and to empirically validate the bounds discussed in Section~\ref{sec:analysis}.

\subsubsection{Problem definition and data collection}
Let $\Omega=(0,1)\subset\mathbb{R}$ and consider the following problem:
\begin{equation*}
\begin{cases}
\dfrac{\partial u}{\partial t}
+u\dfrac{\partial u}{\partial x}
-\nu\dfrac{\partial^2u}{\partial x^2}=0,
& \text{in } \Omega\times(0,T],\\
u=0,
& \text{on } \partial\Omega\times(0,T],\\
u=u_0,
& \text{on } \Omega\times\{0\},
\end{cases}
\end{equation*}
where $u:\Omega\times[0,T]\to\mathbb{R}$ denotes the solution field, $\nu>0$ is the viscosity coefficient, and $T=1$.
To investigate the performance of meshless, encoder-free \glspl{ldm} and their mesh-based counterparts, we treat both the viscosity and the initial condition as varying parameters.
The parameter space for the viscosity is defined as $\nu\in\mP=[0.002, 0.02]$, while the initial condition is prescribed as a Gaussian profile $u_0(x)=\exp\left(-50(x-c)^2\right)$, with the center $c\in[0.3, 0.7]$.

Since we aim at studying the stability of the models and the behavior of the Lipschitz constants in connection with the mesh resolution, we evaluate the methodologies across varying spatial discretizations. 
We select $5$ $\log$-equispaced values for the number of nodes $N_h\in\{100,211,447, 946,2000\}$, and for each discretization, we execute the simulations using $5$ different random seeds, combining $10$ equispaced values of $c$ and $5$ equispaced values of $\nu$, resulting in a total of $50$ distinct parametric configurations.
For the spatial discretization, we employ the \gls{fe} method, defining the space $\mathcal P^1$ on a uniform mesh of $N_h$ nodes.
Time integration is performed via the implicit Euler method with a uniform time step $\Delta t=10^{-2}$, and the nonlinear algebraic system arising at each time step is solved using Newton's method~\cite{quarteroni_numerical_2007}.
The dataset is split by randomly selecting $80\%$ of the trajectories for training, and we restrict the training loss to the first $80\%$ of the system's temporal evolution to explicitly test the capacity of the models for temporal extrapolation.
Figure~\ref{fig:fields_burgers} displays the solution fields corresponding to two different initial conditions and values of $\nu$.

\begin{figure}[tbp]
    \centering
    \begin{subfigure}[b]{0.3\textwidth}
        \centering
        \includegraphics[width=\textwidth]{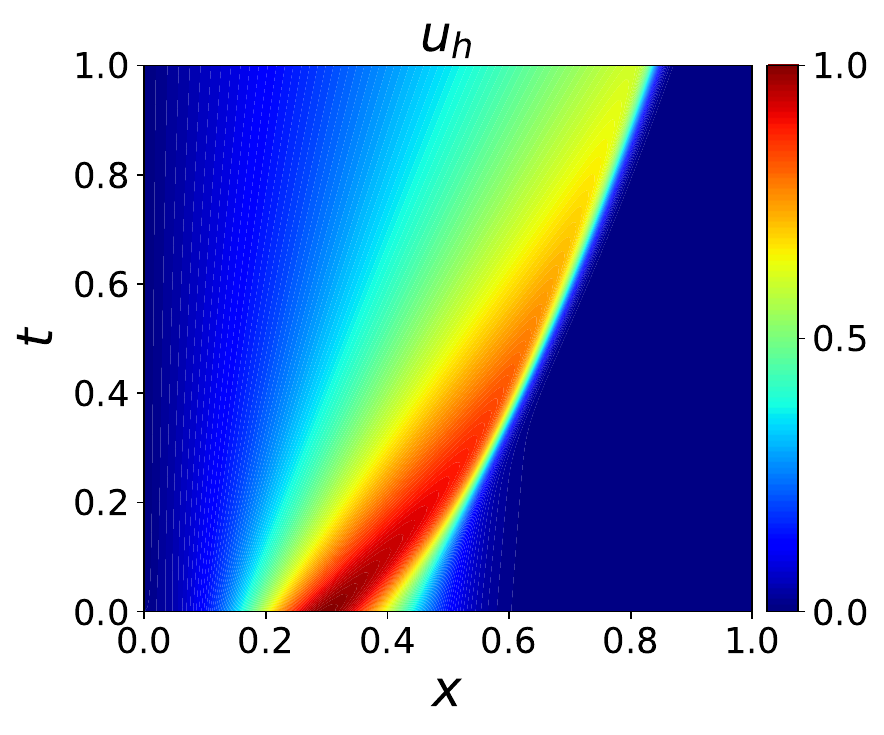}
        \caption{$c=0.3,\,\nu=2\cdot10^{-3}$}
    \end{subfigure}
    \hfill
    \begin{subfigure}[b]{0.3\textwidth}
        \centering
        \includegraphics[width=\textwidth]{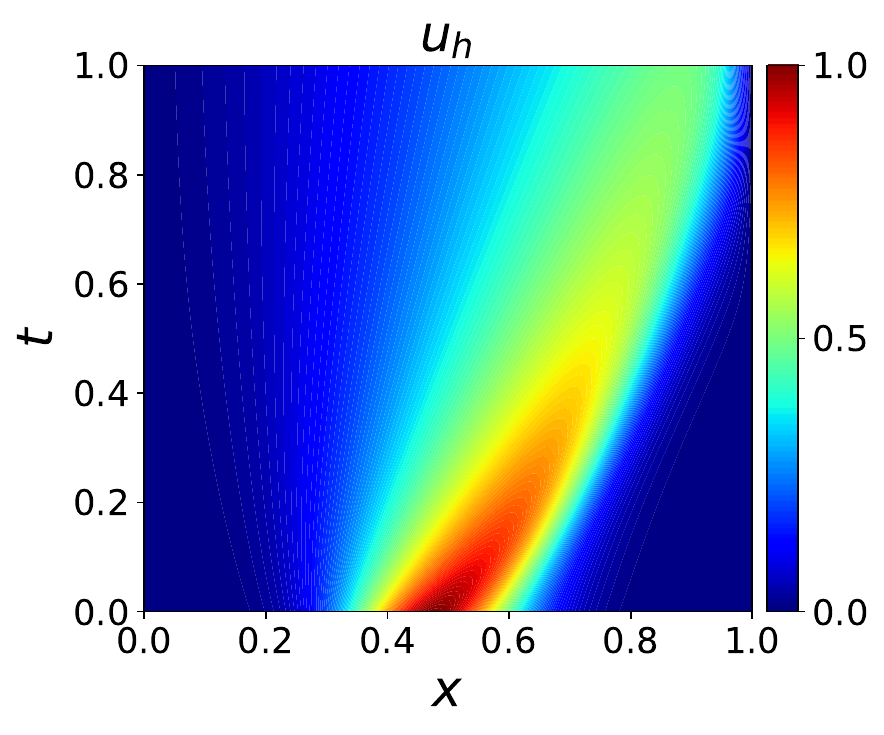}
      \caption{$c=0.48,\,\nu=1.1\cdot10^{-2}$}
    \end{subfigure}
     \hfill
    \begin{subfigure}[b]{0.3\textwidth}
        \centering
        \includegraphics[width=\textwidth]{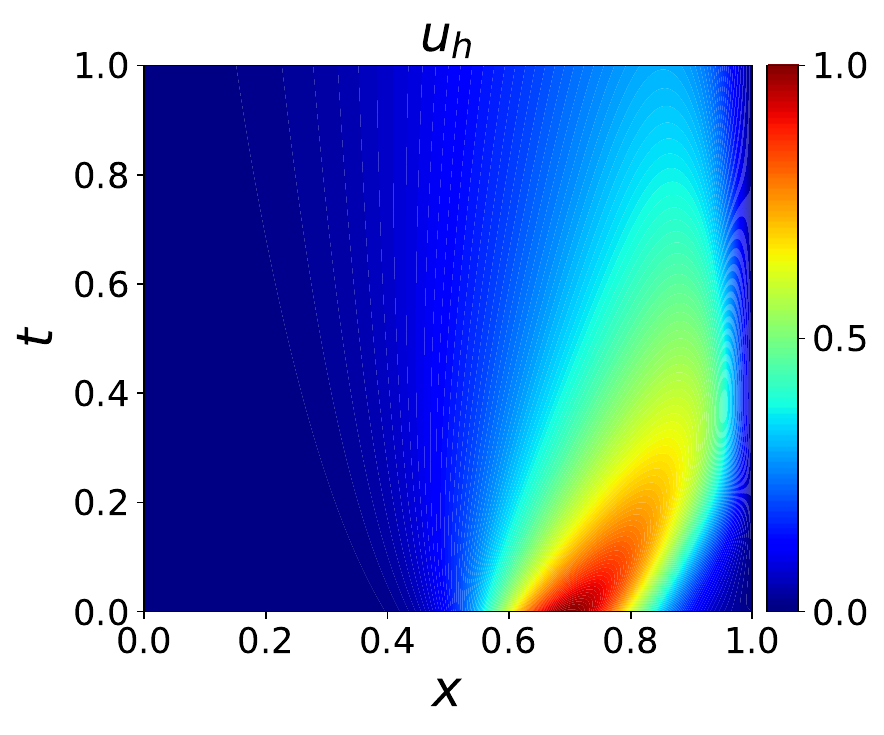}
      \caption{$c=0.7,\,\nu=2\cdot10^{-2}$}
    \end{subfigure}
\caption{Burgers solution fields for different parametric samples.}
    \label{fig:fields_burgers}
\end{figure}

\subsubsection{Model performance and bound investigation}
We evaluate three distinct decoder architectures: a \gls{ldnet} model and two \glspl{ldfnn} with Glorot and He initializations, respectively, to empirically observe the post-training bounds in Equations~\eqref{eq:kaiming} and~\eqref{eq:xavier}.
Consistent with recent approaches~\cite{koike_sparse_2026}, rather than using a trainable encoder, we compute the \gls{pod} on the set of training initial conditions, and for each trajectory, the first $n$ \gls{pod} coefficients of each full-order initial condition are assigned as the initial latent state.

While the error metrics described at the beginning of the section are standard in machine learning, in this setting we instead report $\varepsilon_{L^2(\Omega)}^\text{test}$, which is computed using a numerical approximation of the continuous $L^2(\Omega)$ norm (see Appendix~\ref{appendix:continuous}) to enable empirical evaluation of the theoretical bounds in Section~\ref{sec:analysis}, which are formulated in terms of continuous norms.

Figure~\ref{fig:burgers_analysis} illustrates the values of $S$, the Lipschitz constants $L_D$ for $D\in\{\decnh, \NN{dec}\}$ (estimated as detailed in Appendix~\ref{appendix:continuous}), $\varepsilon^\text{test}_{L^2(\Omega)}$, and the number of trainable parameters for all considered architectures across changing values of $N_h$. 
We observe that, for the Glorot initialization, the Lipschitz constant follows the asymptotic $O(N_h^{-\frac{1}{2}})$ behavior predicted in Equation~\eqref{eq:xavier} even after training for discretizations up to $N_h=946$, before increasing coherently with the discussion in Subsection~\ref{subsec:stability_ML_MB}.
The corresponding values of $S$ increase with $N_h$ up to $N_h=946$ and then plateau, mirroring the behavior of the Lipschitz constant and meaning that the latent trajectories tend to expand to compensate for its decay, following the condition in~\eqref{eq:latent_magnitude_2}.
Conversely, for He-initialized models, the Lipschitz constant shows a less prominent decrease, as predicted in Equation~\eqref{eq:kaiming}, and the corresponding values of $S$ do not show a monotonic trend related to mesh refinement.
This confirms that the expansion of the latent trajectories observed for the Glorot initialization is driven by the decay of the Lipschitz constant of the decoder.
Both \glspl{ldfnn} models show similar behavior, with $\varepsilon^\text{test}_{L^2(\Omega)}$ fluctuating across different refinement levels, and the number of trainable parameters scales as $O(N_h)$.

In contrast, for meshless models, $S$, $\lipdec$, and $\varepsilon^\text{test}_{L^2(\Omega)}$ are empirically shown to be $O(1)$.
Although a single meshless model could be trained on the whole family of meshes, to ensure a fair comparison with mesh-based approaches, here we train one model per discretization.
Therefore, the only difference across refinement levels is the discrete loss, which is a quadrature approximation of the continuous $L^2(\Omega)$ error~\cite{berner_principled_2026}, as opposed to mesh-based models, in which the number of neurons and structure of the architecture change for each different discretization.
Overall, while the independence of the parameter count from $N_h$ holds by design, the results show that the latent dynamics, the Lipschitz constant of the decoder, and the accuracy of meshless models exhibit limited sensitivity to mesh refinement, whereas the behavior of mesh-based models depends more noticeably on $N_h$.

\begin{figure}[tbp]
    \centering
    \includegraphics[width=\linewidth]{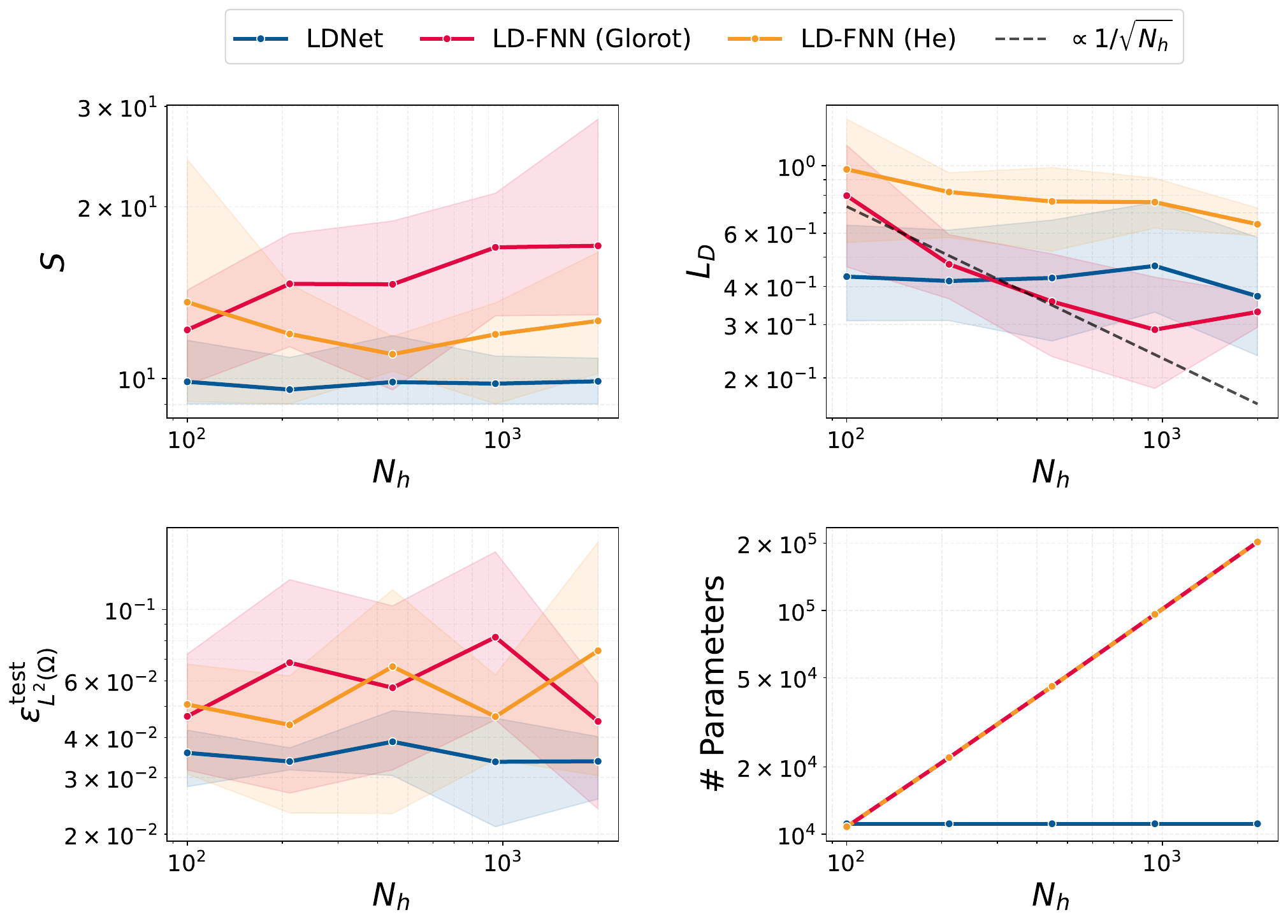}
    \caption{Evolution of $S$ (top left), Lipschitz constants $L_D$ (top right), error $\varepsilon^\text{test}_{L^2(\Omega)}$ (bottom left), and number of trainable parameters (bottom right) for increasing values of $N_h$. Shaded areas denote the interval between the minimum and maximum realization across different seeds.}
    \label{fig:burgers_analysis}
\end{figure}

\subsection{Flow past a cylinder}\label{subsec:navier_stokes}
The second test case investigates the unsteady Navier--Stokes equations governing an incompressible, viscous fluid within a two-dimensional domain containing an obstacle adapted from~\cite{chen_time_2025}.
Here, we aim at both comparing the performance of meshless and mesh-based methodologies, and investigating the effect of strong conditioning in the decoder.

\subsubsection{Problem definition and data collection}
Let us consider the following set of equations:
{\small
\begin{equation*}
\begin{cases}
\dfrac{\partial \bm{u}}{\partial t} -\mu \nabla^2 \bm{u}+ \left(\bm{u} \cdot \nabla\right)\bm{u} +\nabla p =\bm 0 & \text{in } \Omega \times (0, T], \\[4pt]
\nabla \cdot \bm{u} = 0 & \text{in } \Omega \times (0, T],\\[4pt]
\bm u = \bm u_0=\bm 0 & \text{on } \Omega\times\{0\},
\end{cases}
\qquad
\begin{cases}
\bm u=1.5\left[
   \dfrac{4y(0.41-y)}{0.41^2},\, 0\right]^T & \text{on } \Gamma_\text{in}\times (0,T],\\[4pt]
\bm u = \bm 0 & \text{on } \Gamma_\text{w}\times (0,T],\\[4pt]
\mu\dfrac{\partial\bm u}{\partial\bm n}-p\bm n=\bm 0 & \text{on } \Gamma_\text{out}\times (0,T],
\end{cases}
\end{equation*}
}
where $\bm x=(x, y)\in \Omega$ depicted in Figure~\ref{fig:cylinder_geometry_ids}, $\bm u$ and $p$ represent the velocity and pressure fields, respectively, and $\bm n$ is the outward normal vector.
Additionally, $\Gamma_\text{in}$, $\Gamma_\text{out}$ and $\Gamma_\text{w}$ represent the inflow, outflow, and no-slip boundaries, respectively.
As for the parametric range, we set $T=8$ and $\mu\in\mP=[3\cdot10^{-4}, 1.3\cdot10^{-3}]$, corresponding to a Reynolds number $Re\in[76, 330]$ and allowing us to observe the Hopf bifurcation~\cite{chen_time_2025, dijkstra_bifurcation_2023} characterizing the velocity field $\bm u$.

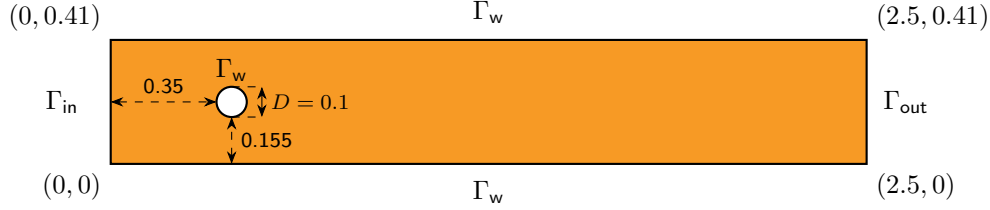
\begin{figure}[tbp]
\centering
\begin{tikzpicture}[scale=4, >=Stealth, font=\sffamily]

    % Fluid domain
    \filldraw[fill=mathlaborange, draw=black, thick] (0,0) rectangle (2.5, 0.41);

    % Domain vertex coordinates
    \node[below left] at (0,0) {$(0,0)$};
    \node[below right] at (2.5,0) {$(2.5,0)$};
    \node[above right] at (2.5,0.41) {$(2.5,0.41)$};
    \node[above left] at (0,0.41) {$(0,0.41)$};

    % Cylinder (obstacle) centered at C_x = 0.4, C_y = 0.205 with r = 0.05
    \filldraw[fill=white, draw=black, thick] (0.4, 0.205) circle (0.05);

    % Cylinder position dimensions
    \draw[<->, dashed, black] (0, 0.205) -- (0.35, 0.205) node[midway, above] {\footnotesize 0.35};
    \draw[<->, dashed, black] (0.4, 0) -- (0.4, 0.155) node[midway, right] {\footnotesize 0.155};
    
    % Cylinder diameter
    \draw[<->, black] (0.5, 0.155) -- (0.5, 0.255) node[midway, right] {\footnotesize $D = 0.1$};
    \draw[dashed, black] (0.4, 0.155) -- (0.5, 0.155);
    \draw[dashed, black] (0.4, 0.255) -- (0.5, 0.255);
     \node[align=center] at (0.4,0.31) {$\Gamma_{\text{w}}$};

    % Boundary labels
    \node[left=0.3cm, align=right] at (0, 0.205) {$\Gamma_{\text{in}}$};
    \node[right=0.1cm, align=left] at (2.5, 0.205) {$\Gamma_{\text{out}}$};
    % align=center to fix the LR mode errors
    \node[above=0.1cm, align=center] at (1.25, 0.41) {$\Gamma_{\text{w}}$};
    \node[below=0.1cm, align=center] at (1.25, 0) {$\Gamma_{\text{w}}$};

\end{tikzpicture}
    \caption{Computational domain in the Navier--Stokes benchmark.}
    \label{fig:cylinder_geometry_ids}
\end{figure}

To perform data collection, we employ the \gls{fe} method, using the Taylor-Hood pair $\mathcal{P}^2/\mathcal P^1$ to satisfy the inf--sup condition~\cite{quarteroni_numerical_2017}.
There are $N_h = 5017$ nodes in the mesh, and we employ the implicit Euler method~\cite{quarteroni_numerical_2007} to perform time-stepping, using a step size equal to $\Delta t=10^{-2}$, for $25$ different trajectories corresponding to equispaced values of $\mu\in\mP$.
To create the training dataset, we subsample the time interval, starting at $t_0=5$, to avoid the initial transient phase, and saving the snapshots at each $5\cdot 10^{-2}$ seconds.
We then only provide the first $70\%$ of the evolution of $70\%$ of randomly selected trajectories.
In this setup, only $740$ out of the $1525$ collected snapshots were seen by the network during training. 
Given the small magnitude of the values of $\mu$, the parameters are standardized to ensure zero mean and unit variance.

\subsubsection{Model performance}
We study the performance of \gls{ldfnn}, \gls{ldgcn}, \gls{ldnet}, and \gls{filmldm} and assess their ability to extrapolate in time and parameter.
In particular, for \gls{filmldm} we consider a decoder in which two layers are strongly conditioned.
Moreover, we consider a second \gls{ldnet} model, denoted as \gls{ldnet}$_{(0)}$, in which Fourier features are not employed (see Appendix~\ref{appendix:architectures}), to show how they help overcome the problem of spectral bias.

Table~\ref{tab:err_navier_stokes} compares the number of trainable parameters, training time, and performance of the different architectures for unseen values of $\mu$ and during extrapolation.
Errors corresponding to times within the training interval are not reported.
Indeed, due to the fixed latent initial condition, the maximum relative error across simulations always occurs at the first snapshot.
Including the training interval in the metrics would therefore mask the maximum error during extrapolation behind the one on the initial condition, which is not representative of how well the dynamics has been learned.

Figure~\ref{fig:errors_NS} shows the error evolution in time, with the shaded areas denoting the range between the minimum and maximum error across the values of $\mu$ for each model.
We first focus on the mesh-based models and on the meshless models employing Fourier features.
As shown in the figure, they all achieve comparable accuracy, with \gls{ldgcn} slightly outperforming the other methods, and \gls{ldnet} guaranteeing a smaller maximum error compared to \gls{ldfnn} despite using approximately $45$ times fewer parameters in the decoder, corresponding to approximately $2\%$ of the decoder parameters of \gls{ldfnn}.
Among these models, the highest relative mean and maximum errors are both attained by \gls{filmldm}, whose decoder is strongly conditioned on the latent state, whereas \gls{ldnet}, with weak conditioning, achieves lower errors with the same latent dimension.
This is consistent with the intuition of Section~\ref{sec:analysis} that a more expressive decoder can degrade performance in extrapolation.
The same architectural difference emerges in training: \gls{filmldm} is also the model most sensitive to the batching strategy.
Indeed, in all meshless models we evaluate the loss on $2000$ of the $5017$ mesh nodes per iteration (as detailed in Appendix~\ref{appendix:architectures}).
More aggressive batching  would significantly reduce the computational costs for meshless models; however, in our experiments, it prevented the decoders, particularly \gls{filmldm}, from overcoming the spectral bias discussed in Subsection~\ref{subsec:spectral_bias}, leading them to learn only the mean field.
This limit on batching is the main reason for the higher training time of meshless models in Table~\ref{tab:err_navier_stokes}.

Figure~\ref{fig:latent_NS} shows the evolution of the seventh component of the latent state $\bm s$ of \gls{ldnet} (chosen as a representative example) for all values of $\mu$, highlighting the trajectories corresponding to the smallest and largest values of $\mu$.
For the smallest value of $\mu$, i.e., the largest Reynolds number, $s_7$ oscillates periodically, reflecting the vortex shedding beyond the Hopf bifurcation, whereas for the largest value of $\mu$ it evolves without oscillations, consistently with the convergence of the solution to a steady state.
The latent trajectories therefore provide a clear and interpretable way to distinguish the bifurcating (periodic) regime from the steady one, similarly to what was observed for \gls{ldgcn} in~\cite{tomada_latent_2026} for the pitchfork bifurcation associated with the Coand\u{a} effect, where the amplitude of a single latent component was shown to recover the bifurcation diagram without accessing the full-order fields.

We now turn to \gls{ldnet}$_{(0)}$, which differs from \gls{ldnet} only in the absence of the Fourier feature mapping of the spatial coordinates.
Removing the Fourier features increases both the mean and the maximum test errors by a factor of approximately $5$, making \gls{ldnet}$_{(0)}$ the least accurate of all the considered models.
This degradation comes with negligible savings, as the Fourier features account for fewer than $5000$ decoder weights and for less than $3\%$ of the training time (Table~\ref{tab:err_navier_stokes}).
Figure~\ref{fig:NS_fields} shows three snapshots for different values of $\mu$ and $t$ (columns), together with their reconstructions using \gls{ldgcn}, \gls{ldnet}, and \gls{ldnet}$_{(0)}$ (rows), and clarifies the nature of this failure.
While \gls{ldnet}$_{(0)}$ accurately reconstructs the steady solution in Figure~\ref{fig:NS_field3}, in the periodic regime (Figures~\ref{fig:NS_field1} and~\ref{fig:NS_field2}) it only recovers the mean flow, missing the vortex shedding in the wake of the cylinder, which is captured by all the other methodologies.
Although undesirable, this failure mode has a clear physical interpretation, as the mean flow is the smooth, low-frequency component of the solution, and is therefore the one favored by the spectral bias discussed in Subsection~\ref{subsec:spectral_bias}, whereas the vortex street consists of the higher spatial frequencies that the decoder fails to resolve without a suitable coordinate embedding.
Overall, these results show that, in this setting, the use of Fourier features is essential for meshless decoders to achieve an accuracy comparable to that of mesh-based models.

\begin{table}[tbp]
\caption{Comparison of the architectures on the Navier--Stokes benchmark. Boldface highlights the best value, while underlined values denote the second-best performance.}
\label{tab:err_navier_stokes}
\centering
\begin{tabular}{c|c|c|c|c|c}
     & \multicolumn{2}{c|}{Mesh-based} & \multicolumn{3}{c}{Meshless} \\
     & \gls{ldfnn} & \gls{ldgcn} & \gls{ldnet} & \gls{filmldm} & \gls{ldnet}$_{(0)}$ \\ 
\midrule
    Training time (s) & $\bm{2960}$ & $\underline{6852}$ & $8781$ & $9267$ & $8522$ \\ 
    $\NN{dyn}$ & $11\,737$ & $11\,737$ & $11\,737$ & $11\,737$ & $11\,737$ \\ 
    $\mathcal{D}_{\bm\theta}^{N_h}$/$\NN{dec}$ & $\num{1016234}$ & $\num{1016240}$ & $\underline{22\,782}$ & $29\,837$ & $\bm{18\,022}$ \\
    $\varepsilon_\text{mean}^{\text{test}}$ & $\underline{2.21\cdot10^{-2}}$ & $\bm{1.66\cdot10^{-2}}$ & $2.44\cdot10^{-2}$ & $3.37\cdot10^{-2}$ & $1.21\cdot10^{-1}$ \\
    $\varepsilon_\text{max}^{\text{test}}$ & $6.84\cdot10^{-2}$ & $\bm{3.84\cdot10^{-2}}$ & $\underline{4.20\cdot10^{-2}}$ & $1.16\cdot10^{-1}$ & $2.33\cdot10^{-1}$ \\ 
\bottomrule
\end{tabular}
\end{table}

\begin{figure}[tbp]
    \centering
    \begin{subfigure}[b]{0.51\textwidth}
        \centering
         \includegraphics[width=\linewidth]{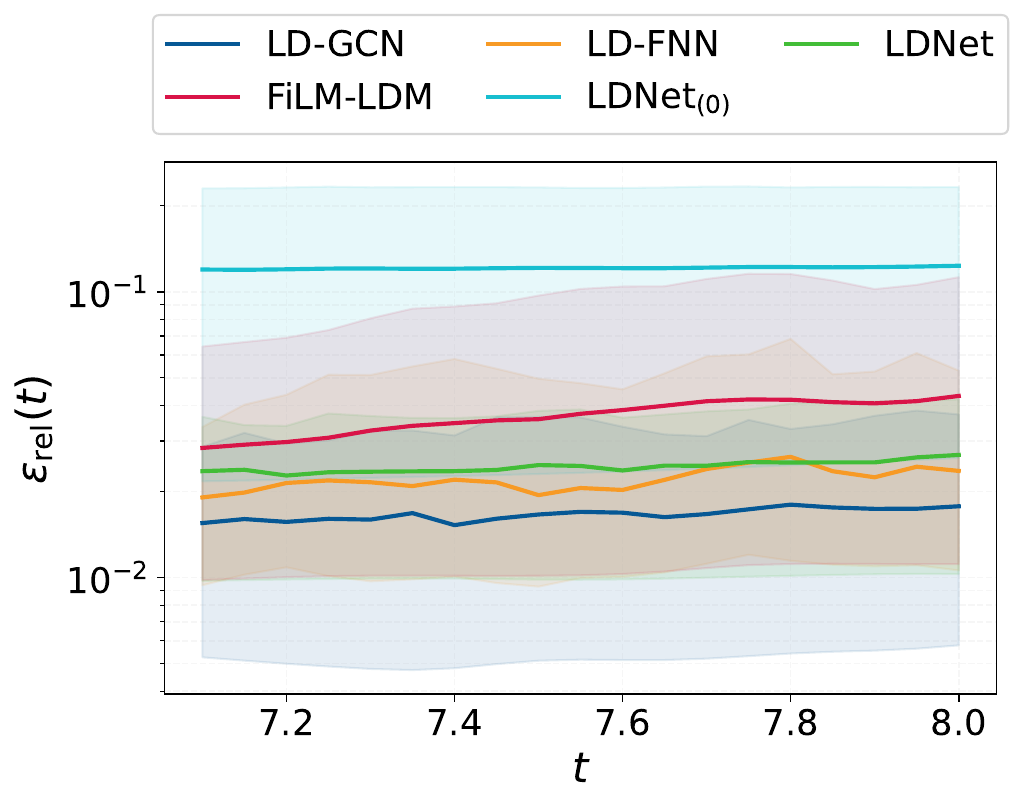}
    \caption{Relative errors}
    \label{fig:errors_NS}
    \end{subfigure}
        \hfill
    \begin{subfigure}[b]{0.48\textwidth}
        \centering
        \includegraphics[width=\textwidth]{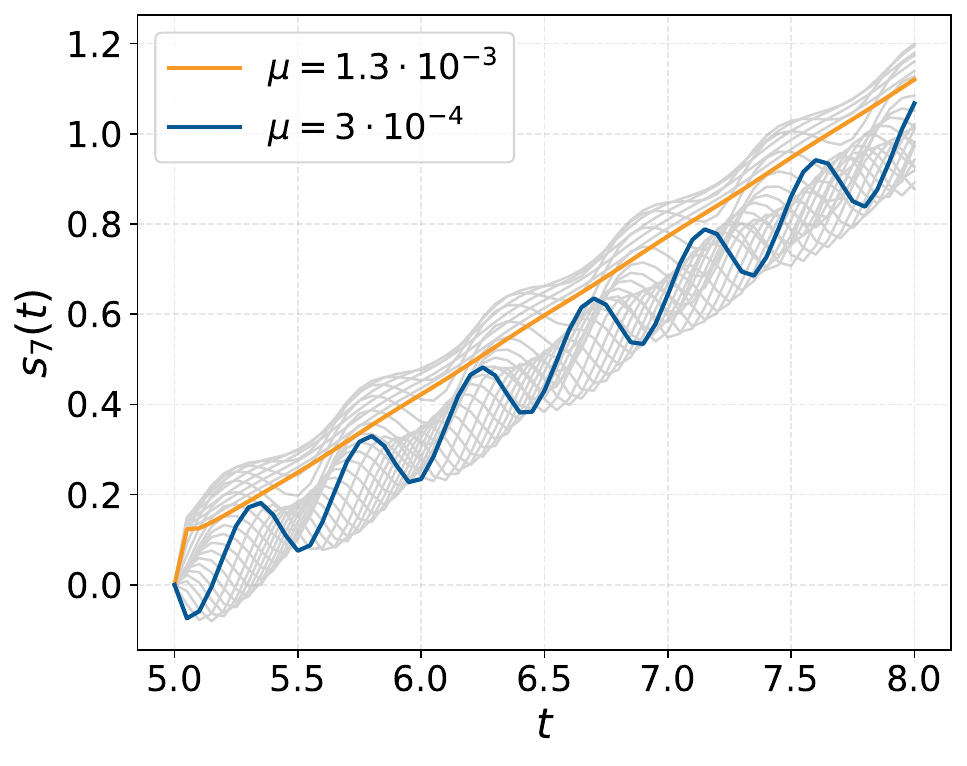}
        \caption{Evolution of $s_7$}
        \label{fig:latent_NS}
    \end{subfigure}
    \caption{Mean relative errors and evolution of the latent trajectories of \gls{ldnet} for all values of $\mu$ in the Navier--Stokes test case.}
    \label{fig:NS_error_and_latents}
\end{figure}

\begin{figure}[tbp]
    \centering
    \begin{subfigure}[b]{0.325\textwidth}
        \centering
        \includegraphics[width=\linewidth]{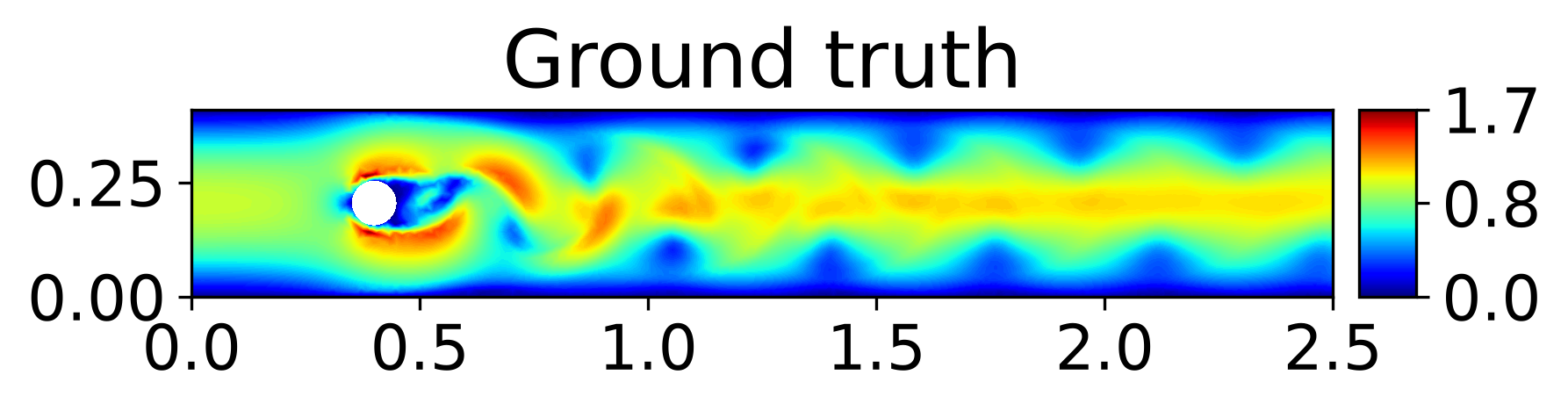}\\[2pt]
        \includegraphics[width=\linewidth]{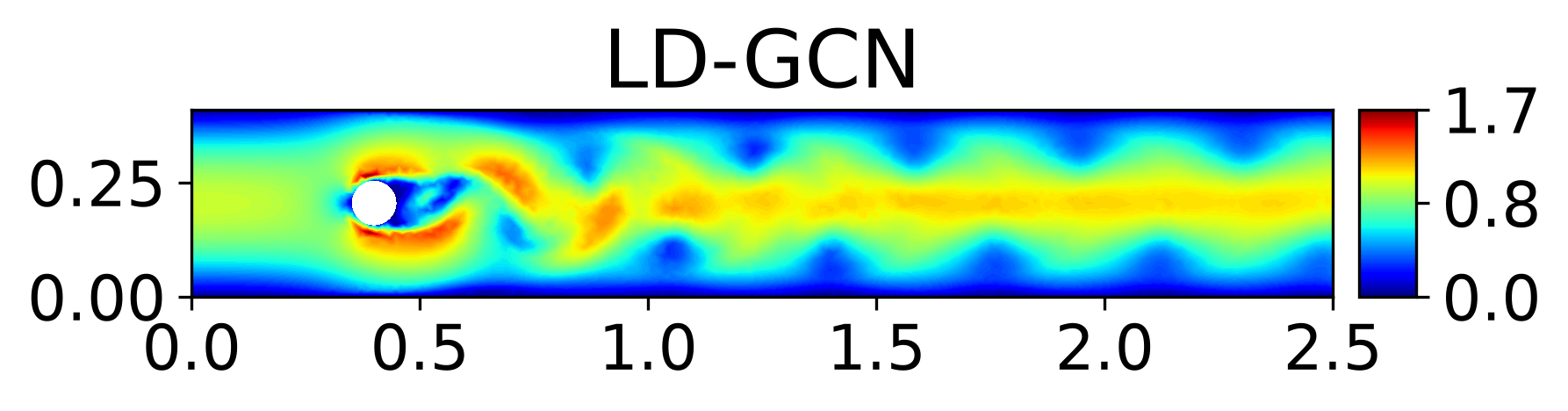}\\[2pt]
        \includegraphics[width=\linewidth]{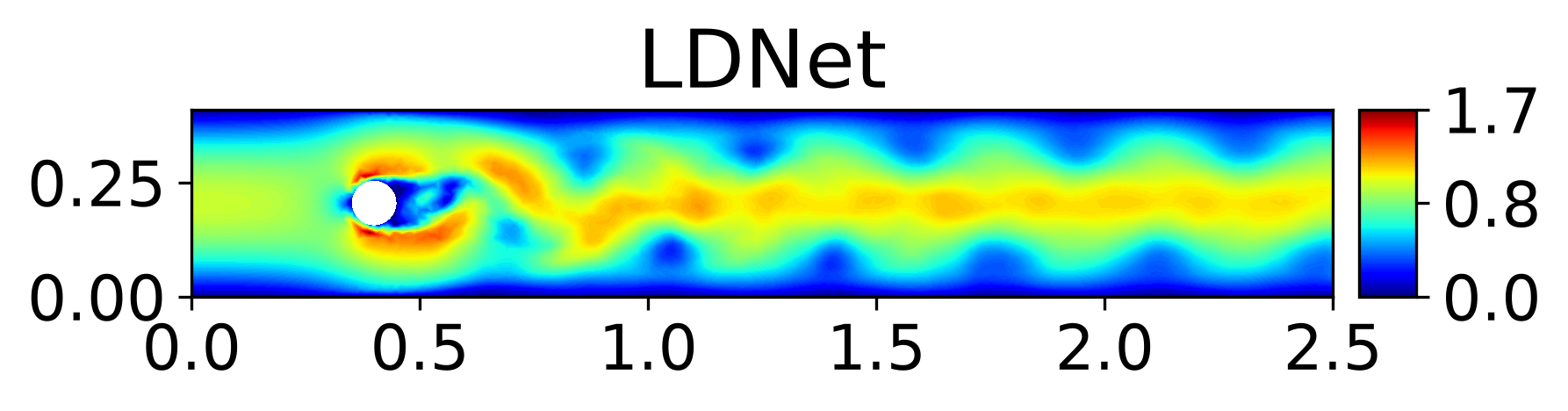}\\[2pt]
        \includegraphics[width=\linewidth]{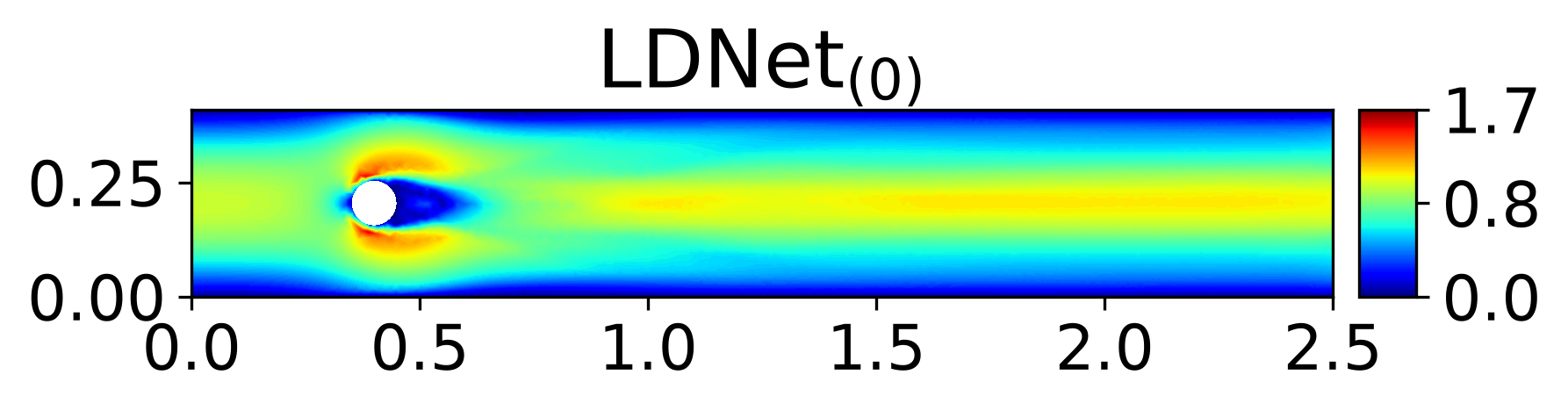}
        \caption{$\mu=3\cdot10^{-4},\,t=7.5$}
        \label{fig:NS_field1}
    \end{subfigure}
    \hfill
    \begin{subfigure}[b]{0.325\textwidth}
        \centering
        \includegraphics[width=\linewidth]{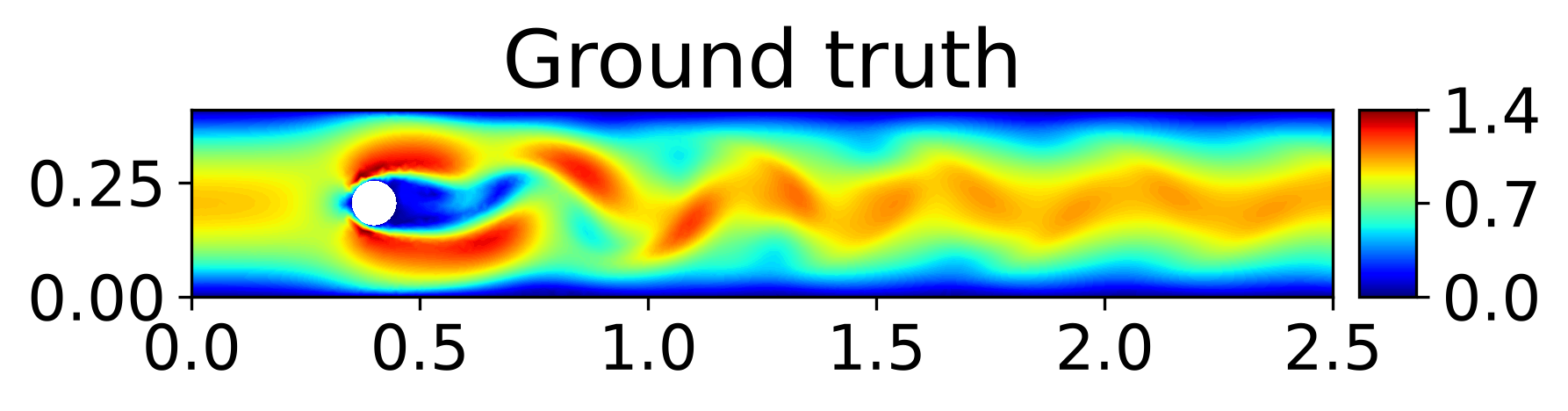}\\[2pt]
        \includegraphics[width=\linewidth]{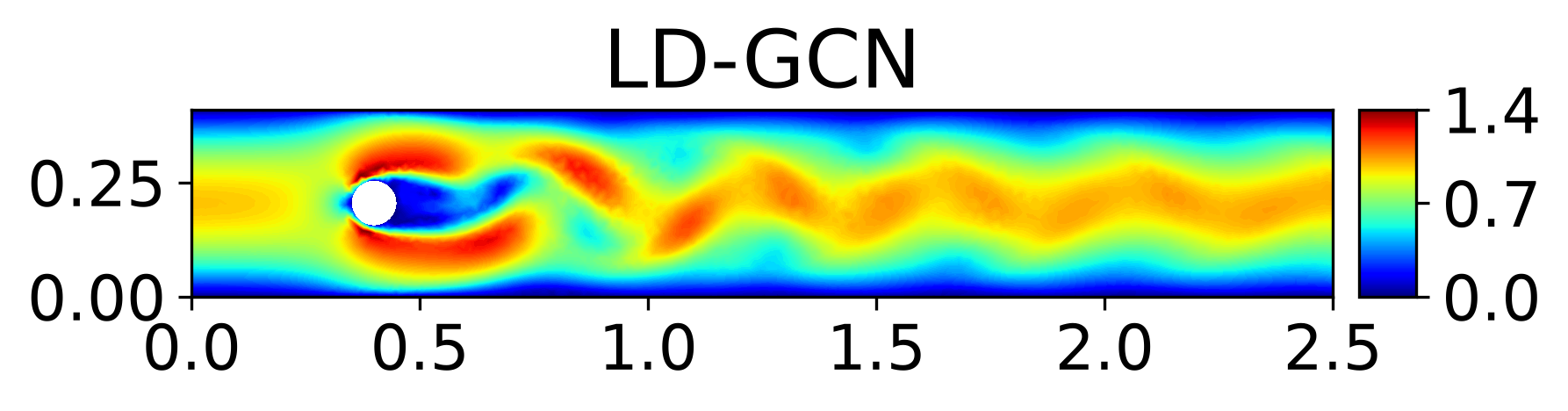}\\[2pt]
        \includegraphics[width=\linewidth]{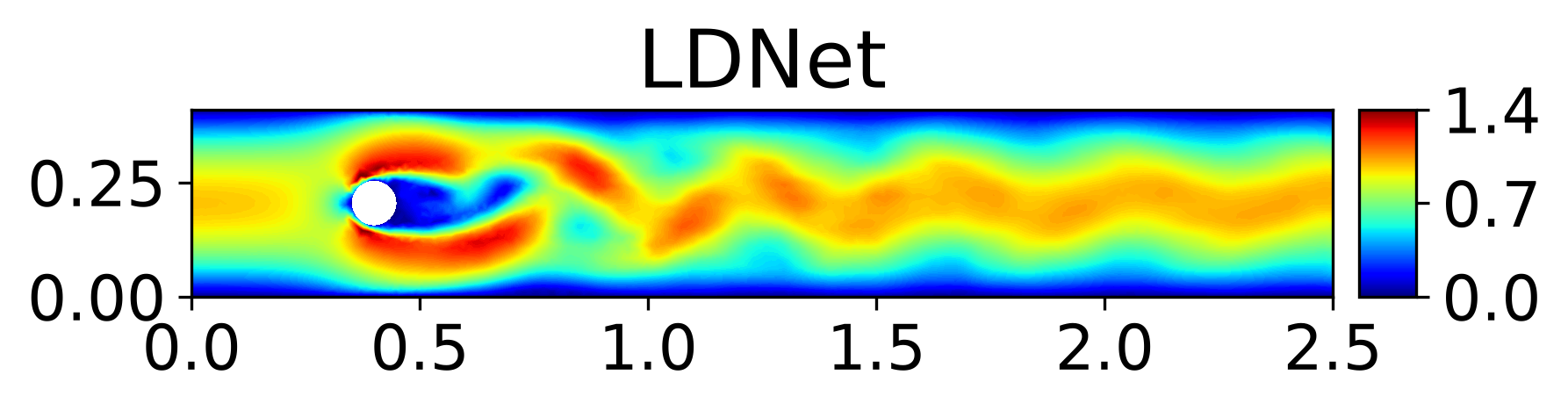}\\[2pt]
        \includegraphics[width=\linewidth]{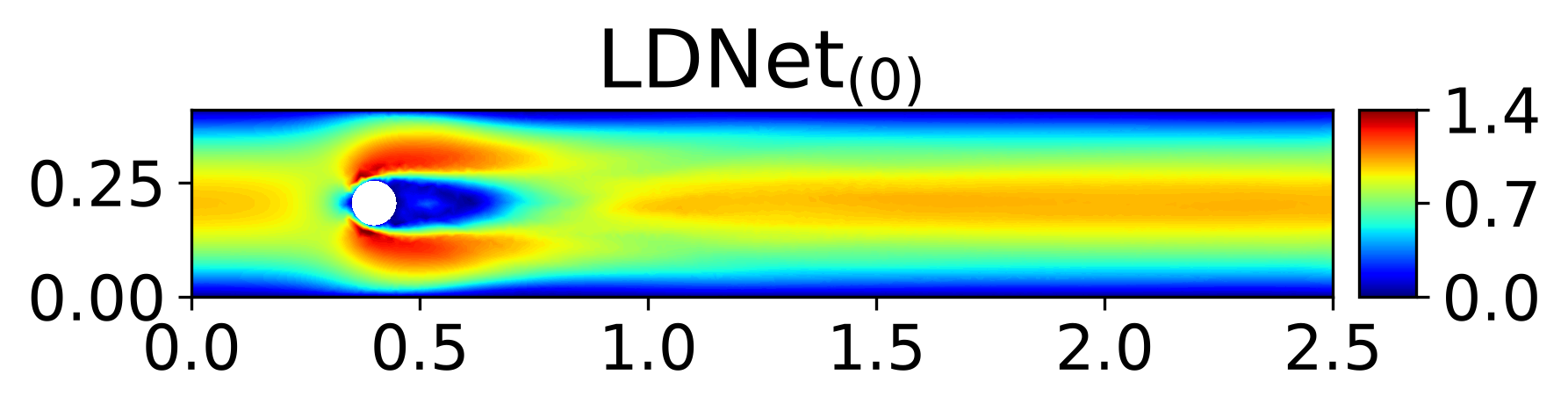}
        \caption{$\mu=8\cdot10^{-4},\,t=7$}
        \label{fig:NS_field2}
    \end{subfigure}
    \hfill
    \begin{subfigure}[b]{0.325\textwidth}
        \centering
        \includegraphics[width=\linewidth]{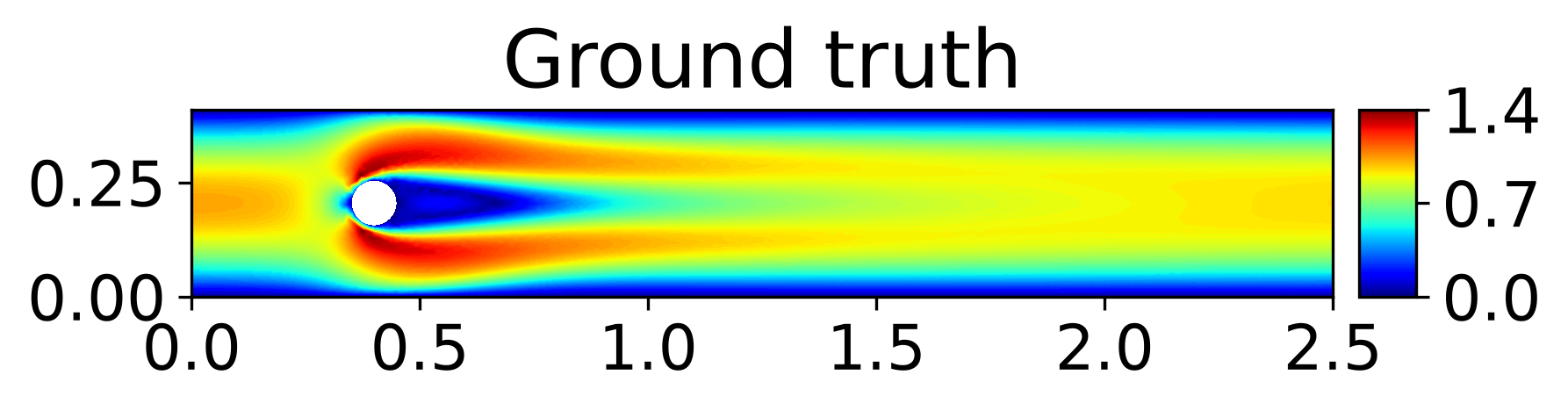}\\[2pt]
        \includegraphics[width=\linewidth]{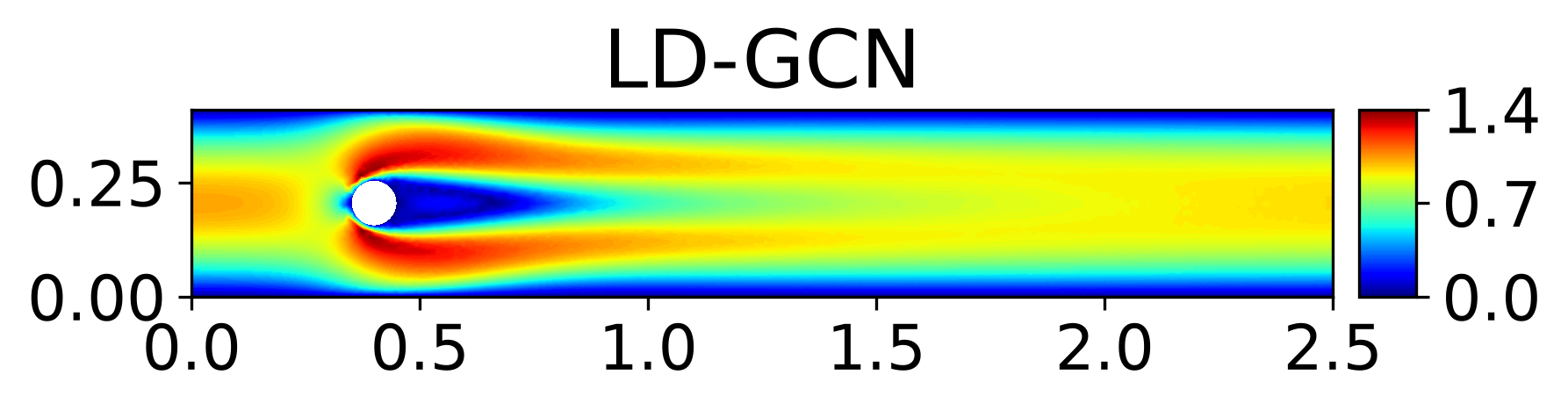}\\[2pt]
        \includegraphics[width=\linewidth]{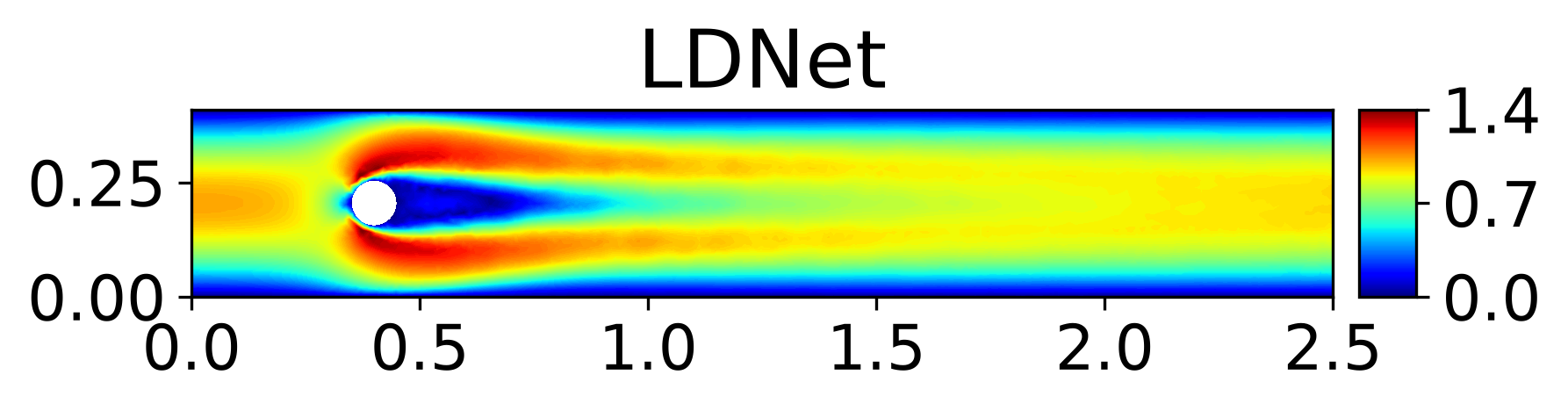}\\[2pt]
        \includegraphics[width=\linewidth]{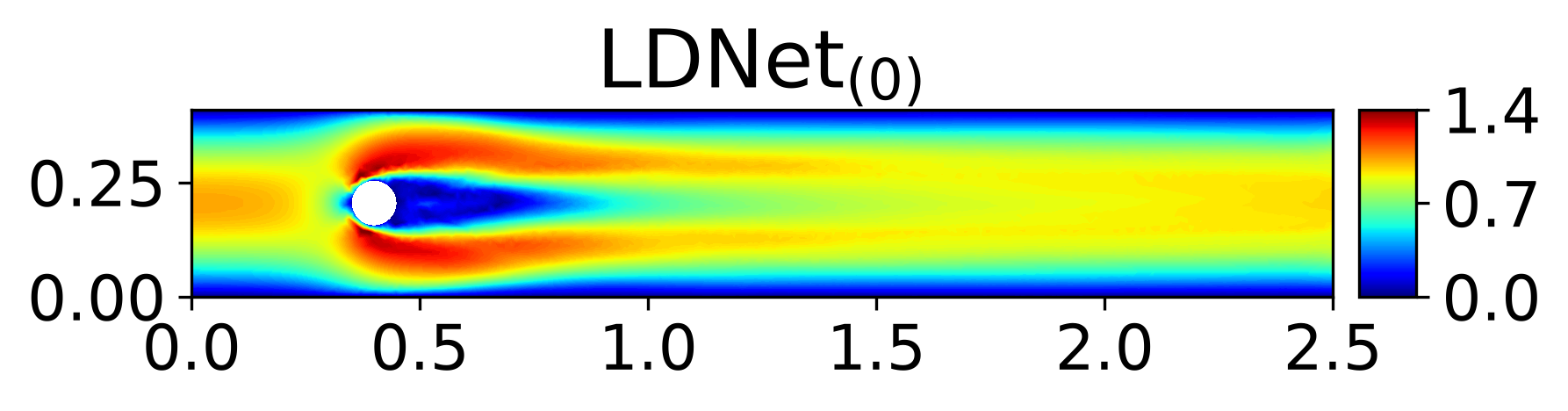}
        \caption{$\mu=1.3\cdot10^{-3},\,t=6.5$}
        \label{fig:NS_field3}
    \end{subfigure}
    \caption{Velocity magnitude of the ground truth solution and of the \gls{ldgcn}, \gls{ldnet}, and \gls{ldnet}$_{(0)}$ approximations (rows) for three selected parameter and time instances (columns).}
    \label{fig:NS_fields}
\end{figure}

\subsection{Linear advection}\label{subsec:advection}
As a final benchmark, we consider a two-dimensional linear advection problem, which we use to assess how the conditioning type and the number of conditioned layers in the decoder affect accuracy and sensitivity to perturbations of the latent state, in light of the analysis of Subsection~\ref{subsec:stability_conditioning}.
This problem is particularly well-suited for evaluating the framework's capacity to learn transport phenomena, which are notoriously challenging for \gls{mor} techniques, even in controlled settings where exact analytical solutions are available.
Moreover, unlike more complex, nonlinear fluid models, this scenario allows us to isolate the effects of pure advection while maintaining a non-trivial parametric dependence.

\subsubsection{Problem definition and data collection}
Let $\Omega=[0,1]^2$ and consider the following linear advection equation:
\begin{equation}\label{eq:advection}
\begin{cases}
\dfrac{\partial u}{\partial t} + \bm\mu\cdot\nabla u=0,& \text{in } \Omega\times(0,T],\\
u(\bm x,t)=u_0(\bm x-\bm\mu t), & \text{on } \Gamma_\text{in}(\bm\mu)\times(0,T],\\
u=u_0,
& \text{on } \Omega\times\{0\},
\end{cases}
\end{equation}
where $u:\Omega\times[0,T]\to\R$ represents the transported scalar field, $T=1$, and parametrized advection velocity $\bm\mu\in\mP=\{(r\cos\theta,r\sin\theta):\theta\in[0,2\pi]\}\subset\R^2$, with $r=0.2$ and angle $\theta$.
The initial condition is chosen as a Gaussian profile centered at $(0.5,0.5)$ with $\sigma=10^{-1}$, i.e., $u_0(\bm x) = \exp\left(-50\left({(x - 0.5)^2} + {(y - 0.5)^2}\right)\right)$, whose evolution under~\eqref{eq:advection} is given analytically by $u(\bm x,t;\bm\mu)=u_0(\bm x-\bm\mu t)$.
Accordingly, the exact solution is prescribed on the inflow boundary $\Gamma_\text{in}(\bm\mu)=\{\bm x\in\partial\Omega:\bm\mu\cdot\bm n(\bm x)<0\}$, where $\bm n$ denotes the outward unit normal to $\partial\Omega$.

The domain $\Omega$ is discretized using a triangular mesh with $1296$ nodes, while the parameter space is sampled with $24$ equispaced values of $\theta\in[0,2\pi]$.
For each parameter instance, we collect $101$ temporal snapshots over the interval $[0,1]$, including the initial condition, resulting in a total of $2424$ solution snapshots.
The dataset is split by randomly selecting $80\%$ of the trajectories for training, and we restrict the training loss to the first $80\%$ of the temporal evolution, i.e., to $t\in[0,T_\text{train}]$ with $T_\text{train}=0.79$, so that the remaining part of the test trajectories corresponds to temporal extrapolation.

\subsubsection{Model performance and sensitivity analysis}
In this setting, we aim to verify the analysis of Subsection~\ref{subsec:stability_conditioning} by isolating the effect of the decoder conditioning; for this reason, all models share the same weakly conditioned architecture for $\NN{dyn}$.
Specifically, we want to (i) study how perturbations in the latent state are amplified by different decoder types, and (ii) further assess how this amplification depends on the spatial coordinate for strong conditioning.

The meshless decoders consist of four layers, of which the first $k\in\{1,2\}$ are conditioned and the remaining $4-k$ are not (Appendix~\ref{appendix:architectures}).
We denote by $\mathrm{W}_k$ and $\mathrm{S}_k$ the decoders with weak and strong conditioning in the first $k$ layers, respectively, so that $\mathrm{W}_1$ corresponds to \gls{ldnet}.
As a mesh-based reference, we also train an \gls{ldfnn} with the same $\NN{dyn}$.
We fix five random seeds, which determine both the split into training and test trajectories and the initialization, and we report the mean over the seeds unless otherwise stated.
Errors are only reported in the extrapolation interval of the test trajectories, to which $\varepsilon^\text{test}_\text{mean}$ and $\varepsilon^\text{test}_\text{max}$ refer.

To quantify the sensitivity of different decoder types, we perturb the latent states $\bm s(t;\bm\mu)$ corresponding to the test trajectories by $\deltas$, with $\|\deltas\|_2=\eta\bar s$, where $\bar s$ is the mean of $\|\bm s(t;\bm\mu)\|_2$ along the trajectory.
The perturbed states are decoded directly, without integrating them in time, so that any amplification of the perturbation is due to the decoder and not to $\NN{dyn}$.
Denoting by $\tilde{\bm u}_\text{sim}(t;\bm\mu)\in\R^{N_h}$ the resulting prediction at the mesh nodes, we define the relative sensitivity
\begin{equation*}
    \mathcal{A}(t;\bm\mu)=\frac{1}{\eta}\,\frac{\|\tilde{\bm u}_\text{sim}(t;\bm\mu)-\bm u_\text{sim}(t;\bm\mu)\|_2}{\|\bm u_\text{sim}(t;\bm\mu)\|_2},
\end{equation*}
i.e., the ratio between the relative perturbations of the output and of the latent state.
We denote by $\mathcal{A}^\text{train}$ and $\mathcal{A}^\text{ext}$ its mean over the training and the extrapolation interval, respectively, and report them for $\eta=10^{-2}$ in Table~\ref{tab:advection}.
This choice of $\eta$ is not restrictive: as shown in Figure~\ref{fig:advection_sensitivity}~(right), the perturbation of the output grows approximately linearly with $\eta$.
Further details on the procedure are given in Appendix~\ref{appendix:sensitivity}.

Subsequently, we investigate the dependence of the bound in Equation~\eqref{eq:pointwise_lipschitz_conditioning} on the spatial coordinate.
Since this bound holds for a single conditioned layer (Proposition~\ref{prop:lipschitz}), we consider the first conditioned layer of the decoder and define $\rho$ as the ratio between the largest and the smallest variation of its output over the mesh nodes (Appendix~\ref{appendix:sensitivity}) for each given $\deltas$.
For weak conditioning, $\bm s(t;\bm\mu)$ only enters through the bias, and the variation of the output of the first layer reduces to $\bm W^{(1)}_{\bm s}\deltas$, which is the same at every node, consistently with $C_1=0$ in Equation~\eqref{eq:pointwise_lipschitz_conditioning}, implying that $\rho=1$.
For strong conditioning, instead, the variation involves the term $\left(\bm\phi^{(1)}(\bm s+\deltas)-\bm\phi^{(1)}(\bm s)\right)\odot\bm W^{(1)}_{\bm h}\bm\gamma(\bm x)$, which depends on $\bm x$, in our case also through the Fourier features $\bm\gamma(\bm x)$.
Therefore, $C_1>0$ in general, and higher values of $\rho$ denote a greater change across the domain.

Table~\ref{tab:advection} compares the training time, the number of trainable parameters, the errors, and the sensitivity of all architectures.
For every $k$, strong conditioning yields more accurate models: $\mathrm{S}_k$ has lower mean and maximum errors than $\mathrm{W}_k$ for all seeds.
Increasing the number of conditioned layers yields at most a small improvement: $\mathrm{S}_2$ is slightly more accurate than $\mathrm{S}_1$ on average, whereas for weak conditioning the converse holds.
\gls{ldfnn} is the fastest model to train, but it has about ten times more decoder parameters and is the least accurate.

Moreover, the values of $\rho$ in Table~\ref{tab:advection} match the structure of the bound: the computed value is $\rho=1$ for every perturbation and seed when weak conditioning is employed.
In contrast, for strongly conditioned models, $\rho$ ranges between $1.26$ and $1.41$ across seeds and values of $k$, consistently with the dependence on the query point in Equation~\eqref{eq:pointwise_lipschitz_conditioning}.

\begin{table}[tbp]
\caption{Comparison of the architectures on the advection benchmark.}
\label{tab:advection}
\centering
\resizebox{\textwidth}{!}{%
\begin{tabular}{c|c|c|c|c|c}
     & Mesh-based & \multicolumn{4}{c}{Meshless} \\
     & \gls{ldfnn} & $\mathrm{W}_1$ (\gls{ldnet}) & $\mathrm{W}_2$ & $\mathrm{S}_1$ & $\mathrm{S}_2$ \\
\midrule
    Training time (s) & $\bm{1802}$ & $2921$ & $2934$ & $\underline{2774}$ & $2800$ \\
    $\NN{dyn}$ & $4674$ & $4674$ & $4674$ & $4674$ & $4674$ \\
    $\mathcal{D}_{\bm\theta}^{N_h}$/$\NN{dec}$ & $\num{131196}$ & $\bm{12\,295}$ & $\underline{12\,423}$ & $12\,487$ & $12\,807$ \\
    $\varepsilon_\text{mean}^{\text{test}}$ & $2.53\cdot10^{-1}$ & $6.54\cdot10^{-2}$ & $7.53\cdot10^{-2}$ & $\underline{4.92\cdot10^{-2}}$ & $\bm{4.63\cdot10^{-2}}$ \\
    $\varepsilon_\text{max}^{\text{test}}$ & $3.74\cdot10^{-1}$ & $1.71\cdot10^{-1}$ & $1.95\cdot10^{-1}$ & $\underline{1.11\cdot10^{-1}}$ & $\bm{1.05\cdot10^{-1}}$ \\
    $\mathcal{A}^\text{train}$ & $1.06$ & $0.743$ & $0.750$ & $\underline{0.729}$ & $\bm{0.718}$ \\
    $\mathcal{A}^\text{ext}$ & $0.944$ & $\underline{0.531}$ & $\bm{0.525}$ & $0.566$ & $0.591$ \\
    $\rho$ & -- & $1.00$ & $1.00$ & $1.35$ & $1.33$ \\
\bottomrule
\end{tabular}}
\end{table}

In terms of sensitivity, all meshless models behave similarly, with $\mathcal{A}^\text{train}$ between $0.72$ and $0.75$, whereas \gls{ldfnn} is more sensitive both during training and extrapolation.
Regarding the comparison between strong and weak conditioning, the observed differences depend on the time interval.
In the training interval, strongly conditioned decoders show a slightly lower value of $\mathcal{A}^\text{train}$.
Nevertheless, during extrapolation, they become more sensitive for both values of $k$, by up to about $13\%$.
Accordingly, the curves in Figure~\ref{fig:advection_sensitivity} cross at $t\approx0.5$--$0.6$, i.e., before the extrapolation interval: the larger sensitivity of strongly conditioned decoders at later times is thus already present within the training interval, and reflects the learned latent parametrization.

Overall, strong conditioning yields higher accuracy; the observed increase in sensitivity during extrapolation relative to weak conditioning, while consistent with Subsection~\ref{subsec:stability_conditioning}, is moderate and does not prevent strongly conditioned models from achieving lower extrapolation errors.
Conditioning more layers, on the other hand, is not always beneficial and, when it is, only leads to minor performance improvements, at least in these experiments.

\begin{figure}[tbp]
    \centering
    \includegraphics[width=\textwidth]{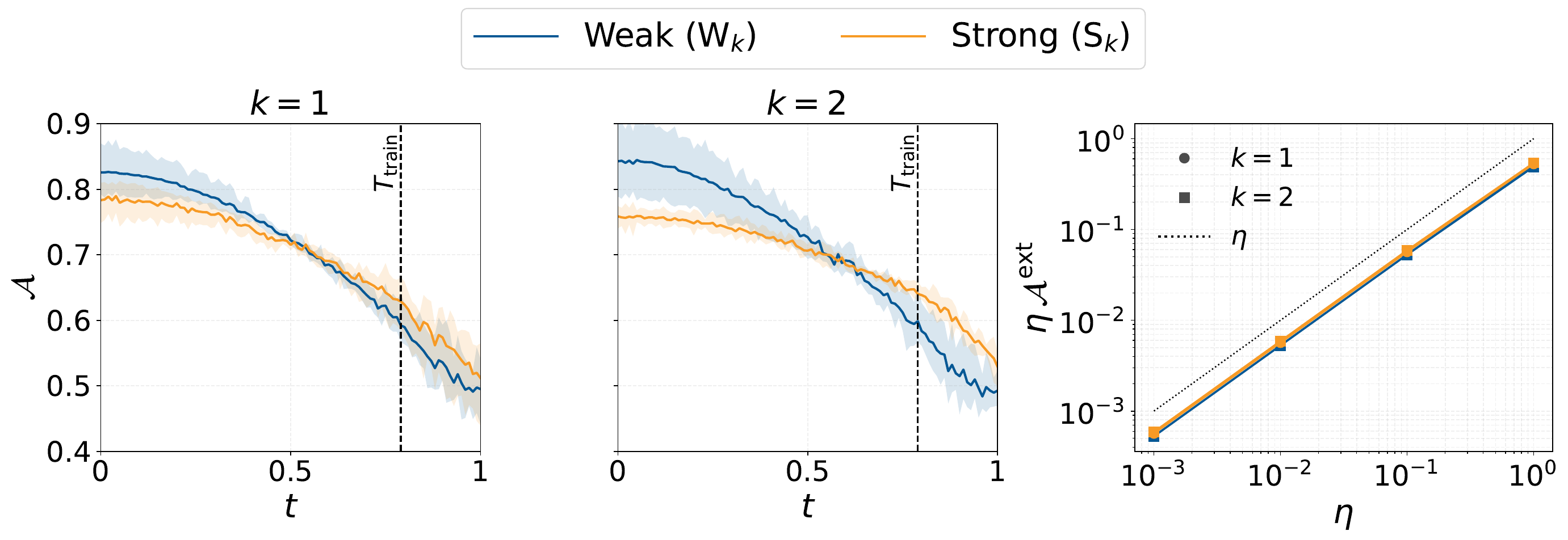}
    \caption{Left: mean relative sensitivity $\mathcal{A}$ for weak and strong conditioning of $k$ decoder layers; shaded areas denote the range over the seeds. Right: relative perturbation of the output $\eta\mathcal{A}^\text{ext}$ in the extrapolation interval.}
    \label{fig:advection_sensitivity}
\end{figure}

%% file: sections/conclusion.tex
\section{Conclusions and future perspectives}\label{sec:conclusion}
In this work, we introduced a unified mathematical framework for mesh-agnostic \glspl{ldm}, interpreting meshless models as \glspl{cnef} and organizing them into a hierarchy according to their conditioning mechanism.
This classification allowed us to compare meshless models theoretically, beyond accuracy, through their stability and robustness properties.
In particular, we derived Lipschitz bounds highlighting the trade-off between decoder expressivity and stability under temporal evolution: while stronger conditioning increases the representational power of the decoder, it can also amplify the propagation of latent perturbations during temporal integration.
Thus, this analysis provides theoretical support for weak conditioning as a robust design choice. 
We also proved a \gls{uat} for weakly conditioned meshless \glspl{ldm}, which naturally extends to stronger forms of conditioning.
Moreover, by studying the continuous limit of the spatial discretization, we showed that meshless decoders naturally avoid the discretization-dependent stability issues that may arise in mesh-based architectures, establishing a conceptual link to operator learning.

The numerical experiments on the Burgers' equation, the Navier--Stokes equations, and the linear advection problem support these theoretical findings.
Specifically, the Burgers benchmark shows that the approximation error, the latent trajectory magnitude, and the decoder Lipschitz constant remain nearly unchanged across different mesh resolutions for meshless models.
In contrast, these quantities exhibit a more marked dependence on the spatial discretization for mesh-based architectures, partially aligning with the scaling behavior predicted for standard weight initializations.
Across all benchmarks, meshless models achieved an accuracy comparable to or better than their mesh-based counterparts while using up to $45$ times fewer trainable decoder parameters.
Furthermore, the experiments show that strong conditioning can either improve accuracy, as in the advection benchmark, or degrade it, as in the Navier--Stokes benchmark.

While meshless \glspl{ldm} offer clear advantages in terms of flexibility and parameter efficiency, they also present specific computational challenges.
Coordinate-based decoders are inherently affected by spectral bias, complicating the reconstruction of high-frequency spatial details, and require pointwise evaluations to recover the continuous solution field.
This spatial querying results in higher memory consumption and slower training times compared to single-pass mesh-based decoders.
To alleviate this computational bottleneck, an important direction for future work is the investigation of adaptive spatial sampling strategies and alternative training objectives.
Additionally, we plan to extend the framework by developing fully mesh-agnostic encoders capable of embedding varying initial conditions without relying on fixed spatial discretizations, consistent with recent developments in \glspl{no}~\cite{dummer_ronom_2026, wu_functional_2024,brivio_handling_2025}.
We also intend to understand in which regimes the higher expressivity of strong conditioning outweighs its higher sensitivity.
Another interesting direction for future research is the extension of our theoretical analysis, for example, examining the behavior of post-training Lipschitz constants and weight distributions, and the relationship between the continuous limit and the optimizer~\cite{xu_width_2026}.
Finally, we plan to extend this framework to \glspl{gnn}, leveraging conditioning to handle complex geometries in \gls{mor}.

\section*{Acknowledgements}
The authors acknowledge the support provided by INdAM-GNCS with the project ``Machine learning scientifico dai modelli surrogati al controllo ottimo'' (CUP E53C25002010001), and the European Union - NextGenerationEU, in the framework of the iNEST - Interconnected Nord-Est Innovation Ecosystem (iNEST ECS00000043 - CUP G93C22000610007) consortium and its CC5 Young Researchers initiative.
The authors also thank Guglielmo Padula for the fruitful discussions.

%% file: sections/appendix.tex
\begin{appendix}
\numberwithin{table}{section}

\section{Methodological details}\label{appendix:formalism}
Here, we provide some definitions and formal details for the theoretical discussions in the manuscript.
Furthermore, we discuss the computational costs of the different conditioning strategies.

\subsection{Notation and basic properties}
Let us start by formally introducing the notation that is used throughout the work, and by proving some useful results.

\begin{definition}\label{def:hadamard}
    Let $\bm A\in\R^{m\times n}$ and $\bm b\in\R^{m}$.
    The Hadamard product is defined as the matrix whose entry in position $(i,j)$ is given by $(\bm b\odot \bm A)_{ij} = b_{i}\cdot a_{ij}$, for $(i,j)\in\{1,\dots, m\}\times\{1,\dots, n\}$, and it coincides with the componentwise product of vectors when $n=1$.
    Moreover, given $\bm{c}\in\R^{n}$, the following associativity property holds $\bm{b}\odot(\bm A\bm c) = (\bm b\odot \bm A)\bm c$.
\end{definition}

We now show that the conditioning mechanisms introduced in Definition~\ref{def:types} can be viewed within the framework of Definition~\ref{def:conditioned_layer}.

Practically speaking, the following result states that (i) input concatenation as in weak conditioning (Equation~\eqref{eq:weak}) can be seen as the choice of only letting the bias depend on $\bm a$ and (ii) the Hadamard product of a conditioning vector with the output of a linear layer can itself be seen as a linear layer in which the weight matrix depends on the conditioning vector.

\begin{lemma}\label{lemma:conditioning}
Let $\hkmo\in\R^{N_{k-1}}$ and $\bm a\in\mA$.
\begin{enumerate}[label=(\roman*)]
\item Let $\tilde{\bm W}\ek\in\R^{N_k\times(N_{k-1}+N_{\bm a})}$, and let $\bm W\uh\ek\in\R^{N_k\times N_{k-1}}$ and $\bm W\ua\ek\in\R^{N_k\times N_{\bm a}}$ be the matrices formed by its first $N_{k-1}$ and last $N_{\bm a}$ columns, respectively.
Then
\begin{equation}\label{eq:weak2}
    \sk\bigl(\tilde{\bm W}\ek(\hkmo\oplus\bm a)\bigr)=\sk\bigl(\bm W\uh\ek\hkmo+\bm W\ua\ek\bm a\bigr),
\end{equation}
i.e., a weakly conditioned layer is a conditioned layer with $\wk(\bm a)\equiv\bm W\uh\ek$ and $\bk(\bm a)=\bm W\ua\ek\bm a$.
\item Let $\bm W\uh\ek\in\R^{N_k\times N_{k-1}}$.
Then
\begin{equation*}
    \sk\bigl(\fk(\bm a)\odot\bm W\uh\ek\hkmo+\bk(\bm a)\bigr)=\sk\bigl(\wk(\bm a)\hkmo+\bk(\bm a)\bigr),
\end{equation*}
with $\wk(\bm a)=\fk(\bm a)\odot\bm W\uh\ek$, i.e., a strongly conditioned layer is a conditioned layer.
\end{enumerate}
\end{lemma}
\begin{proof}
    \begin{enumerate}[label=(\roman*)]
    \item For all $i\in\{1,\dots,N_k\}$, it holds that
\begin{equation*}
    \bigl(\tilde{\bm W}\ek (\bm h\ekmo \oplus\bm a)\bigr)_{i}= \sum_{s=1}^{N_{k-1}}w_{is}h\ekmo_s + \sum^{N_{\bm a}}_{r=1}w_{i(N_{k-1}+r)}a_r=  \bigl(\bm W\uh\ek \bm h\ekmo\bigr)_{i}+ \bigl(\bm W\ek\ua \bm a\bigr)_{i}.
    \end{equation*}
\item Using the associative property of the Hadamard product from Definition~\ref{def:hadamard} we can write
    \begin{equation*}
        \fk(\bm a)\odot(\bm W\ek\uh\bm h\ekmo)=(\fk(\bm a)\odot \bm W\ek\uh)\bm h\ekmo=\bm W\ek(\bm a)\bm h \ekmo.
    \end{equation*}
    \end{enumerate}
\end{proof}

\begin{remark}\label{rmk:all_ones}
Assume that $\bm\phi:\mathbb A\to\R^{N_k}$ is a \gls{nn} such that its last layer is linear.
Then, it is always possible to choose its weights in such a way that $\bm\phi(\bm a)\equiv\bm{1}_{\R^{N_k}}$.
Indeed, it suffices to set all the entries of the weight matrix of the last layer of $\bm\phi$ equal to $0$, and the bias term equal to $\bm{1}_{\R^{N_k}}$.
\end{remark}

We note that in Equation~\eqref{eq:weak2} it is possible to consider an additional bias $\tilde{\bm{b}}^{(k)}\in\R^{N_k}$ that is independent of $\bm{a}$.
This modification is consistent with Definition~\ref{def:conditioned_layer}, as we can simply define $\bk(\bm a)=\bm{W}\ek\ua\bm a + \tilde{\bm{b}}^{(k)}$.
For brevity, we will often avoid explicitly writing this bias, but we employ it in our numerical experiments in Section~\ref{sec:results}. 
More generally, the bias term may depend nonlinearly on $\bm a$; however, throughout this work, we restrict $\bk$ to be affine in $\bm a$.
With this choice, setting $\fk(\bm a)\equiv\bm 1$ in strong conditioning recovers exactly weak conditioning, i.e., the concatenation of $\bm a$ to the layer input.
The two mechanisms therefore differ only in the gating term $\fk(\bm a)$, whose contribution can thus be isolated and studied.

We now prove Corollary~\ref{theo:film}, which is a \gls{uat} for \gls{film}.
\begin{proof}[Proof of Corollary~\ref{theo:film}]
We consider the usual assumptions on $ g(\bm h, \bm a)$ and $\sigma$, under which classical universal approximation theorems hold~\cite{tikk_survey_2003}.
In this setting, given $g$, there exists a shallow network
\begin{equation}\label{eq:uat}
    \mathrm{NN}_\varepsilon(\bm h, \bm a) = \bm c^T\sigma\big(\tilde{\bm W} (\bm h\oplus\bm a) + \tilde{\bm b}\big),\qquad \text{with}\ \bm c\in\R^{N_k}, \quad \text{such that}\quad \| g - \mathrm{NN}_\varepsilon\|<\varepsilon.
\end{equation}
By Lemma~\ref{lemma:conditioning}, it holds that $\tilde{\bm W} (\bm h\oplus\bm a) = \bm W_{\bm h}\bm h + \bm W\ua \bm a$ for suitable weight matrices $\bm W_{\bm h}$ and $\bm W\ua$.
Thus, defining $\bm b(\bm a)\coloneqq \bm W\ua\bm a+\tilde{\bm b}$, Equation~\eqref{eq:uat} represents a shallow \gls{nn} in which the first layer is conditioned using \gls{film} with $\bm\phi\equiv\bm1_{\R^{N_k}}$ as in Remark~\ref{rmk:all_ones}.
\end{proof}

\subsection{Computational cost of conditioning}
We report in Table~\ref{tab:conditioning_cost} the number of trainable parameters and the computational complexity (measured in floating-point operations during the forward pass) for the three conditioning strategies introduced in Section~\ref{sec:methodology}.
To understand the scalings in a clear setting, we assume that all the networks (including $\bm\phi$ and $\bm W$ in strong and pure conditioning, respectively) comprise a single linear layer.
We denote the input dimension by $N_\mathrm{in}$, the output dimension by $N_\mathrm{out}$, and the conditioning vector dimension by $N_{\bm a}$.

\begin{table}[htbp]
\centering
\caption{Comparison of conditioning types in terms of trainable parameters and computational complexity.}
\label{tab:conditioning_cost}
\begin{tabular}{c|c|c}
Conditioning type & \# Trainable parameters & \# Operations \\ 
\midrule
Weak & 
$N_{\mathrm{out}}\left(N_{\mathrm{in}}+N_{\bm a}+1\right)$ & 
$2N_{\mathrm{out}}\left(N_{\mathrm{in}}+N_{\bm a}\right)$ \\ 
%\midrule
Strong & 
$N_{\mathrm{out}}\left(N_{\mathrm{in}}+2N_{\bm a}+2\right)$ & 
$N_{\mathrm{out}}\left(2N_{\mathrm{in}}+4N_{\bm a}+1\right)$ \\ 
Pure & 
$N_{\mathrm{out}}(N_\mathrm{in}+1)(N_{\bm a}+1)$ & 
$2N_{\mathrm{out}}\big(N_{\mathrm{in}}N_{\bm a}+N_{\mathrm{in}}+N_{\bm a}\big)$ \\ 
\bottomrule
\end{tabular}
\end{table}

To analyze the scaling behavior, we consider a typical hidden layer configuration where $N_\mathrm{out}=N_\mathrm{in}=N>2N_{\bm a}$.
From Table~\ref{tab:conditioning_cost} it can be observed that for weak and strong conditioning the leading term in the sum expressing the number of parameters is $N^2$, while for pure conditioning it is $N_{\bm a}N^2$, which could be significantly larger in the case of a high-dimensional parameter space, for instance in complex \gls{mor} tasks.
Furthermore, the same scaling holds for the number of operations.
Within the framework of meshless \glspl{ldm} discussed in Subsection~\ref{subsec:ldms}, the conditioning variable is the latent state $\bm a=\bm s\in\R^n$, meaning that even for moderate latent dimensions the use of pure conditioning could already become infeasible.

From a practical standpoint, we remark that, for the evaluation of $\bm \phi(\bm a)\odot\bm W_{\bm h}\bm h$, it is favorable to compute  $\bm \phi(\bm a)\odot(\bm W_{\bm h}\bm h)$ rather than $(\bm\phi(\bm a)\odot\bm W_{\bm h})\bm h$, since it requires $2N_\mathrm{out}N_\mathrm{in}$ instead of $N_\mathrm{out}(3N_\mathrm{in}-1)$ operations.

\section{Proofs of the main results}\label{appendix:proofs}

\subsection{Proof of the Universal Approximation Theorem and of the off-grid bound}\label{subsec:proof_UAT}
We prove Theorem~\ref{theo:uat} following the procedure detailed for mesh-based \glspl{ldm} in~\cite{farenga_latent_2025}; we first outline the main steps of the proof and then discuss the technical differences with respect to their work.
The proof consists of five steps:
(i) we write the exact latent dynamics in integral form;
(ii) we approximate the encoder, the dynamics and the decoder separately by \glspl{nn} using the approximation result in~\cite{guhring_approximation_2021}, stating the $\varepsilon$-approximability of Lipschitz continuous functions by a \gls{nn} with $O(\varepsilon^{-D})$ non-zero weights, with $D$ the input dimension;
(iii) we bound the resulting latent error at the continuous level;
(iv) we add the time-discretization error and show that the approximate latent states remain in a neighborhood $K$ of the latent manifold $\mathcal S_n=\bm z(I\times\mP)$;
(v) we propagate the latent error to the solution and add the decoder error, so that the total error is at most $\varepsilon$.

The main difference in the result consists in approximating a decoder map from $\overline\Omega\times\R^n\subset\R^d\times\R^n$ to $\R$ rather than one from $\R^n$ to $\R^{N_h}$.
Besides this, our proof differs from that of~\cite{farenga_latent_2025} in two technical points.
First, in Step~5 we bound the effect of the latent error through the Lipschitz constant of the exact decoder $\Psi$, i.e., we estimate $|\Psi(\bm x,\bm z(t_k;\bm\mu))-\Psi(\bm x,\bm s^k(\bm\mu))|$, rather than the corresponding difference for its \gls{nn} approximation, whose Lipschitz constant is set to one without loss of generality in~\cite{farenga_latent_2025}.
However, the Lipschitz constant of the network is not actually controlled by~\cite{guhring_approximation_2021}, as  for each $\varepsilon$ we obtain a different network $\mathcal{NN}(\varepsilon)$ with Lipschitz constant $L(\varepsilon)$ that in general depends on $\varepsilon$.
Our choice avoids this issue at the cost of evaluating the exact decoder at $\bm s^k(\bm\mu)$, which in general does not belong to $\mathcal S_n$.
Hence, $\Psi$ has to be approximated on $\overline\Omega\times K$ rather than only on $\overline\Omega\times\mathcal S_n$.
Second, since~\cite{guhring_approximation_2021} concerns functions on the unit cube, we extend the maps to be approximated using classical extension theorems~\cite{evans_measure_2015} and rescale them accordingly, as detailed below.

We begin by stating the definition of a convergent quadrature formula in our setting.
\begin{definition}[Convergent quadrature formula]\label{def:quadrature}
For $N_{\Delta t}\in\mathbb N$, let $\Delta t=(T-t_0)/N_{\Delta t}$, $t_k=t_0+k\Delta t$, $\mathcal T_{\Delta t}=\{t_1,\dots,t_{N_{\Delta t}}\}$, and let $\{(\omega_j,\tau_j)\}_{j=0}^{N_{\Delta t}}\subset\R\times I$ be quadrature weights and nodes, which may depend on $N_{\Delta t}$.
We say that the quadrature formula is \emph{convergent} if, for every continuous $\bm g:I\times\mP\to\R^n$,
\begin{equation}\label{eq:uat_quadrature}
\max_{k\in\{1,\dots,N_{\Delta t}\}}\sup_{\bm\mu\in\mP}\Bigl\|\int_{t_0}^{t_k}\bm g(\tau;\bm\mu)\,d\tau-\Delta t\sum_{j=0}^{k}\omega_j\,\bm g(\tau_j;\bm\mu)\Bigr\|\longrightarrow0\qquad\text{as }\Delta t\to0.
\end{equation}
\end{definition}

\paragraph{\textbf{Conventions and approximation result.}}
Throughout this proof, $\|\cdot\|$ denotes the maximum norm on $\R^{d_{\bm\mu}}$, $\R^n$ and $\R^{N_h}$; since all norms on a finite-dimensional space are equivalent, Definition~\ref{def:perfect_embedding} still holds, with possibly different constants.
With this choice, approximating a vector-valued map within a certain tolerance corresponds to approximating each of its components within the same tolerance.
Accordingly, we regard a vector-valued network as a collection of scalar networks with the same input, and its number of active weights, i.e., of nonzero weights and biases, as the sum of theirs.
Using~\cite{guhring_approximation_2021}, for every $D\in\mathbb N$ there exist $C_D,\tilde\varepsilon_D>0$, depending only on $D$ and on the activation function, such that, for every $f\in W^{1,\infty}((0,1)^D)$ with $\|f\|_{W^{1,\infty}((0,1)^D)}\leq1$ and every $\varepsilon\in(0,\tilde\varepsilon_D)$, there exists a network with at most $C_D\varepsilon^{-D}$ active weights whose distance from $f$ in $L^\infty((0,1)^D)$ is at most $\varepsilon$.
Since the maps we approximate and the networks are continuous, this bound holds at every point of $[0,1]^D$.
 
\paragraph{\textbf{Simplified setting.}}
To simplify the notation, we assume that:
(i)~$t_0=0$, so that $I=[0,T]$;
(ii)~$\Psi$ is Lipschitz continuous with constant at most one with respect to $\bm s$ on a set $\overline\Omega\times K$;
(iii)~each map that we approximate in Step~2 is defined on a unit cube $[0,1]^D$;
(iv)~the components of each such map have norm at most one in $W^{1,\infty}((0,1)^D)$.
These assumptions are not restrictive and only affect the constants hidden in $O(\cdot)$.
Indeed, choosing $R=\max_{(t,\bm\mu)\in I\times\mP}\|\bm z(t;\bm\mu)\|$ we have that $\mathcal{S}_n\subset[-R,R]^n$, and denoting $K=[-(R+1),R+1]^n$ and by $L_\Psi$ the Lipschitz constant of $\Psi$ with respect to $\bm s$ on $\overline\Omega\times K$, we observe that:
\begin{enumerate}
\item[(i)] It follows from a time shift;
\item[(ii)] It follows from the change of variable $\bm z\mapsto\lambda\bm z$, with $\lambda=\max\{1,L_\Psi\}$, which replaces $\enh$, $\bm F_n$ and $\Psi(\bm x,\bm s)$ with $\lambda\enh$, $\lambda\bm F_n$ and $\Psi(\bm x,\bm s/\lambda)$, and $R$ with $\lambda R$;
\item[(iii)] Each map that we approximate is Lipschitz continuous on a compact set $S\subset\R^D$, namely $\bm F_n$ on $I\times\mP$, $\enh$ on $\mathcal S_0$ and $\Psi$ on $\overline\Omega\times K$; the latter is not a box in general, since $\overline\Omega$ need not be one.
Applied componentwise, classical extension theorems~\cite[Theorem~3.1]{evans_measure_2015} extend each such map to a Lipschitz continuous function on $\R^D$ with the same constant.
As a result, we can evaluate each function on a box enclosing its domain.
 This box can then be bijectively affinely mapped to $[0,1]^D$ and only enters the estimate of $O(\,\cdot\,)$ by constants;
\item[(iv)] For the obtained functions $g\in W^{1,\infty}((0,1)^D)$, we assume that $M_g=\|g\|_{W^{1,\infty}((0,1)^D)}\leq 1$.
Indeed, if that is not the case, it suffices to approximate $g/M_g$ with  $\varepsilon/M_g<\tilde\varepsilon_D$ and then multiply the last layer of the \gls{nn} approximation by $M_g$.
\end{enumerate}

\begin{proof}[Proof of Theorem~\ref{theo:uat}]
Let $n\geq n^\ast$ and $\varepsilon>0$.
 
\paragraph{\textbf{Step 1: exact \gls{ldm}}}
We start by defining the latent dynamics of the system.
By Definition~\ref{def:perfect_embedding}, $\bm z$ is Lipschitz continuous on the compact set $I\times\mP$, with some constant $\Lambda_{\bm z}$, meaning that $R=\max_{I\times\mP}\|\bm z\|<\infty$.
Extending $\bm f_n$ by continuity to $t=0$, the map $\bm F_n(t;\bm\mu)=\bm f_n(t,\bm z(t;\bm\mu);\bm\mu)$ is Lipschitz continuous on $I\times\mP$, with constant $\ell_t+\ell_{\bm z}\Lambda_{\bm z}+\ell_{\bm\mu}$.
Unlike the case of time-dependent signals~\cite{regazzoni_learning_2024,tomada_latent_2026}, $\bm F_n$ depends only on $(t,\bm\mu)$ and, since $\dot{\bm z}=\bm F_n$ and $\bm z(\,\cdot\,;\bm\mu)\in C^1(I;\R^n)$, we can write the latent dynamics as
\begin{equation}\label{eq:uat_integral}
\bm z(t;\bm\mu)=\enh(\bm u_{0,h}(\bm\mu))+\int_0^t\bm F_n(\tau;\bm\mu)d\tau\qquad\forall(t,\bm\mu)\in I\times\mP.
\end{equation}
 
\paragraph{\textbf{Step 2: neural network approximations}}
We begin by describing the domains on which we want to find \gls{nn} approximations of $\bm F_n$, $\Psi$ and $\enh$.
First we notice that, from the continuity of $\bm u_{0,h}$, the set $\mathcal S_0=\bm u_{0,h}(\mP)$, on which $\enh$ is defined, is compact.
Regarding the decoder, as detailed in Step~4, we want to approximate $\Psi$ on the compact set $\overline\Omega\times K$, where $K=[-(R+1),R+1]^n$.

We then apply~\cite[Proposition~4.8]{guhring_approximation_2021}, as recalled above, to the components of $\bm F_n$ on $I\times\mP$, to those of $\enh$ on $\mathcal S_0$, and to $\Psi$ on $\overline\Omega\times K$, with tolerances $\varepsilon/(8T)$, $\varepsilon/8$ and $\varepsilon/2$, respectively.
To be smaller than the needed threshold $\tilde\varepsilon_D$, in which $D$ is the number of input variables, the tolerance needs to satisfy $\varepsilon\leq\varepsilon_0=\min\{2,4T\tilde\varepsilon_{d_{\bm\mu}+1},4\tilde\varepsilon_{N_h},\tilde\varepsilon_{d+n}\}$, where the value $2$ will be used in Step~4.
This is not restrictive: if some networks and some $\Delta t^\ast$ satisfy the final error bound in Equation~\eqref{eq:uat_bound} with tolerance $\varepsilon_0$, then they also satisfy it with any tolerance $\varepsilon>\varepsilon_0$, whereas the bounds on the numbers of active weights in Theorem~\ref{theo:uat} only concern $\varepsilon\to0$.
Hence, we assume $\varepsilon\leq\varepsilon_0$ in the following. 
Thus, from~\cite{guhring_approximation_2021} we obtain networks $\NN{dyn}$, $\ethnh$ and $\NN{dec}$ with at most $O(n\varepsilon^{-(d_{\bm\mu}+1)})$, $O(n\varepsilon^{-N_h})$ and $O(\varepsilon^{-(d+n)})$ active weights, respectively, such that
\begin{equation}\label{eq:uat_nn_bounds}
\begin{aligned}
&\|\bm F_n(t;\bm\mu)-\NN{dyn}(t;\bm\mu)\|\leq\frac{\varepsilon}{8T},\quad\text{and}\quad\|\enh(\bm u_{0,h}(\bm\mu))-\ethnh(\bm u_{0,h}(\bm\mu))\|\leq\frac{\varepsilon}{8},&&\forall(t,\bm\mu)\in I\times\mP,\\
&|\Psi(\bm x,\bm s)-\NN{dec}(\bm x\,|\,\bm s)|\leq\frac{\varepsilon}{2},&&\forall(\bm x,\bm s)\in\overline\Omega\times K.
\end{aligned}
\end{equation}
 
\paragraph{\textbf{Step 3: continuous latent error}}
Let $\bm s(t;\bm\mu)=\ethnh(\bm u_{0,h}(\bm\mu))+\int_0^t\NN{dyn}(\tau;\bm\mu)\,d\tau$ be the approximate latent state at the continuous level, obtained from~\eqref{eq:uat_integral} by replacing $\enh$ with $\ethnh$ and $\bm F_n$ with $\NN{dyn}$.
By~\eqref{eq:uat_integral} and~\eqref{eq:uat_nn_bounds}, the latent error at the continuous level can be bounded as
\begin{equation*}
\begin{aligned}
    E_{\mathrm{CONT}}&=\sup_{(t,\bm\mu)\in I\times\mP}\left\|\bm z(t;\bm\mu)-\bm s(t;\bm\mu)\right\|\\
    &\leq \sup_{(t,\bm\mu)\in I\times\mP}\left\|\enh(\bm u_{0,h}(\bm\mu))-\ethnh(\bm u_{0,h}(\bm\mu))\right\|+\sup_{(t,\bm\mu)\in I\times\mP}\left\|\int_0^t\left(\bm F_n(\tau;\bm\mu)-\NN{dyn}(\tau;\bm\mu)\right)d\tau\right\|\\
    &\leq\frac{\varepsilon}{8}+T\frac{\varepsilon}{8T}=\frac{\varepsilon}{4}.
\end{aligned}
\end{equation*}
 
\paragraph{\textbf{Step 4: time discretization}}
Since $\NN{dyn}$ is continuous and the quadrature formula is convergent, Definition~\ref{def:quadrature} with $\bm g=\NN{dyn}$ yields $\Delta t^*>0$, depending on $\varepsilon$, on $\NN{dyn}$ and on the quadrature formula, such that, for every $\Delta t=T/N_{\Delta t}\leq\Delta t^*$, the quadrature error is bounded by
\begin{equation*}
\begin{aligned}
E_{\mathrm{QUAD}}&=\max_{k\in\{1,\dots,N_{\Delta t}\}}\sup_{\bm\mu\in\mP}\|\bm s(t_k;\bm\mu)-\bm s^k(\bm\mu)\|\\
&=\max_{k\in\{1,\dots,N_{\Delta t}\}}\sup_{\bm\mu\in\mP}\Bigl\|\int_0^{t_k}\NN{dyn}(\tau;\bm\mu)\,d\tau-\Delta t\sum_{j=0}^{k}\omega_j\NN{dyn}(\tau_j;\bm\mu)\Bigr\|\leq\frac{\varepsilon}{4}.
\end{aligned}
\end{equation*}
Putting together the previous bounds, we obtain the following bound on the discretization error:
\begin{equation*}
\begin{aligned}
    E_{\mathrm{DISCR}}&=\max_{k\in\{1,\dots,N_{\Delta t}\}}\sup_{\bm\mu\in\mP}\|\bm z(t_k;\bm\mu)-\bm s^k(\bm\mu)\|\\
    &\leq  \max_{k\in\{1,\dots,N_{\Delta t}\}}\sup_{\bm\mu\in\mP}\|\bm z(t_k;\bm\mu)-\bm s(t_k;\bm\mu)\|+\max_{k\in\{1,\dots,N_{\Delta t}\}}\sup_{\bm\mu\in\mP}\|\bm s(t_k;\bm\mu)-\bm s^k(\bm\mu)\|   \\
    &=E_\mathrm{CONT}+E_\mathrm{QUAD}  \leq \frac{\varepsilon}{4}+\frac{\varepsilon}{4}= \frac{\varepsilon}{2}.
\end{aligned}
\end{equation*}
Finally, we check that the latent states $\bm z(t_k;\bm\mu)$ and $\bm s^k(\bm\mu)$ at which $\Psi$ and $\NN{dec}$ are evaluated in Step~5 belong to $K$, where both the $1$-Lipschitz bound on $\Psi$ of the simplified setting and the decoder bound in~\eqref{eq:uat_nn_bounds} hold.
Recalling that $K=[-(R+1),R+1]^n$, with $R=\max_{I\times\mP}\|\bm z\|$, we have that $\bm z(t_k;\bm\mu)\in[-R,R]^n\subset K$ and, as $\varepsilon\leq2$,
\begin{equation*}
\|\bm s^k(\bm\mu)\|\leq\|\bm z(t_k;\bm\mu)\|+\|\bm z(t_k;\bm\mu)-\bm s^k(\bm\mu)\|\leq R+E_{\mathrm{DISCR}}\leq R+\frac{\varepsilon}{2}\leq R+1,
\end{equation*}
i.e., $\bm s^k(\bm\mu)\in K$, for all $k\in\{1,\dots,N_{\Delta t}\}$ and $\bm\mu\in\mP$.
 
\paragraph{\textbf{Step 5: error bound}}
Let $k\in\{1,\dots,N_{\Delta t}\}$, $\bm x\in\overline\Omega$ and $\bm\mu\in\mP$, and recall that $u_h(\bm x,t_k;\bm\mu)=\Psi(\bm x,\bm z(t_k;\bm\mu))$.
Adding and subtracting $\Psi(\bm x,\bm s^k(\bm\mu))$ yields a bound in terms of the latent and decoder errors, $E_{\mathrm{LAT}}$ and $E_{\mathrm{DEC}}$ respectively, as follows:
\begin{equation*}
|u_h(\bm x,t_k;\bm\mu)-u^k_{\mathrm{sim}}(\bm x;\bm\mu)|\leq\underbrace{\bigl|\Psi(\bm x,\bm z(t_k;\bm\mu))-\Psi(\bm x,\bm s^k(\bm\mu))\bigr|}_{E_{\mathrm{LAT}}} + \underbrace{\bigl|\Psi(\bm x,\bm s^k(\bm\mu))-\NN{dec}(\bm x\,|\,\bm s^k(\bm\mu))\bigr|}_{E_{\mathrm{DEC}}}.
\end{equation*}
Since $(\bm x,\bm s^k(\bm\mu)), (\bm x,\bm z(t_k;\bm\mu))\in\overline\Omega\times K$, thanks to the assumption that $\Psi$ is $1$-Lipschitz, it holds
\begin{equation*}
   E_{\mathrm{LAT}}\leq \|\bm z(t_k;\bm\mu)-\bm s^k(\bm \mu)\|\leq E_\mathrm{DISCR}\leq\frac{\varepsilon}{2}.
\end{equation*}
Finally, since by Equation~\eqref{eq:uat_nn_bounds} we have that $E_{\mathrm{DEC}}\leq\varepsilon/2$, the resulting bound is obtained as
\begin{equation*}
    |u_h(\bm x,t_k;\bm\mu)-u^k_{\mathrm{sim}}(\bm x;\bm\mu)|\leq\varepsilon.
\end{equation*}
\end{proof}

\begin{remark}[Continuous-time version]\label{rmk:uat_continuous}
The networks constructed in the proof of Theorem~\ref{theo:uat} for a tolerance $\varepsilon$ also satisfy
\begin{equation*}
\sup_{(t,\bm\mu)\in I\times\mP}|u_h(\bm x,t;\bm\mu)-\NN{dec}(\bm x\,|\,\bm s(t;\bm\mu))|\leq\varepsilon\qquad\forall\bm x\in\overline\Omega,
\end{equation*}
where $\bm s(t;\bm\mu)$ is the continuous approximate latent state of Step~3.
Indeed, $E_{\mathrm{CONT}}\leq\varepsilon/4\leq1$ implies $\bm s(t;\bm\mu)\in K$, so that Step~5 applies with $\bm s(t;\bm\mu)$ in place of $\bm s^k(\bm\mu)$.
Consequently, the same holds in $L^2(\Omega)$ with $\varepsilon|\Omega|^{1/2}$ in place of $\varepsilon$, and~\eqref{eq:condition_corollary} is satisfied by choosing the tolerance $\varepsilon/(2|\Omega|^{1/2})$.
\end{remark}

We now aim at proving a bound for the pointwise error at any arbitrary off-grid coordinate $\bm x\in\Omega$.
We assume that $\NN{dec}$ is Lipschitz continuous with respect to $\bm x$ uniformly in $\bm s$ with constant $C_\mathrm{sim}$, and that $u_h$ is Lipschitz continuous with respect to $\bm x$ uniformly in $(t,\bm\mu)$ with constant $C_u$.
\begin{proof}[Proof of Remark~\ref{rmk:off_grid}.] Given $\bm x\in\Omega$, let $\bm x^\ast \in \mathcal{M}_h$ be the nearest grid point\footnote{If $\bm x^\ast$ is not unique, we arbitrarily choose one of the nearest grid points.}.
By assumption, the maximum absolute error at the grid points is bounded by $\varepsilon_{\mathrm{grid}}$, implying that
\begin{equation*}|u_h(\bm x^\ast, t; \bm\mu) - u_{\mathrm{sim}}(\bm x^\ast, t; \bm\mu)| \leq \varepsilon_{\mathrm{grid}}.\end{equation*}
Using the triangle inequality and the Lipschitz continuity, we can write:
\begin{equation*}
    \begin{aligned}
 |u_h(\bm x,t;\bm\mu)- u_\mathrm{sim}(\bm x,t;\bm\mu)|&= |u_h(\bm x,t;\bm\mu)- \NN{dec}(\bm x\,|\, \bm s(t;\bm\mu))| \\
 &\leq |u_h(\bm x,t;\bm\mu)- u_h(\bm x^\ast,t;\bm\mu)|\\
 & \qquad\qquad + |u_h(\bm x^\ast,t;\bm\mu)- \NN{dec}(\bm x^\ast\,|\, \bm s(t;\bm\mu))| \\
 & \qquad\qquad + |\NN{dec}(\bm x^\ast\,|\, \bm s(t;\bm\mu)) -\NN{dec}(\bm x\,|\, \bm s(t;\bm\mu))| \\
 & \leq C_u\|\bm x-\bm x^\ast\|_2+ \varepsilon_\mathrm{grid}+C_\mathrm{sim}\|\bm x - \bm x^\ast\|_2\\
 & \leq Ch_{\max}+\varepsilon_\mathrm{grid},
    \end{aligned}
\end{equation*}
with $C = C_u + C_{\mathrm{sim}}$.
\end{proof}

\subsection{Proof of the stability results}\label{appendix:stability}
We prove the results related to stability that are discussed in Subsections~\ref{subsec:stability_conditioning} and~\ref{subsec:stability_ML_MB}.
Specifically, we first compute the Lipschitz constant of a conditioned layer (Proposition~\ref{prop:lipschitz}) and then provide a bound on the Lipschitz constant of the decoder as the mesh is refined (Proposition~\ref{prop:stability_L2}).

\begin{proof}[Proof of Proposition~\ref{prop:lipschitz}]
Let $\bm a_1,\bm a_2\in\mA$.
Using the Lipschitz continuity of $\sigma$, the triangle inequality, and the fact that the $p$-norm is a natural matrix norm, it follows that
\begin{equation*}
\begin{aligned}
      \|\bm{l}(\bm h\,|\,\bm a_1)-\bm{l}(\bm h\,|\,\bm a_2)\|_p&=\|\sigma\left(\bm W(\bm a_1)\bm h+\bm b(\bm a_1)\right) - \sigma\left(\bm W(\bm a_2)\bm h+\bm b(\bm a_2)\right)\|_p\\
      &\leq C_\sigma\left(\|\left(\bm W(\bm a_1)-\bm W(\bm a_2)\right)\bm h\|_p+\|\bm b(\bm a_1)-\bm b(\bm a_2)\|_p\right) \\
      &\leq C_\sigma\left(\|\bm W(\bm a_1)-\bm W(\bm a_2)\|_p\|\bm h\|_p + C_{\bm b}\|\bm a_1-\bm a_2\|_p\right) \\
      &\leq  C_\sigma\left(C_{\bm W}\|\bm a_1-\bm a_2\|_p\|\bm h\|_p + C_{\bm b}\|\bm a_1-\bm a_2\|_p\right)\\
      & = C_\sigma (C_{\bm W}\|\bm h\|_p+C_{\bm b})\|\bm a_1-\bm a_2\|_p,
\end{aligned}
\end{equation*}
proving Equation~\eqref{eq:stability}.

Now, from the linearity of $\bm b$ discussed in Section~\ref{sec:methodology}, we can write $\bm b(\bm a )=\bm W_{\bm a}\bm a$, so that $C_{\bm b}\leq \|\bm W_{\bm a}\|_p$.
We remark that, up to this point, the only assumption on the norm that is used is that it is compatible~\cite{quarteroni_numerical_2007}, i.e., such that $\|\bm A\bm x\|\leq\|\bm A\| \|\bm x\|$ for each $\bm x$.

To bound $C_{\bm W}$, we write the definition of matrix norm:
\begin{equation*}
\|\bm W(\bm a_1)-\bm W(\bm a_2)\|_p = \max_{\|\bm c\|_p=1}\|\left(\bm W(\bm a_1)-\bm W(\bm a_2)\right)\bm c\|_p =  \max_{\|\bm c\|_p=1}\|\left(\bm\phi(\bm a_1)-\bm \phi(\bm a_2)\right)\odot \bm W_{\bm h}\bm c\|_p.
\end{equation*}
From the previous expression, since $\|\bm u\odot\bm v\|_p\leq\|\bm u\|_\infty\|\bm v\|_p\leq\|\bm u\|_p\|\bm v\|_p$ for all $\bm u,\bm v$ (as $|u_iv_i|\leq\|\bm u\|_\infty|v_i|$), choosing $\bm u=\bm\phi(\bm a_1)-\bm\phi(\bm a_2)$ and $\bm v=\bm W_{\bm h}\bm c$, we have that
\begin{equation*}
\begin{aligned}
    \|\bm W(\bm a_1)-\bm W(\bm a_2)\|_p& =\max_{\|\bm c\|_p=1}\|(\bm\phi(\bm a_1)-\bm\phi(\bm a_2))\odot\bm W_{\bm h}\bm c\|_p\\
    &\leq \max_{\|\bm c\|_p=1} \|\bm\phi(\bm a_1)-\bm \phi(\bm a_2)\|_p\|\bm W_{\bm h}\bm c\|_p\\
    &\leq  \|\bm\phi(\bm a_1)-\bm \phi(\bm a_2)\|_p\|\bm W_{\bm h}\|_p\\
    &\leq C_{\bm\phi}\|\bm a_1-\bm a_2\|_p\|\bm W_{\bm h}\|_p,
    \end{aligned}
\end{equation*}
so that the Lipschitz constant of $\bm W$ is at most $C_{\bm\phi}\|\bm W_{\bm h}\|_p$.
To conclude, if weak conditioning is used, $\bm W(\bm a_1)=\bm W(\bm a_2)$ for all $\bm a_1, \bm a_2\in\mA$, meaning that the Lipschitz constant of $\bm W$ is $0$.
\end{proof}

We now bound the $L^2(\Omega)$ effect of a latent perturbation $\deltas$ on the decoder output (Proposition~\ref{prop:stability_L2}).
The meshless case follows from the Lipschitz continuity of the activation, while the mesh-based case relies on bounding the mass-matrix quadratic form through its largest eigenvalue.

\begin{proof}[Proof of Proposition~\ref{prop:stability_L2}]
In the meshless case, the inequality follows immediately observing that
\begin{equation*}
\begin{split}
    \|\NN{dec}(\,\cdot\,|\,\bm s +\deltas)-\NN{dec}(\,\cdot\,|\,\bm s )\|^2_{L^2(\Omega)}={\int_\Omega |\NN{dec}(\bm x\,|\,\bm s + \deltas)-\NN{dec}(\bm x\,|\,\bm s )|^2d\bm x}\\
    = {\int_\Omega |\sigma\left(\bm W_{\bm x}\bm x + \bm{W}_{\bm s}(\bm s+\deltas)+\bm b\right)-\sigma\left(\bm W_{\bm x}\bm x + \bm{W}_{\bm s}\bm s+\bm b\right)|^2d\bm x}\\
    \leq {\int_\Omega C_\sigma^2\|\bm W_{\bm s}\|_2^2\|\deltas\|_2^2d\bm x}
    ={|\Omega|}C^2_\sigma\|\bm W_{\bm s}\|_2^2 \|\deltas\|_2^2.
    \end{split}
\end{equation*}
In the case of the mesh-based model, we can write
\begin{equation*}
\begin{split}
    \left\|\tilde{\mathcal{D}}^{N_h}_{\bm\theta}(\,\cdot\,,\,\bm s+\deltas)-\tilde{\mathcal{D}}^{N_h}_{\bm\theta}(\,\cdot\,,\,\bm s)\right\|_{L^2(\Omega)}^2={\int_\Omega\left|\tilde{\mathcal{D}}^{N_h}_{\bm\theta}(\bm x,\bm s+\deltas)-\tilde{\mathcal{D}}^{N_h}_{\bm\theta}(\bm x, \bm s)\right|^2d\bm x}\\
        ={\int_\Omega \left(\sum_{k=1}^{N_h}\left(\mathcal{D}^{N_h}_{\bm\theta}(\bm s+\deltas)_k-\mathcal{D}^{N_h}_{\bm\theta}(\bm s)_k\right)\varphi_k(\bm x)\right)^2 d\bm x}\\
          ={\int_\Omega \left(\sum_{k=1}^{N_h}\left(\mathcal{D}^{N_h}_{\bm\theta}(\bm s+\deltas)_k-\mathcal{D}^{N_h}_{\bm\theta}(\bm s)_k\right)\varphi_k(\bm x)\right)\left(\sum_{j=1}^{N_h}\left(\mathcal{D}^{N_h}_{\bm\theta}(\bm s+\deltas)_j-\mathcal{D}^{N_h}_{\bm\theta}(\bm s)_j\right)\varphi_j(\bm x)\right) d\bm x}\\
                    ={\sum_{k,j=1}^{N_h}\left(\mathcal{D}^{N_h}_{\bm\theta}(\bm s+\deltas)_k-\mathcal{D}^{N_h}_{\bm\theta}(\bm s)_k\right)\left(\mathcal{D}^{N_h}_{\bm\theta}(\bm s+\deltas)_j-\mathcal{D}^{N_h}_{\bm\theta}(\bm s)_j\right)\int_\Omega \varphi_k(\bm x)\varphi_j(\bm x) d\bm x}={\bm\delta^T\bm M\bm\delta},
\end{split}
\end{equation*}
having defined the mass matrix $\bm M \in \mathbb{R}^{N_h \times N_h}$ such that $M_{ij}=\int_\Omega\varphi_j(\bm x)\varphi_i(\bm x)d\bm x$, and writing $\bm\delta=\mathcal{D}^{N_h}_{\bm\theta}(\bm s+\deltas)-\mathcal{D}^{N_h}_{\bm\theta}(\bm s)$.
By the Courant--Fisher theorem~\cite{quarteroni_numerical_2007}, since $\bm M$ is symmetric, it holds that
\begin{equation*}
    {\bm\delta^T\bm M\bm\delta}\leq\lambda_{\mathrm{max}}(\bm M)\|\bm\delta\|_2^2,
\end{equation*}
where $\lambda_\mathrm{max}(\bm M)$ is the maximum eigenvalue of $\bm M$, which is real.
However, from classical \gls{fe} theory~\cite{ern_theory_2004} we know that there exists a constant $C_M>0$ independent of $h$ such that $\lambda_{\mathrm{max}}(\bm M)\leq C_M^2h^d$ for each $h$.
Additionally, noticing that
\begin{equation*}
        \|\bm\delta\|_2=\left\|\sigma\left(\bm W_{\bm\theta}^{N_h}(\bm s+\deltas)+\bm b\right) -\sigma\left(\bm W_{\bm\theta}^{N_h}\bm s+\bm b\right)\right\|_2\leq C_\sigma\left\|\bm W_{\bm\theta}^{N_h}\right\|_2\|\deltas\|_2,
\end{equation*}
we obtain 
\begin{equation*}
     \left\|\tilde{\mathcal{D}}^{N_h}_{\bm\theta}(\,\cdot\,,\,\bm s+\deltas )-\tilde{\mathcal{D}}^{N_h}_{\bm\theta}(\,\cdot\,,\,\bm s )\right\|_{L^2(\Omega)}=\sqrt{\bm\delta^T\bm M\bm\delta}\leq \sqrt{\lambda_\mathrm{max}(\bm M)\|\bm\delta\|_2^2}\leq h^{\frac{d}{2}}C_MC_\sigma \left\|\bm W^{N_h}_{\bm\theta}\right\|_2\|\deltas\|_2.
\end{equation*}
The final asymptotic bound follows immediately from the fact that, given our hypotheses on the mesh, it holds that $h^d = O\left(\frac{1}{N_h}\right)$~\cite{ern_theory_2004}.
\end{proof}

We now prove Proposition~\ref{corollary:magnitude}, which lower-bounds the magnitude $S$ of the latent trajectories in terms of the variation $C$ of the solution, the approximation error $\varepsilon$, and the Lipschitz constant of the decoder.

\begin{proof}[Proof of Proposition~\ref{corollary:magnitude}]
Let $t_1,t_2\in I$ and $\bm\mu\in\mP$.
Denoting by $L_{\NN{dec}}$ the Lipschitz constant of the decoder with respect to the latent state, by the triangle inequality and Equation~\eqref{eq:condition_corollary} it holds that
\begin{equation*}
\begin{aligned}
  \| u_h(\,\cdot\,,t_1;\bm\mu)- u_h(\,\cdot\,,t_2;\bm\mu)\|_{L^2(\Omega)}
  &\leq \| u_h(\,\cdot,t_1;\bm\mu)-\NN{dec}(\,\cdot\,|\,\bm s(t_1;\bm\mu))\|_{L^2(\Omega)}\\
   &\qquad\qquad+\| \NN{dec}(\,\cdot\,|\,\bm s(t_1;\bm\mu))- \NN{dec}(\,\cdot\,|\,\bm s(t_2;\bm\mu)) \|_{L^2(\Omega)}\\
  &\qquad\qquad+\|\NN{dec}(\,\cdot\,|\,\bm s(t_2;\bm\mu))- u_h(\,\cdot\,,t_2;\bm\mu)\|_{L^2(\Omega)}\\
  &\leq \frac{\varepsilon}{2}+L_{\NN{dec}}\|\bm s(t_1;\bm\mu)-\bm s(t_2;\bm\mu)\|_2 + \frac{\varepsilon}{2}\\
  &\leq \varepsilon+L_{\NN{dec}}(\|\bm s(t_1;\bm\mu)\|_2+\|\bm s(t_2;\bm\mu)\|_2)\\
   &\leq \varepsilon+2SL_{\NN{dec}}.
\end{aligned}
\end{equation*}
Taking the supremum over $t_1,t_2\in I$ and $\bm\mu\in\mP$ gives $C\leq\varepsilon+2SL_{\NN{dec}}$.
Subtracting $\varepsilon$ and dividing by $2L_{\NN{dec}}$ on both sides gives Equation~\eqref{eq:latent_magnitude_2}.
The mesh-based case is identical, since only Equation~\eqref{eq:condition_corollary} and the Lipschitz continuity of the decoder were used.
\end{proof}

Although we do not adopt it, spectral normalization~\cite{miyato_spectral_2018} could be used to directly control the Lipschitz constant of the network, and hence the bound on $S$ previously discussed.

\section{Training setup and architectures}\label{appendix:architectures}
We now describe the training procedure together with the architectural and hyperparameter choices adopted in the numerical experiments.
All models were trained by minimizing the \gls{mse} loss with an $L^1$ regularization term weighted by a parameter $\lambda$.
We set $\lambda=10^{-5}$ for the Navier--Stokes and advection benchmarks, whereas $\lambda=0$ for the Burgers test case in order to study the relationship between the Lipschitz constants and $N_h$ without directly affecting the former.
The latent dimensions and hidden-layer configurations of all architectures are summarized in Table~\ref{tab:network_parameters}.
Optimization was performed using Adam~\cite{prince_understanding_2023} with a learning rate of $10^{-3}$, followed by a number of \gls{lbfgs}~\cite{liu_limited_1989} epochs.
To prevent overfitting, the models evaluated in Section~\ref{sec:results} correspond to the minimum validation loss achieved during training.
The validation dataset consisted of the test parameters, with trajectories truncated to the portion of the temporal evolution used for training.

Regarding the network architectures, we used $\tanh$ activations for $\NN{dyn}$ and \gls{elu} for all decoder architectures.
We further incorporated skip connections~\cite{prince_understanding_2023} in all decoder architectures and applied layer normalization to meshless decoders only.
Whenever strong conditioning was employed, the conditioning network $\bm{\phi}$ consisted of a single-layer \gls{nn} using the same activation function as the corresponding main branch ($\NN{dyn}$ or decoder), as this consistently outperformed deeper alternatives in our experiments.

Hidden-layer configurations are reported in Table~\ref{tab:network_parameters} as pairs comprising the conditioned layers, followed by the non-conditioned layers.
The subscript denotes the conditioning type: $\mathrm{W}$ (weak) indicates weak conditioning, whereas $\mathrm{S}$ (strong) denotes strong conditioning.
For mesh-based architectures, all decoders consisted of two hidden layers with width $100$.

In accordance with Remark~\ref{rmk:signature}, decoder architectures may additionally depend on $\bm{\mu}(t)$ and $t$.
For the Burgers benchmark, the decoder received only $\nu$ as an additional input, whereas for the Navier--Stokes benchmark it received both $\mu$ and $t$.
In the advection benchmark, the decoder depended solely on the latent state.
For the Navier--Stokes test case, the spatial coordinates were rescaled to the unit square $[0,1]^2$.

To mitigate the spectral bias of our meshless models, we applied a Fourier feature mapping $\bm{\gamma}(\bm{x})$.
For an input coordinate $\bm{x} \in \mathbb{R}^d$, we defined $\bm{\gamma}: \mathbb{R}^d \to \mathbb{R}^{d(2F+1)}$ as
\begin{equation*}
\bm{\gamma}(\bm{x}) = \left[ \bm{x}, \sin(2^i \pi \bm{x}), \cos(2^i \pi \bm{x}) \right]_{i=0}^{F-1},
\end{equation*}
where the sine and cosine functions are applied elementwise.
The value of $F$ adopted for each benchmark is reported in Table~\ref{tab:network_parameters}.
For $F=0$, the mapping reduces to the identity, $\bm{\gamma}(\bm{x})=\bm{x}$, i.e., no Fourier features are employed: this is the case of \gls{ldnet}$_{(0)}$, which is otherwise identical to \gls{ldnet} in the Navier--Stokes benchmark.

For \gls{ldnet}, \gls{ldnet}$_{(0)}$, and \gls{filmldm} in the Navier--Stokes benchmark, we randomly selected $2000$ of the $5017$ mesh nodes at each training iteration to perform the forward pass and compute the loss.
This batching strategy substantially reduced the training time while ensuring good accuracy of the trained network.
However, more aggressive batching prevented the meshless decoders (particularly \gls{filmldm}) from overcoming spectral bias, leading them to learn only the mean field.

All experiments were performed on a workstation equipped with an NVIDIA RTX A2000 GPU (12 GB VRAM).

\begin{table}[htbp]
\caption{Summary of hyperparameters considered in the numerical experiments. The number of training epochs is expressed as the sum of Adam and \gls{lbfgs} iterations.}
\label{tab:network_parameters}
\centering
\begin{tabular}{c|c|c|c|c|c|c}
%\toprule
 & Architecture & $n$ & $\NN{dec}$ & $\NN{dyn}$ & Epochs & $F$ \\ 
\midrule
 & LD-FNN (Glorot) & $5$ & \ding{55} & $[64]_\mathrm{W},\; [64, 64]$ & $\num{1500}+\num{1000}$ & \ding{55} \\
 & LD-FNN (He) & $5$ & \ding{55} & $[64]_\mathrm{W},\; [64, 64]$ & $\num{1500}+\num{1000}$ & \ding{55} \\
\multirow{-3}{*}{Burgers} 
 & LDNet & $5$ & $[64]_\mathrm{W},\; [64, 64, 64]$ & $[64]_\mathrm{W},\; [64, 64]$ & $\num{1500}+\num{1000}$ & $7$ \\ 
\midrule
 & \gls{ldfnn} & 25 & \ding{55} & $[64]_\mathrm{W},\; [64, 64, 64]$ & $\num{2000}+\num{1000}$ & \ding{55} \\
 & \gls{ldgcn} & 25 & \ding{55} & $[64]_\mathrm{W},\; [64, 64, 64]$ & $\num{2000}+\num{1000}$ & \ding{55} \\
 & \gls{ldnet} & 25 & $[85]_\mathrm{W},\; [85, 85, 85]$ & $[64]_\mathrm{W},\; [64, 64, 64]$ & $\num{2000}+\num{1000}$ & $7$ \\
 & \gls{ldnet}$_{(0)}$ & 25 & $[85]_\mathrm{W},\; [85, 85, 85]$ & $[64]_\mathrm{W},\; [64, 64, 64]$ & $\num{2000}+\num{1000}$ & $0$ \\
\multirow{-5}{*}{Navier--Stokes} 
 & \gls{filmldm} & 25 & $[85, 85]_\mathrm{S},\; [85, 85]$ & $[64]_\mathrm{W},\; [64, 64, 64]$ & $\num{2000}+\num{1000}$ & $7$ \\ 
\midrule
 & \gls{ldfnn} & 2 & \ding{55} & $[64]_\mathrm{W},\; [64, 64]$ & $\num{2000}+\num{1000}$ & \ding{55} \\
 & $\mathrm{W}_1$ (\gls{ldnet}) & 2 & $[64]_\mathrm{W},\; [64, 64, 64]$ & $[64]_\mathrm{W},\; [64, 64]$ & $\num{2000}+\num{1000}$ & $6$ \\
 & $\mathrm{W}_2$ & 2 & $[64, 64]_\mathrm{W},\; [64, 64]$ & $[64]_\mathrm{W},\; [64, 64]$ & $\num{2000}+\num{1000}$ & $6$ \\
 & $\mathrm{S}_1$ & 2 & $[64]_\mathrm{S},\; [64, 64, 64]$ & $[64]_\mathrm{W},\; [64, 64]$ & $\num{2000}+\num{1000}$ & $6$ \\
\multirow{-5}{*}{Advection}
 & $\mathrm{S}_2$ & 2 & $[64, 64]_\mathrm{S},\; [64, 64]$ & $[64]_\mathrm{W},\; [64, 64]$ & $\num{2000}+\num{1000}$ & $6$ \\
\bottomrule
\end{tabular}
\end{table}

\section{Computation of continuous quantities}\label{appendix:continuous}
Since the bound in Equation~\eqref{eq:latent_magnitude_2} involves continuous $L^2(\Omega)$ norms rather than their discrete counterparts, we approximate spatial integral norms for the Burgers benchmark using the trapezoidal rule~\cite{quarteroni_numerical_2007}.
The final error metric $\varepsilon_{L^2(\Omega)}^\mathrm{test}$ is defined as the mean $L^2(\Omega)$ error over all values of $\nu$ in the test dataset.

Furthermore, evaluating the theoretical bounds in Section~\ref{sec:analysis} requires computing the Lipschitz constant $L_D$ of each decoder network $D\in\{\NN{dec}, \tilde{\mathcal{D}}_{\bm \theta}^{N_h}\}$.
In practice, the complexity of the evaluated architectures (particularly the inclusion of skip connections and layer normalization) makes the derivation of analytical bounds more involved.
Even when derivable, such bounds tend to be excessively loose, complicating the direct empirical validation of our theoretical results.
Therefore, for a given decoder $D\in\{\NN{dec}, \tilde{\mathcal{D}}_{\bm \theta}^{N_h}\}$, we estimate the corresponding Lipschitz constant $L_D$ via a Monte Carlo-based procedure, providing a practical proxy for the true Lipschitz constant.

To this end, we sample a large set of latent states $\bm s$ from the latent trajectories.
For each of them, we generate random perturbation directions uniformly distributed on the unit hypersphere, and scale them such that $\|\Delta \bm s\|_2 \in [10^{-4}, 10^{2}]$ to capture both local and global variations in the latent landscape.

The empirical Lipschitz constant is then estimated as the maximum observed secant ratio between perturbations in latent space and the corresponding variations in the physical domain, i.e.,
\begin{equation*}
    L_D \approx \max_{\bm s, \Delta \bm{s}} \frac{\|D(\bm s + \Delta \bm s) - D(\bm s)\|_{L^2(\Omega)}}{\|\Delta \bm s\|_2},
\end{equation*}
where the $L^2(\Omega)$ norm in the numerator is approximated using the trapezoidal rule.
By taking the maximum over the randomly sampled latent states and perturbation scales, we obtain an empirical proxy for $L_D$.

\section{Sensitivity analysis in the advection benchmark}\label{appendix:sensitivity}
We detail here how the quantities $\mathcal{A}$ and $\rho$ of Subsection~\ref{subsec:advection} are computed.

To isolate the sensitivity of the decoder, the latent dynamics of each test trajectory is first integrated without perturbations; the decoder is then evaluated for each value of $t$ and $\bm\mu$ at the perturbed latent state $\bm s(t;\bm\mu)+\deltas$, where $\deltas=\eta\bar s\,\bm v$, $\bar s$ is the mean of $\|\bm s(t;\bm\mu)\|_2$ along the corresponding trajectory, and $\bm v$ is drawn uniformly on the unit sphere of $\R^n$.
Scaling the perturbation by $\bar s$ makes $\mathcal{A}$ invariant to a rescaling of the latent space, whose magnitude is learned and therefore differs across models.
For each latent state, we draw ten realizations of $\bm v$, shared by all models and by all values of $\eta\in\{10^{-3},10^{-2},10^{-1},1\}$, and we average $\mathcal{A}$ over the realizations, the test trajectories, the snapshots of the considered time interval, and the seeds.

To compute $\rho$, we use the same perturbations with $\eta=10^{-2}$ and measure how much the response of the first layer varies across the domain.
Specifically, denoting by $\bm h^{(1)}(\bm x\,|\,\bm s)$ the output of the first conditioned layer (before the skip connection, the normalization, and the activation, so that $\rho=1$ exactly for weak conditioning), we quantify the response at each mesh node $\bm x_i$ as
\begin{equation*}
    \zeta_i=\frac{\|\bm h^{(1)}(\bm x_i\,|\,\bm s+\deltas)-\bm h^{(1)}(\bm x_i\,|\,\bm s)\|_2}{\|\deltas\|_2},
\end{equation*}
and we set $\rho=\max_i\zeta_i/\min_i\zeta_i$, computed for each realization, snapshot, trajectory and seed, and averaged in the same way as $\mathcal{A}$.

\end{appendix}